\documentclass[a4paper,11pt]{elsarticle}
\usepackage[T1]{fontenc}
\usepackage[utf8]{inputenc}
\usepackage{lmodern,amsmath,amssymb,amsthm,bm}
\usepackage{graphicx,multirow,booktabs,array,microtype,xcolor}
\usepackage[a4paper,margin=23mm]{geometry}
\usepackage[colorlinks=true,linkcolor=blue,citecolor=blue,urlcolor=blue,
pdftitle={Reynolds-Semi-Robust Error Estimates for a Globally Divergence-Free HDG Method for the Smagorinsky Model},
pdfauthor={Shuaijun Liu and Xiaoping Xie}]{hyperref}
\usepackage[capitalise,noabbrev]{cleveref}
\allowdisplaybreaks[2]
\newcommand{\w}{\bm w}
\renewcommand{\u}{\bm u}
\renewcommand{\v}{\bm v}
\newcommand{\x}{\bm x}
\newcommand{\V}{\bm V}

\newcommand{\et}{\bm\eta}
\newcommand{\Th}{\mathcal T_h}
\newcommand{\Fh}{\mathcal F_h}
\newcommand{\n}{\bm n}
\newcommand{\grad}{\nabla_h}

\newcommand{\jn}[1]{j(#1)}
\newcommand{\nnorm}[1]{\left|\!\left|\!\left|#1\right|\!\right|\!\right|}

\newcommand{\Amap}{\mathcal A}
\newcommand{\dt}{d_\tau}
\newcommand{\mus}{\mu_s}
\newcommand{\divop}{\operatorname{div}}
\newcommand{\R}{\mathbb R}

\newtheorem{mytheorem}{Theorem}[section]
\newtheorem{mylemma}[mytheorem]{Lemma}
\newtheorem{myprop}[mytheorem]{Proposition}
\newtheorem{mycorollary}[mytheorem]{Corollary}
\newtheorem{myassumption}[mytheorem]{Assumption}
\theoremstyle{remark}
\newtheorem{myremark}[mytheorem]{Remark}
\numberwithin{equation}{section}

\newif\ifauthorchecks
\authorcheckstrue

\newcommand{\PaperGraphic}[3][]{%
  \IfFileExists{#3}{\includegraphics[#1]{#3}}{%
  \fbox{\parbox[c][0.57\dimexpr#2\relax][c]{\dimexpr#2-2\fboxsep-2\fboxrule\relax}{%
  \centering\scriptsize\textbf{Figure not supplied}\par\smallskip
  \texttt{\detokenize{#3}}}}}}

\begin{document}
\begin{frontmatter}
\title{A Reynolds-Semi-Robust, Globally Divergence-Free HDG Method for the Smagorinsky Model\tnoteref{t1}}
\tnotetext[t1]{This work is supported by the National Natural Science Foundation of China (Grant No.\ 12571434) and the Open Research Project of the National Key Laboratory of Fundamental Algorithms and Models for Engineering Simulation.}
\author{Shuaijun Liu}
\ead{sj\_liu123@163.com}
\author{Xiaoping Xie\corref{cor1}}
\ead{xpxie@scu.edu.cn}
\cortext[cor1]{Corresponding author.}
\address{School of Mathematics, Sichuan University, Chengdu 610064, China\\
National Key Laboratory of Fundamental Algorithms and Models for Engineering Simulation, Sichuan University, Chengdu 610207, China}

\begin{abstract}
We develop and analyze a fully discrete, globally divergence-free hybridizable discontinuous Galerkin (HDG) method for a gradient-based Smagorinsky model.  The method combines backward Euler time stepping, interior-penalty discretizations of molecular and nonlinear eddy diffusion, and an upwind convective flux.  The discrete velocity is $H(\operatorname{div})$-conforming and pointwise divergence-free, which yields pressure robustness.  For sufficiently large penalty parameters, we prove unconditional energy stability and existence of a discrete solution, and establish uniqueness under additional smallness conditions. 
A velocity error estimate is derived without explicit inverse powers of the molecular viscosity. 
The nonlinear facet residuals are controlled using local trace-approximation estimates and a viscosity-independent facet penalty. 
We retain the dependence of the discrete Gronwall factor on the filter scale and the mesh size; a mesh-uniform bound follows under suitable solution regularity, a fixed time-step margin, and the scaling $\delta=O(h)$ on quasi-uniform meshes. 
The reported manufactured-solution results are consistent with the resulting pre-asymptotic error bounds. Further flow examples illustrate the dissipative behavior of the method and are distinguished from the boundary conditions and parameter range covered by the analysis.

\end{abstract}
\begin{keyword}
Smagorinsky model \sep hybridizable discontinuous Galerkin method \sep globally divergence-free discretization \sep pressure robustness \sep Reynolds-semi-robust error estimates
\end{keyword}
\end{frontmatter}

\section{Introduction}
Let $\Omega\subset\R^d$, $d\in\{2,3\}$, be a bounded, connected polygonal or polyhedral domain with Lipschitz boundary $\Gamma=\partial\Omega$, and let $T>0$. We consider the incompressible evolution problem
\begin{equation}\label{Smagorinsky_eq}
\left\{\begin{aligned}
\partial_t\w+(\w\cdot\nabla)\w-\nu\Delta\w+\nabla r
 -\nabla\cdot\bigl(\mus|\nabla\w|\nabla\w\bigr)&=\bm f
 &&\text{in }\Omega\times(0,T],\\
\nabla\cdot\w&=0&&\text{in }\Omega\times(0,T],\\
\w&=\bm0&&\text{on }\Gamma\times(0,T],\\
\w(\cdot,0)&=\w_0&&\text{in }\Omega,
\end{aligned}\right.
\end{equation}
where $\w$ and $r$ denote velocity and pressure, $\nu>0$ is the molecular kinematic viscosity, and $\bm f$ is the body force. Throughout the analysis,
\begin{equation}\label{nuT}
\mus:=(C_s\delta)^2,\qquad \nu_T(\w)=\mus|\nabla\w|,
\end{equation}
where $C_s\in[0.1,0.2]$ is the Smagorinsky constant, and $\delta$ denotes a characteristic filter or grid scale. The norm of a matrix is its Frobenius norm. The choice $\mus=0$ recovers the incompressible Navier--Stokes equations. The Reynolds number is $\mathrm{Re}=UL/\nu$, where $U$ and $L$ are characteristic velocity and length scales.

Equation~\eqref{Smagorinsky_eq} uses the full velocity gradient. It is a gradient-based Smagorinsky/Ladyzhenskaya-type model \cite{Du,John,Smagorinsky1963,Smagorinsky1993}. It should be distinguished from the strain-based closure formulated with $D(\w)=(\nabla\w+\nabla\w^{\mathsf T})/2$. These two nonlinear operators are not identical. The discrete formulation and the analysis below apply to the model as written in~\eqref{Smagorinsky_eq}; an extension to a strain-based closure would require a separate treatment of its stress and discrete Korn inequalities.

Accurate simulation of incompressible flows at high Reynolds numbers remains a central challenge in computational fluid dynamics. Resolving all dynamically relevant scales by direct numerical simulation can be prohibitively expensive, whereas large-eddy simulation (LES) resolves larger flow structures and models the effect of unresolved scales \cite{Guermond2004,Lesieur2005,Moin1998,Pope2000,Rebollo2014,Sagaut2006}. Smagorinsky-type closures are widely studied because of their comparatively simple, nonlinear dissipative structure \cite{Lilly,Smagorinsky1963}. Numerous modifications have been developed to address their limitations in different flow regimes \cite{Deardorff1970,Germano1991,Lilly1992,chorfi2020analysis,Huang2024}.

Mathematical analyses of related nonlinear-viscosity models include the work of Du and Gunzburger \cite{Du} and Par\'es \cite{Pares1992}. John and Layton \cite{John} analyzed numerical errors in LES. Burman, Hansbo, and Larson \cite{Burman} studied stability under scale separation and derived pre-asymptotic error estimates for stabilized divergence-free approximations. A posteriori analysis of an implicit-Euler finite element discretization is available in \cite{Nassreddine2023}. Stationary approximations and nonlinear solvers have also been investigated; see, for example, \cite{du1990finite,Borggaard2008,Su2014,Shi2019,Yang2023,Zhang2025}.

Two distinct structural requirements motivate the present method. First, convection-dominated flows require appropriate stabilization. Second, irrotational forces should be balanced by pressure without polluting the velocity approximation. The latter property, referred to as pressure robustness, is obtained by testing the discrete momentum equation in a genuinely divergence-free subspace \cite{JLMNR2017,LM2016}. For discontinuous velocities, elementwise vanishing divergence alone is insufficient: normal continuity and the appropriate boundary condition are also needed. Globally divergence-free HDG and related methods provide this structure \cite{CFX2016,CS2014,LS2016,Rhebergen2018-3,Chen2024}. Reynolds-semi-robust estimates for incompressible flow discretizations have been developed in several settings \cite{Han2021,Han2022,Han2023,Beirao2025,Quiroz2025}.

The HDG framework combines discontinuous element fields with facet unknowns and permits local elimination of suitable element variables \cite{Cockburn2009}. For incompressible flow formulations, representative contributions include \cite{CNP2010,NPC2010,NPC2011,QS2016,FJQ2019}. Here, element and facet velocities have degree $k$, while element and facet pressures have degrees $k-1$ and $k$, respectively, with $k\ge1$. The pressure coupling enforces a globally divergence-free element velocity. Molecular diffusion is discretized by a symmetric interior-penalty form, convection by an upwind flux, and nonlinear eddy diffusion by a nonlinear interior-penalty form. An additional, viscosity-independent quadratic facet penalty is retained explicitly in the analysis.

The main analytical difficulty is the difference of nonlinear facet fluxes. Volume monotonicity of $G\mapsto |G|G$ does not imply monotonicity of the complete nonlinear HDG operator. We therefore keep the monotone volume and jump terms on the left-hand side and estimate the remaining facet terms separately. In particular, traces of interpolation errors are bounded by trace-approximation estimates, not by polynomial inverse inequalities.

For quasi-uniform meshes, sufficiently regular solutions, and a time-step restriction specified in Section~\ref{Sec:err}, the resulting nodal velocity estimate has the form
\begin{equation}\label{intro:rate}
\max_{0\le n\le N}\|\w(t_n)-\w_h^n\|_{L^2(\Omega)}
\le C\mathcal G_N^{1/2}\Bigl(
\tau+\nu^{1/2}h^k+h^{k+1/2}
+C_s\delta\bigl(h^{3k/2}+h^{3k/4}\bigr)\Bigr).
\end{equation}
Here $\mathcal G_N$ is an explicit discrete Gronwall factor. Neither its exponent nor the error prefactor involves an explicit inverse power of $\nu$. They do depend on norms of the exact velocity, which need not be uniform as $\nu$ or $\delta$ varies. 
Under $\delta=O(h)$, uniformly bounded solution norms, and a fixed margin in the time-step restriction, $\mathcal G_N$ is uniformly bounded. In the convection-dominated regime $\nu\lesssim h$, this gives the pre-asymptotic bound
\[
O\!\left(\tau+h^{k+1/2}+h^{3k/2+1}+h^{3k/4+1}\right).
\]
The estimate therefore characterizes the pre-asymptotic behavior of the method in high-Reynolds-number regimes, while the exact discrete incompressibility ensures pressure robustness.

Section~\ref{Sec:Pre} introduces the continuous formulation and a modelling-error estimate. Section~\ref{Sec:num} defines the HDG method and proves its structural and stability properties. Section~\ref{Sec:err} develops the velocity error analysis. Numerical results and their scope are discussed in Section~\ref{Sec:exp}.

\section{Notation and continuous formulation}\label{Sec:Pre}
We use the standard Sobolev spaces $W^{\ell,p}(D)$, with norms $\|\cdot\|_{W^{\ell,p}(D)}$ and seminorms $|\cdot|_{W^{\ell,p}(D)}$, and write $H^\ell(D)=W^{\ell,2}(D)$. Vector and matrix norms are understood componentwise. The symbols $(\cdot,\cdot)_D$ and $\langle\cdot,\cdot\rangle_{\partial D}$ denote volume and boundary pairings; matrix products in these pairings use the Frobenius inner product. We abbreviate $(\cdot,\cdot)_\Omega$ by $(\cdot,\cdot)$ and use standard Bochner-space notation.

The constant $C$ may change from line to line. Unless stated otherwise, it may depend on the domain, the polynomial degree, mesh shape regularity, fixed penalty parameters, and displayed solution norms, but not explicitly on $h$, $\tau$, $\delta$, or inverse powers of $\nu$. Dependence on the filter scale and mesh size through a Gronwall factor is always displayed. An assertion of parameter-uniform convergence additionally requires parameter-uniform bounds on the relevant solution norms.

Set
\[
\V=[W^{1,3}_0(\Omega)]^d,\qquad
\V_{\divop0}=\{\v\in\V:\nabla\cdot\v=0\},\qquad Q=L^2_0(\Omega).
\]
For a sufficiently regular solution, the weak formulation of~\eqref{Smagorinsky_eq} is
\begin{equation}\label{weak_eq}
\begin{aligned}
(\partial_t\w,\v)+\nu(\nabla\w,\nabla\v)
+((\w\cdot\nabla)\w,\v)
+\mus(|\nabla\w|\nabla\w,\nabla\v)
-(r,\nabla\cdot\v)&=(\bm f,\v),\\
(q,\nabla\cdot\w)&=0
\end{aligned}
\end{equation}
for all $(\v,q)\in\V\times Q$, with $\w(0)=\w_0$. The natural pressure integrability for a general $W^{1,3}$ weak formulation is $L^{3/2}$; the choice $Q=L^2_0$ here specifies the more regular pressure class used in the subsequent consistency analysis. We do not infer the regularity required below from minimal weak-solution data. For continuous existence and regularity theory, see \cite{Du,Pares1992,John2016}.

Testing~\eqref{weak_eq} with $(\v,q)=(\w,r)$ gives, whenever the solution is regular enough to justify the identity,
\begin{equation}\label{energy:continuous}
\frac12\frac{d}{dt}\|\w\|_{L^2(\Omega)}^2
+\nu\|\nabla\w\|_{L^2(\Omega)}^2
+\mus\|\nabla\w\|_{L^3(\Omega)}^3=(\bm f,\w).
\end{equation}
The convective term vanishes because $\nabla\cdot\w=0$ and the velocity has zero boundary trace. The last term is nonnegative model dissipation; this particular closure does not describe backscatter. For weak solutions, the corresponding energy inequality is the appropriate statement unless additional regularity establishes equality.

\begin{mylemma}[Monotonicity and continuity]\label{Lemm_1}
For vectors or matrices $G,H$, define $\Amap(G)=|G|G$. Then
\begin{align}
(\Amap(G)-\Amap(H)):(G-H)&\ge\tfrac14|G-H|^3,\label{Mono}\\
|\Amap(G)-\Amap(H)|&\le(|G|+|H|)|G-H|.\label{Conti}
\end{align}
For vectors, the colon in~\eqref{Mono} means the Euclidean inner product. Consequently, for $\u,\w,\v\in[W^{1,3}(\Omega)]^d$,
\[
\bigl|(\Amap(\nabla\u)-\Amap(\nabla\w),\nabla\v)\bigr|
\le(\|\nabla\u\|_{L^3}+\|\nabla\w\|_{L^3})
\|\nabla(\u-\w)\|_{L^3}\|\nabla\v\|_{L^3}.
\]
\end{mylemma}
\begin{proof}
Set $a=|G|$ and $b=|H|$. A direct expansion gives
\[
(\Amap(G)-\Amap(H)):(G-H)
=\tfrac{a+b}{2}\left(|G-H|^2+(a-b)^2\right)
\ge\tfrac12|G-H|^3.
\]
The last step uses $a+b\ge|G-H|$, and implies the conservative bound~\eqref{Mono}. Also,
$\Amap(G)-\Amap(H)=a(G-H)+(a-b)H$, so the reverse triangle inequality proves~\eqref{Conti}. The integrated estimate follows by H\"older's inequality with exponents $(3,3,3)$. These are the standard cubic-growth inequalities used, for example, in~\cite{Du,John}; the weaker constant $1/4$ is retained uniformly below.
\end{proof}

\subsection{Modelling error relative to Navier--Stokes}
Let $\u$ solve the Navier--Stokes problem obtained by setting $\mus=0$, with the same forcing, initial value, and boundary data as $\w$. The following comparison is conditional on the stated regularity of $\u$; it is not a global smooth-solution assertion for three-dimensional Navier--Stokes flow.

\begin{mylemma}\label{lemma::pertur_err}
Let $\et=\u-\w$, with $\et(0)=0$, and suppose the two solutions are sufficiently regular to justify testing their difference with $\et$. Set
\[
R_u=2\int_0^T\|\nabla\u(t)\|_{L^\infty(\Omega)}\,dt<\infty.
\]
If $\nabla\u\in L^3(0,T;L^3(\Omega))$, then
\begin{equation}\label{model_er1}
\begin{aligned}
&\|\et\|_{L^\infty(0,T;L^2)}^2
+\int_0^T\left(2\nu\|\nabla\et\|_{L^2}^2+	\frac{\mus}{4}\|\nabla\et\|_{L^3}^3\right)dt\\
&\qquad\le \frac{16\sqrt2}{3\sqrt3}\,\mus e^{R_u}
\int_0^T\|\nabla\u\|_{L^3}^3\,dt.
\end{aligned}
\end{equation}
If, in addition, $\nabla\cdot\Amap(\nabla\u)\in L^2(0,T;[L^2(\Omega)]^d)$, then
\begin{equation}\label{model_er2}
\begin{aligned}
&\|\et\|_{L^\infty(0,T;L^2)}^2
+\int_0^T\left(2\nu\|\nabla\et\|_{L^2}^2+	\frac{\mus}{2}\|\nabla\et\|_{L^3}^3\right)dt\\
&\qquad\le 2\mus^2 e^{R_u+T}
\int_0^T\|\nabla\cdot\Amap(\nabla\u)\|_{L^2}^2\,dt.
\end{aligned}
\end{equation}
\end{mylemma}
\begin{proof}
Subtract the two momentum equations, add and subtract $\mus\Amap(\nabla\u)$, and test with $\et$. The pressure contribution vanishes. Since both velocities are divergence-free and have zero boundary trace,
\[
((\w\cdot\nabla)\w-(\u\cdot\nabla)\u,\et)
=-((\et\cdot\nabla)\u,\et)
\le\|\nabla\u\|_{L^\infty}\|\et\|_{L^2}^2.
\]
Thus, by~\eqref{Mono},
\begin{equation}\label{model_eq}
\tfrac12\tfrac{d}{dt}\|\et\|_{L^2}^2+\nu\|\nabla\et\|_{L^2}^2
+\tfrac{\mus}{4}\|\nabla\et\|_{L^3}^3
\le \mus(\Amap(\nabla\u),\nabla\et)
+\|\nabla\u\|_{L^\infty}\|\et\|_{L^2}^2.
\end{equation}
Young's inequality, with $\varepsilon^3=3/8$, gives
\[
\mus\|\nabla\u\|_{L^3}^2\|\nabla\et\|_{L^3}
\le \tfrac{\epsilon^3}{3}\mus \|\nabla \et\|_{L^3(\Omega)}^3 
+ \tfrac{2}{3\epsilon^{3/2}}\mus\|\nabla \u\|_{L^3(\Omega)}^3
\le\tfrac{\mus}{8}\|\nabla\et\|_{L^3}^3
+\frac{4\sqrt2}{3\sqrt3}\mus\|\nabla\u\|_{L^3}^3.
\]
Multiplying~\eqref{model_eq} by two and applying the integrating-factor form of Gronwall's inequality bounds, at every $t$, the squared error plus the accumulated dissipation up to $t$ by
\[
\frac{8\sqrt2}{3\sqrt3}\mus e^{R_u}\int_0^T\|\nabla\u\|_{L^3}^3dt.
\]
Taking the supremum of the error and retaining the full dissipation integral separately costs at most a factor of two, yielding~\eqref{model_er1}.

For the stronger estimate, integration by parts and the zero trace of $\et$ give
\[
\mus(\Amap(\nabla\u),\nabla\et)
=-\mus(\nabla\cdot\Amap(\nabla\u),\et)
\le\tfrac{\mus^2}{2}\|\nabla\cdot\Amap(\nabla\u)\|_{L^2}^2
+\tfrac12\|\et\|_{L^2}^2.
\]
The same argument, now with growth coefficient $2\|\nabla\u\|_{L^\infty}+1$, proves~\eqref{model_er2}.
\end{proof}

\begin{myremark}
For fixed $C_s$, the modeling error is of order $O(\delta)$ in $L^\infty(0,T;L^2)$ under~\eqref{model_er1}, and improves to $O(\delta^2)$ under the stronger assumption~\eqref{model_er2}. These bounds involve the full gradient of the Navier--Stokes solution through the growth factor and do not invoke the scale-separated stability mechanism analyzed in~\cite{Burman}. The modeling error and the discretization error correspond to different comparisons of solutions and should therefore be distinguished.
\end{myremark}

\section{HDG discretization and structural properties}\label{Sec:num}
\subsection{Mesh, spaces, and interpolation}
Let $\{\Th\}$ be a shape-regular family of conforming simplicial meshes of $\Omega$. For $T\in\Th$, let $h_T=\operatorname{diam}(T)$ and let $\n$ be the outward unit normal on $\partial T$. Set $h=\max_T h_T\le1$. The set of all facets is denoted by $\Fh$, with interior and boundary subsets $\mathcal F_I$ and $\mathcal F_B$. We use
\[
(z,v)_{\Th}=\sum_{T\in\Th}(z,v)_T,\qquad
\langle z,v\rangle_{\partial\Th}=\sum_{T\in\Th}\langle z,v\rangle_{\partial T}.
\]
Consequently, an interior facet is counted twice in a pairing over $\partial\Th$, with the outward normal of the corresponding element. All factors involving $h_T$ remain inside element sums. Quasi-uniformity, namely $c_{\rm qu}h\le h_T\le h$ for a fixed $c_{\rm qu}>0$, will be imposed only for the simplified global-$h$ estimates.

For $k\ge1$, define
\begin{align*}
\V_h&=\{\v_h\in[L^2(\Omega)]^d:\v_h|_T\in[P_k(T)]^d\ \forall T\in\Th\},\\
Q_h&=\{q_h\in L^2_0(\Omega):q_h|_T\in P_{k-1}(T)\ \forall T\in\Th\},\\
\widetilde\V_h&=\{\widetilde\v_h\in[L^2(\Fh)]^d:
\widetilde\v_h|_F\in[P_k(F)]^d\ \forall F\in\Fh,
\ \widetilde\v_h|_{\mathcal F_B}=\bm0\},\\
\widetilde Q_h&=\{\widetilde q_h\in L^2(\Fh):
\widetilde q_h|_F\in P_k(F)\ \forall F\in\Fh\}.
\end{align*}
Facet unknowns are single-valued on each facet. The zero-mean constraint is imposed only on the element pressure. Write
\[
\V_h^\star=\V_h\times\widetilde\V_h,\qquad Q_h^\star=Q_h\times\widetilde Q_h,
\qquad \v_h^\star=(\v_h,\widetilde\v_h),\quad q_h^\star=(q_h,\widetilde q_h).
\]
For a pair $\v^\star=(\v,\widetilde\v)$, define its element--facet difference by
\begin{equation}\label{def:jump}
\jn{\v^\star}|_{\partial T}=\v|_{\partial T}-\widetilde\v|_{\partial T}.
\end{equation}
This is not the jump between the two element traces of an interior facet. The broken gradient is denoted by $\grad$; on a single element it is simply $\nabla$.

We use the Raviart--Thomas space $\mathrm{RT}_k(T)=[P_k(T)]^d+\x P_k(T)$ and its canonical interpolation operator $\Pi_{RT}$, defined by
\begin{equation}\label{eq:RTdef}
\begin{aligned}
(\Pi_{RT}\v-\v,\bm z)_T&=0&&\forall\bm z\in[P_{k-1}(T)]^d,\\
\langle(\Pi_{RT}\v-\v)\cdot\n,q\rangle_F&=0
&&\forall q\in P_k(F),\quad F\subset\partial T.
\end{aligned}
\end{equation}
The assembled interpolant is $H(\divop,\Omega)$-conforming and satisfies
\begin{equation}\label{eq:RTcommute}
\nabla\cdot\Pi_{RT}\v=\Pi_{k}(\nabla\cdot\v),
\end{equation}
where $\Pi_k$ is the elementwise $L^2$ projection onto $P_k$ \cite{Boffi,Ern}. In particular,
\begin{equation}\label{eq:RT4}
\nabla\cdot\v=0\quad\Longrightarrow\quad
\nabla\cdot\Pi_{RT}\v=0,\qquad
(\Pi_{RT}\v)|_T\in[P_k(T)]^d.
\end{equation}
Indeed, the only possible degree-$(k+1)$ component of an $\mathrm{RT}_k$ field is $\x q_k$, where $q_k$ is homogeneous of degree $k$. Its divergence contains the term $(d+k)q_k$, which must vanish if the RT field is divergence-free. Since $d+k>0$, this implies $q_k=0$. Hence, the degree-$(k+1)$ component vanishes, and the divergence-free RT field belongs to the element velocity space used here. Homogeneous normal boundary conditions are also preserved.

Let $\Pi_F$ denote the facetwise $L^2$ projection onto $[P_k(F)]^d$, and let $\gamma_h$ denote restriction of a sufficiently regular function to the entire mesh skeleton. For a divergence-free velocity with zero boundary trace, set
\begin{equation}\label{def:Pi}
\Pi\w^\star=(\Pi_{RT}\w,\Pi_F\gamma_h\w)\in\V_h^\star.
\end{equation}

\begin{mylemma}[Scaling and approximation estimates]\label{RT-lem}
For fixed polynomial degree and shape-regular simplices, the following estimates hold with mesh-independent constants. For a polynomial $z_h$,
\begin{equation}\label{inverse-lem}
\|\nabla z_h\|_{L^p(T)}\lesssim h_T^{-1}\|z_h\|_{L^p(T)},\qquad
\|z_h\|_{L^p(\partial T)}\lesssim h_T^{-1/p}\|z_h\|_{L^p(T)},
\quad 1\le p\le\infty.
\end{equation}
If $1\le\ell\le k+1$, $\v\in[W^{\ell,p}(T)]^d$, and $0\le j\le\ell$, then
\begin{align}
\|D^j(\v-\Pi_{RT}\v)\|_{L^p(T)}
&\lesssim h_T^{\ell-j}|\v|_{W^{\ell,p}(T)},\label{eq:RT1}\\
\|D^j(\v-\Pi_{RT}\v)\|_{L^p(\partial T)}
&\lesssim h_T^{\ell-j-1/p}|\v|_{W^{\ell,p}(T)},
\qquad 0\le j\le\ell-1.\label{eq:RT2}
\end{align}
For $\v\in[W^{1,\infty}(T)]^d$,
\begin{equation}\label{eq:RTbound3}
\|\nabla\Pi_{RT}\v\|_{L^\infty(T)}\lesssim\|\nabla\v\|_{L^\infty(T)},\qquad
\|\v-\Pi_{RT}\v\|_{L^\infty(T\cup\partial T)}\lesssim h_T\|\nabla\v\|_{L^\infty(T)}.
\end{equation}
The facet projection satisfies the corresponding $L^p$ trace-approximation estimates. In particular,
\begin{equation}\label{eq:facetapprox}
\|\gamma_h\v-\Pi_F\gamma_h\v\|_{L^p(\partial T)}
\lesssim h_T^{\ell-1/p}|\v|_{W^{\ell,p}(T)},\qquad 1\le\ell\le k+1.
\end{equation}
\end{mylemma}
\begin{proof}
The inverse bounds follow by scaling and equivalence of polynomial norms on a reference simplex. The interpolation estimates follow from reference-element boundedness, polynomial reproduction, and the trace inequality; see \cite{Boffi,Ciarlet2002,Ern}. For~\eqref{eq:RTbound3}, subtract an elementwise constant, use reproduction of constants, and apply reference-element $W^{1,\infty}$ stability. The facet projection is bounded in $L^p$ for fixed degree on shape-regular facets, so comparison with a local polynomial approximant proves~\eqref{eq:facetapprox}. The restriction $j\le\ell-1$ in~\eqref{eq:RT2} ensures that the indicated derivative has a controlled trace.
\end{proof}


\subsection{Forms and fully discrete scheme}
Let $N\ge1$, $\tau=T/N$, and $t_n=n\tau$. For a continuous function, $\w^n=\w(t_n)$ denotes its exact nodal value, whereas $\w_h^n$ is a discrete unknown. Set $\dt z^n=(z^n-z^{n-1})/\tau$.
We prescribe nodal data $\bm f^n\in[L^2(\Omega)]^d$. In the consistency and error analysis, $\bm f^n=\bm f(t_n)$ with well-defined point values is assumed. For merely $L^2$-in-time forcing, interval averages can instead be used, but their quadrature residual must then be included in the consistency equation.

For pairs of element and facet functions, define
\begin{align}
a_h(\u^\star,\v^\star)
={}&\nu\sum_{T\in\Th}\Bigl[
(\nabla\u,\nabla\v)_T
-\langle\nabla\u\n,\jn{\v^\star}\rangle_{\partial T}
-\langle\jn{\u^\star},\nabla\v\n\rangle_{\partial T}
+\tfrac{\alpha_1}{h_T}\langle\jn{\u^\star},\jn{\v^\star}\rangle_{\partial T}\Bigr],\label{def:ah}\\
s_h(\bm z^\star,\u^\star,\v^\star)
={}&\mus\sum_{T\in\Th}\Bigl[
(|\nabla\bm z|\nabla\u,\nabla\v)_T
-\langle|\nabla\bm z|\nabla\u\n,\jn{\v^\star}\rangle_{\partial T}\notag\\
&\hspace{16mm}-\langle|\nabla\bm z|\jn{\u^\star},\nabla\v\n\rangle_{\partial T}
+\tfrac{\alpha_2}{h_T^2}\langle|\jn{\bm z^\star}|\jn{\u^\star},\jn{\v^\star}\rangle_{\partial T}\Bigr],\label{def:sh}\\
c_h(\bm b,\u^\star,\v^\star)
={}&-(\u\otimes\bm b,\grad\v)_{\Th}
+\tfrac12\langle(\bm b\cdot\n)(\u+\widetilde\u),\jn{\v^\star}\rangle_{\partial\Th}\notag\\
&+\tfrac12\langle|\bm b\cdot\n|\jn{\u^\star},\jn{\v^\star}\rangle_{\partial\Th},\label{def:ch}\\
b_h(\v^\star,q^\star)
={}&(\grad\cdot\v,q)_{\Th}
-\langle\jn{\v^\star}\cdot\n,\widetilde q\rangle_{\partial\Th},\label{def:bh}\\
J_h(\u^\star,\v^\star)
={}&\langle\jn{\u^\star},\jn{\v^\star}\rangle_{\partial\Th}.\label{def:Jh}
\end{align}
Here $\alpha_1,\alpha_2,\alpha_3$ are positive penalty parameters, fixed with respect to $h$, $\tau$, $\nu$, and $\delta$. We also write $S_h(\u^\star;\v^\star)=s_h(\u^\star,\u^\star,\v^\star)$. The forms $s_h$ and $c_h$ are not trilinear: their first arguments enter through absolute values. For a fixed first argument, they are bilinear in their last two arguments.

Given $\w_h^0=\Pi_{RT}\w_0$, find
$(\w_h^{n,\star},r_h^{n,\star})\in\V_h^\star\times Q_h^\star$ for $n=1,\ldots,N$ such that
\begin{equation}\label{eq:fulld}
\begin{aligned}
(\dt\w_h^n,\v_h)_{\Th}+a_h(\w_h^{n,\star},\v_h^\star)
+S_h(\w_h^{n,\star};\v_h^\star)
+c_h(\w_h^n,\w_h^{n,\star},\v_h^\star)&\\
+\alpha_3J_h(\w_h^{n,\star},\v_h^\star)
-b_h(\v_h^\star,r_h^{n,\star})&=(\bm f^n,\v_h)_{\Th},\\
b_h(\w_h^{n,\star},q_h^\star)&=0
\end{aligned}
\end{equation}
for every $(\v_h^\star,q_h^\star)\in\V_h^\star\times Q_h^\star$.
The first backward Euler step requires only the element velocity at the preceding time level. Although an initial facet value may be prescribed as $\Pi_F\gamma_h\w_0$, it does not contribute to the first backward Euler difference.

For clarity, the element momentum flux corresponding to~\eqref{eq:fulld} is
\[
\bm\sigma_h^n=-r_h^n\bm I+\nu\nabla\w_h^n+\mus\Amap(\nabla\w_h^n)-\w_h^n\otimes\w_h^n.
\]
With $j_h^n=\jn{\w_h^{n,\star}}$, its numerical normal trace on $\partial T$ is
\begin{equation}\label{momentum flux trace}
\begin{aligned}
\widehat{\bm\sigma}_h^n\n={}&-\widetilde r_h^n\n+\nu\nabla\w_h^n\n+\mus\Amap(\nabla\w_h^n)\n
-\Bigl(\tfrac{\nu\alpha_1}{h_T}+\alpha_3\Bigr)j_h^n
-\tfrac{\mus\alpha_2}{h_T^2}|j_h^n|j_h^n\\
&-\tfrac12(\w_h^n\cdot\n)(\w_h^n+\widetilde\w_h^n)
-\tfrac12|\w_h^n\cdot\n|j_h^n.
\end{aligned}
\end{equation}
The last two terms select the element value on outflow and the facet value on inflow. In the element momentum equation, this flux is supplemented by the two symmetrizing consistency terms
$-\langle\nu j_h^n,\nabla\v_h\n\rangle_{\partial T}$ and
$-\mus\langle|\nabla\w_h^n|j_h^n,\nabla\v_h\n\rangle_{\partial T}$.
Testing the facet equation enforces weak continuity of the numerical normal momentum flux. This is the same interior-penalty/upwind discretization expressed by the compact form~\eqref{eq:fulld}.

\subsection{Norms, traces, and basic bounds}
For discrete pairs, set
\begin{align}
\nnorm{\v_h^\star}_a^2&=\sum_T\left(\|\nabla\v_h\|_{L^2(T)}^2+	\frac{\alpha_1}{h_T}\|\jn{\v_h^\star}\|_{L^2(\partial T)}^2\right),\label{def:norma}\\
\nnorm{\v_h^\star}_s^3&=\sum_T\left(\|\nabla\v_h\|_{L^3(T)}^3+	\frac{\alpha_2}{h_T^2}\|\jn{\v_h^\star}\|_{L^3(\partial T)}^3\right),\label{def:norms}\\
\nnorm{q_h^\star}_q^2&=\sum_T\left(\|q_h\|_{L^2(T)}^2+h_T\|\widetilde q_h\|_{L^2(\partial T)}^2\right).\label{def:normq}
\end{align}
The zero boundary value of the facet velocity makes $\nnorm{\cdot}_a$ and $\nnorm{\cdot}_s$ norms. The broken Poincar\'e inequality gives $\|\v_h\|_{L^2(\Omega)}\le C_P\nnorm{\v_h^\star}_a$; the interelement jumps and the boundary trace of $\v_h$ are controlled by its element--facet differences.

To evaluate the forms on exact solutions, define
\[
\V_{\rm reg}=[H^2(\Omega)\cap W^{1,3}_0(\Omega)\cap W^{1,\infty}(\Omega)]^d,
\qquad Q_{\rm reg}=H^1(\Omega)\cap L^2_0(\Omega),
\]
\[
\V^\star(h)=(\V_h+\V_{\rm reg})\times(\widetilde\V_h+\gamma_h\V_{\rm reg}),
\quad Q^\star(h)=(Q_h+Q_{\rm reg})\times(\widetilde Q_h+\gamma_hQ_{\rm reg}).
\]
The added trace spaces are traces on all facets, not spaces on $\Gamma$ alone. For $\v^\star\in\V^\star(h)$, also set
\begin{equation}\label{def:normap}
\nnorm{\v^\star}_{a'}^2=\nnorm{\v^\star}_a^2+
\sum_T\tfrac{h_T}{\alpha_1}\|\nabla\v\n\|_{L^2(\partial T)}^2.
\end{equation}

\begin{mylemma}[Linear stability estimates]\label{lem:linear}
For sufficiently large $\alpha_1$, there exist positive mesh-independent constants $C_a,C_b,\beta_b$ such that
\begin{align}
a_h(\v_h^\star,\v_h^\star)&\ge C_a\nu\nnorm{\v_h^\star}_a^2,\label{ah_eq}\\
|a_h(\u^\star,\v_h^\star)|&\le C\nu\nnorm{\u^\star}_{a'}\nnorm{\v_h^\star}_a,\label{ah_bound}\\
|b_h(\v_h^\star,q_h^\star)|&\le C_b\nnorm{\v_h^\star}_a\nnorm{q_h^\star}_q,\label{eq_q1}\\
\beta_b\nnorm{q_h^\star}_q&\le
\sup_{\bm0\ne\v_h^\star\in\V_h^\star}
\frac{b_h(\v_h^\star,q_h^\star)}{\nnorm{\v_h^\star}_a}.\label{eq_q2}
\end{align}
Here $\u^\star\in\V^\star(h)$, $\v_h^\star\in\V_h^\star$, and $q_h^\star\in Q_h^\star$.
\end{mylemma}
\begin{proof}
Coercivity follows from the polynomial inverse-trace inequality and Young's inequality. Cauchy--Schwarz in the weighted facet terms gives continuity in the extended $a'$-norm; the inverse-trace estimate gives $\nnorm{\v_h^\star}_{a'}\lesssim\nnorm{\v_h^\star}_a$ for discrete test functions. The pressure estimates are the standard inf-sup bounds for this element/facet pressure pairing; see \cite{Rhebergen,Rhebergen2018-2}. The pressure normalization removes the joint constant-pressure kernel.
\end{proof}

\begin{mylemma}[Nonlinear diffusion estimates]\label{lemma_sta_s_h}
For sufficiently large $\alpha_2$, there is a mesh-independent $\beta_s>0$ such that
\begin{equation}\label{eq_ts}
S_h(\v_h^\star;\v_h^\star)\ge\beta_s\mus\nnorm{\v_h^\star}_s^3.
\end{equation}
Moreover, for discrete arguments,
\begin{align}
|s_h(\bm z_h^\star,\u_h^\star,\v_h^\star)|
&\le C\mus\nnorm{\bm z_h^\star}_s\nnorm{\u_h^\star}_s\nnorm{\v_h^\star}_s,\label{eq_tb}\\
|S_h(\u_h^\star;\v_h^\star)-S_h(\bm z_h^\star;\v_h^\star)|
&\le C\mus\bigl(\nnorm{\u_h^\star}_s+\nnorm{\bm z_h^\star}_s\bigr)
\nnorm{\u_h^\star-\bm z_h^\star}_s\nnorm{\v_h^\star}_s.\label{eq:Slipschitz}
\end{align}
\end{mylemma}
\begin{proof}
On each element, put $X_T=\|\nabla\v_h\|_{L^3(T)}$ and
$Y_T=(\alpha_2/h_T^2)^{1/3}\|\jn{\v_h^\star}\|_{L^3(\partial T)}$. The two consistency terms in $S_h(\v_h^\star;\v_h^\star)$ have total absolute value at most
\[
C\|\nabla\v_h\|_{L^3(\partial T)}^2\|\jn{\v_h^\star}\|_{L^3(\partial T)}
\le C\alpha_2^{-1/3}X_T^2Y_T.
\]
Since $X_T^2Y_T\le\tfrac23X_T^3+\tfrac13Y_T^3$, a sufficiently large, mesh-independent $\alpha_2$ absorbs these terms into the positive volume and penalty terms, proving~\eqref{eq_ts}.

H\"older's inequality bounds the volume and penalty terms in~\eqref{eq_tb}. For a consistency term, scale its factors as
$h_T^{1/3}\|\nabla\bm z_h\|_{L^3(\partial T)}$,
$h_T^{1/3}\|\nabla\u_h\|_{L^3(\partial T)}$, and
$h_T^{-2/3}\|\jn{\v_h^\star}\|_{L^3(\partial T)}$.
The polynomial trace estimate and H\"older's inequality over the elements yield~\eqref{eq_tb}. Finally, expand each difference in~\eqref{eq:Slipschitz}, use
$||G|-|H||\le|G-H|$ and~\eqref{Conti}, and apply the same bounds.
\end{proof}

\subsection{Exact mass conservation and pressure robustness}
Define the discrete kernel
\begin{equation}\label{def:kernel}
\V_{\divop,h}^\star=\{\v_h^\star\in\V_h^\star:
 b_h(\v_h^\star,q_h^\star)=0\quad\forall q_h^\star\in Q_h^\star\}.
\end{equation}

\begin{mylemma}[Global divergence constraint]\label{lem:div-free}
For every $\v_h^\star\in\V_{\divop,h}^\star$,
\begin{equation}\label{eq:divfree}
\v_h\in H_0(\divop,\Omega),\qquad
\nabla\cdot\v_h|_T=0\quad\text{on every }T\in\Th.
\end{equation}
Here $H_0(\divop,\Omega)$ denotes vector fields in $H(\divop,\Omega)$ with zero normal boundary trace. In addition, $\Pi\w^\star\in\V_{\divop,h}^\star$ for every sufficiently regular $\w\in\V_{\divop0}$.
\end{mylemma}
\begin{proof}
First take $q_h=0$ in~\eqref{def:kernel}. On an interior facet, the two contributions from $\widetilde\v_h$ cancel, since it is single-valued and the normals have opposite signs. Since the element normal traces belong to $P_k(F)$, arbitrary facet pressure tests imply their normal continuity. On a boundary facet, $\widetilde\v_h=0$, so the same tests imply $\v_h\cdot\n=0$. Therefore $\v_h\in H_0(\divop,\Omega)$.

It follows that $\int_\Omega\grad\cdot\v_h=0$. Hence $\grad\cdot\v_h\in Q_h$ and can be used as the element pressure test, with $\widetilde q_h=0$. This yields $\|\grad\cdot\v_h\|_{L^2}^2=0$. Finally,~\eqref{eq:RT4}, normal continuity, and the single-valuedness and zero boundary trace of $\Pi_F\gamma_h\w$ give $b_h(\Pi\w^\star,q_h^\star)=0$.
\end{proof}

\begin{myprop}[Pressure robustness]\label{prop:pressure}
Assume the variational pairings are evaluated exactly. Replacing $\bm f^n$ by $\bm f^n+\nabla\phi^n$, with $\phi^n\in H^1(\Omega)$, leaves the set of discrete velocity solutions of~\eqref{eq:fulld} unchanged. In particular, a uniquely determined discrete velocity is unchanged.
\end{myprop}
\begin{proof}
For $\v_h^\star\in\V_{\divop,h}^\star$, integration by parts in $H_0(\divop,\Omega)$ gives
\[
(\nabla\phi^n,\v_h)=- (\phi^n,\nabla\cdot\v_h)=0.
\]
Thus the equation restricted to the kernel is unchanged. The discrete inf-sup condition recovers a pressure for every such velocity, both before and after the change of force.
\end{proof}

For $\bm b\in H(\divop,\Omega)$ with elementwise zero divergence, elementwise integration by parts gives
\begin{equation}\label{eq:ch1}
\begin{aligned}
c_h(\bm b,\u^\star,\v^\star)
={}&((\bm b\cdot\grad)\u,\v)_{\Th}
-\tfrac12\langle(\bm b\cdot\n)\jn{\u^\star},\v+\widetilde\v\rangle_{\partial\Th}\\
&+\tfrac12\langle|\bm b\cdot\n|\jn{\u^\star},\jn{\v^\star}\rangle_{\partial\Th},
\end{aligned}
\end{equation}
provided the facet functions have the homogeneous boundary trace used here. In particular,
\begin{equation}\label{tri_eq}
c_h(\bm b,\v_h^\star,\v_h^\star)
=\tfrac12\langle|\bm b\cdot\n|,|\jn{\v_h^\star}|^2\rangle_{\partial\Th}\ge0.
\end{equation}
The cancellation of the facet contributions follows from the normal continuity of $\bm b$ across interelement facets, in addition to its broken divergence-free property.

\subsection{Consistency, energy stability, and existence}
\begin{mylemma}[Consistency]\label{lem:consistency}
Suppose $\w^n\in\V_{\rm reg}$, $r^n\in Q_{\rm reg}$, and the exact solution satisfies~\eqref{Smagorinsky_eq} at $t_n$ with sufficient regularity for the indicated normal fluxes. Put $\w^{n,\star}=(\w^n,\gamma_h\w^n)$ and $r^{n,\star}=(r^n,\gamma_h r^n)$. Then
\begin{equation}\label{consist}
\begin{aligned}
&(\dt\w^n,\v_h)_{\Th}+a_h(\w^{n,\star},\v_h^\star)
+S_h(\w^{n,\star};\v_h^\star)
+c_h(\w^n,\w^{n,\star},\v_h^\star)
-b_h(\v_h^\star,r^{n,\star})\\
&\qquad=(\bm f^n,\v_h)_{\Th}+(\dt\w^n-\partial_t\w^n,\v_h)_{\Th},\\
&b_h(\w^{n,\star},q_h^\star)=0.
\end{aligned}
\end{equation}
\end{mylemma}
\begin{proof}
All element--facet differences of the exact velocity vanish. Integrate the remaining volume terms by parts on each element. The facet test functions pair with the normal component of the total flux
\[
-r^n\bm I+\nu\nabla\w^n+\mus\Amap(\nabla\w^n)-\w^n\otimes\w^n.
\]
On each interior facet, the normal traces from the two adjacent elements cancel, while the facet velocity test function vanishes on the external boundary. The same cancellation applies to the pressure traction.
The PDE yields the momentum identity with $\partial_t\w^n$. Replacing $\partial_t\w^n$ by $\dt\w^n$ then gives~\eqref{consist}. The constraint follows from the exact incompressibility of $\w^n$. Moreover, $J_h(\w^{n,\star},\v_h^\star)=0$.
\end{proof}

\begin{mytheorem}[Unconditional energy stability]\label{Theo:energy}
Choose $\alpha_1$ and $\alpha_2$ as in Lemmas~\ref{lem:linear} and~\ref{lemma_sta_s_h}, and let $\alpha_3>0$. Every solution sequence of~\eqref{eq:fulld} satisfies, for $1\le M\le N$,
\begin{equation}\label{eq:energybound}
\begin{aligned}
&\max_{0\le m\le M}\|\w_h^m\|_{L^2}^2
+\sum_{n=1}^M\|\w_h^n-\w_h^{n-1}\|_{L^2}^2\\
&\quad+\tau\sum_{n=1}^M\left[
\nu C_a\nnorm{\w_h^{n,\star}}_a^2+
\mus\beta_s\nnorm{\w_h^{n,\star}}_s^3+
\alpha_3\|j_h^n\|_{L^2(\partial\Th)}^2+
\tfrac12\langle|\w_h^n\cdot\n|,|j_h^n|^2\rangle_{\partial\Th}\right]\\
&\qquad\le C\left(\|\w_h^0\|_{L^2}+\tau\sum_{n=1}^M\|\bm f^n\|_{L^2}\right)^2.
\end{aligned}
\end{equation}

\end{mytheorem}
\begin{proof}
The pressure equation and Lemma~\ref{lem:div-free} imply $\w_h^{n,\star}\in\V_{\divop,h}^\star$. Test the momentum equation with $\w_h^{n,\star}$. The pressure term vanishes, and all four spatial contributions are bounded below by the nonnegative terms in~\eqref{eq:energybound}. Dropping them first gives
\[
\|\w_h^n\|_{L^2}\le\|\w_h^{n-1}\|_{L^2}+\tau\|\bm f^n\|_{L^2},
\]
with the inequality immediate if $\w_h^n=0$. Iteration bounds the maximum velocity norm by the quantity in parentheses in~\eqref{eq:energybound}.

Next retain the spatial terms and use
\[
2(\w_h^n-\w_h^{n-1},\w_h^n)
=\|\w_h^n\|_{L^2}^2-\|\w_h^{n-1}\|_{L^2}^2
+\|\w_h^n-\w_h^{n-1}\|_{L^2}^2.
\]
Multiply by $\tau$, sum in time, and bound the work of the force by
\begin{align*}
   \sum_{n=1}^M \tau\|\bm f^n\|_{L^2}\|\w_h^n\|_{L^2}
	\lesssim &\sum_{n=1}^M \tau\|\bm f^n\|_{L^2}\Bigl(\sum_{i=1}^n\tau\|\bm f^i\|_{L^2}+\|\w_h^{0}\|_{L^2} \Bigr)
	\lesssim   \Big (\sum_{n=1}^M \tau\|\bm f^n\|_{L^2}\Bigr)^2,
\end{align*}
Combining this estimate with the preceding maximum bound and absorbing the initial-data contribution into the absolute constant proves~\eqref{eq:energybound}, after enlarging the constant if necessary. In particular, for zero forcing, the discrete kinetic energy is nonincreasing at every time step.
\end{proof}

\begin{mytheorem}[Existence]\label{thm:existence}
Let $\nu>0$ and let the penalty parameters satisfy the hypotheses of Theorem~\ref{Theo:energy}. For any $\tau>0$, initial element velocity $\w_h^0$, and prescribed $\bm f^n\in[L^2(\Omega)]^d$, problem~\eqref{eq:fulld} admits at least one solution at every time step. The pressure pair is unique once the velocity pair is fixed.
\end{mytheorem}
\begin{proof}
Fix the previous element velocity and put $\bm F^n=\bm f^n+\w_h^{n-1}/\tau$. On the finite-dimensional space $\V_{\divop,h}^\star$, use the inner product generating $\nnorm{\cdot}_a$. Define the continuous map $\Phi$ by the Riesz representation of
\[
\begin{aligned}
(\Phi(\bm z_h^\star),\v_h^\star)_a={}&
\tau^{-1}(\bm z_h,\v_h)+a_h(\bm z_h^\star,\v_h^\star)
+S_h(\bm z_h^\star;\v_h^\star)\\
&+c_h(\bm z_h,\bm z_h^\star,\v_h^\star)
+\alpha_3J_h(\bm z_h^\star,\v_h^\star)-(\bm F^n,\v_h).
\end{aligned}
\]
Continuity follows from the finite-dimensional polynomial representations, continuity of absolute values, and~\eqref{eq:Slipschitz}. By coercivity,~\eqref{tri_eq}, and the broken Poincar\'e inequality,
\[
(\Phi(\bm z_h^\star),\bm z_h^\star)_a
\ge C_a\nu\nnorm{\bm z_h^\star}_a^2
-C_P\|\bm F^n\|_{L^2}\nnorm{\bm z_h^\star}_a.
\]
This is positive on a sphere of radius
$R>C_P\|\bm F^n\|_{L^2}/(C_a\nu)$. The finite-dimensional Brouwer argument therefore gives a zero of $\Phi$ inside this sphere; see \cite{Girault}.

The residual of the momentum equation on the full velocity space vanishes on $\ker b_h$. Finite-dimensional duality and~\eqref{eq_q2} give a pressure pair that represents this residual, and the same inf-sup condition ensures uniqueness of that pressure pair. Repeating the construction in time establishes existence of a solution sequence.
\end{proof}

\subsection{Conditional uniqueness}
The full nonlinear HDG diffusion operator is not assumed to be globally monotone. Accordingly, the following result is conditional and should not be confused with unconditional energy stability.

\begin{mytheorem}[Conditional uniqueness]\label{thm:unique_full_discrete}
Assume the hypotheses of Theorem~\ref{thm:existence}. Suppose that, at each time level, every possible discrete solution obeys common finite bounds $M_s^n,M_c^n$ such that
\begin{align*}
\max_{T\in\Th}\left(\|\nabla\w_h^n\|_{L^\infty(\partial T)}
+h_T^{-1}\|\jn{\w_h^{n,\star}}\|_{L^\infty(\partial T)}\right)&\le M_s^n,\\
\max_{T\in\Th}\left(\|\w_h^n\|_{L^\infty(T)}+
\|\widetilde\w_h^n\|_{L^\infty(\partial T)}\right)&\le M_c^n.
\end{align*}
There are constants $C_S,C_C>0$, depending on the mesh shape regularity, degree, and fixed penalties but not on $h$, $\tau$, $\nu$, or $\delta$, such that the conditions
\begin{equation}\label{eq:uniqueconditions}
C_S\mus M_s^n\le\frac{C_a\nu}{4},\qquad
\frac{2\tau C_C(M_c^n)^2}{\nu}<1,
\qquad n=1,\ldots,N,
\end{equation}
imply uniqueness of the solution sequence.
\end{mytheorem}
\begin{proof}
Compare two solutions at a fixed time level and set $\bm e^\star=\w_{h,1}^\star-\w_{h,2}^\star$. The difference belongs to $\V_{\divop,h}^\star$. For the nonlinear diffusion difference, volume and jump monotonicity give
\begin{equation}\label{eq:uniques}
S_h(\w_{h,1}^\star;\bm e^\star)-S_h(\w_{h,2}^\star;\bm e^\star)
\ge\frac{\mus}{4}\nnorm{\bm e^\star}_s^3
-C_S\mus M_s^n\nnorm{\bm e^\star}_a^2.
\end{equation}
To verify the remaining term,~\eqref{Conti} bounds the flux difference by
$C M_s^n|\nabla\bm e|$ on each facet. In the symmetrizing consistency term, write
\[
|\nabla\w_{h,1}|\jn{\w_{h,1}^\star}
-|\nabla\w_{h,2}|\jn{\w_{h,2}^\star}
=|\nabla\w_{h,1}|\jn{\bm e^\star}
+(|\nabla\w_{h,1}|-|\nabla\w_{h,2}|)\jn{\w_{h,2}^\star}.
\]
The absolute values of the two facet contributions are bounded by
\[
C M_s^n\sum_T\left[
\|\nabla\bm e\|_{L^2(\partial T)}\|\jn{\bm e^\star}\|_{L^2(\partial T)}
+h_T\|\nabla\bm e\|_{L^2(\partial T)}^2\right]
\le C_S M_s^n\nnorm{\bm e^\star}_a^2.
\]
This uses polynomial inverse traces only for $\bm e$ and retains the local element sizes.

For convection, expand
$\w_{h,1}\otimes\w_{h,1}-\w_{h,2}\otimes\w_{h,2}
=\w_{h,1}\otimes\bm e+\bm e\otimes\w_{h,2}$.
The second term cancels its corresponding central facet term by integration by parts. The first volume term is bounded by $C M_c^n\|\bm e\|_{L^2}\|\grad\bm e\|_{L^2}$. Split the upwind difference as
\[
|\w_{h,1}\cdot\n|\jn{\bm e^\star}
+(|\w_{h,1}\cdot\n|-|\w_{h,2}\cdot\n|)\jn{\w_{h,2}^\star}.
\]
The first part is nonnegative when tested with $\jn{\bm e^\star}$, while the second is bounded using
$||\w_{h,1}\cdot\n|-|\w_{h,2}\cdot\n||\le|\bm e|$.
The remaining central and upwind facet terms are therefore bounded by
$C M_c^n\sum_T\|\bm e\|_{L^2(\partial T)}\|\jn{\bm e^\star}\|_{L^2(\partial T)}$.
A weighted Cauchy--Schwarz inequality, polynomial inverse traces, and Young's inequality yield
\begin{equation}\label{eq:uniquec}
\begin{aligned}
&c_h(\w_{h,1},\w_{h,1}^\star,\bm e^\star)
-c_h(\w_{h,2},\w_{h,2}^\star,\bm e^\star) \ge-\frac{C_a\nu}{2}\nnorm{\bm e^\star}_a^2
-\frac{C_C(M_c^n)^2}{\nu}\|\bm e\|_{L^2}^2.
\end{aligned}
\end{equation}
Testing the difference equation with $\bm e^{n,\star}$, using~\eqref{eq:uniques}--\eqref{eq:uniquec}, and retaining the nonnegative terms gives
\[
\begin{aligned}
&\left(\frac1{2\tau}-\frac{C_C(M_c^n)^2}{\nu}\right)\|\bm e^n\|_{L^2}^2
+\frac1{2\tau}\|\bm e^n-\bm e^{n-1}\|_{L^2}^2
+\frac{C_a\nu}{4}\nnorm{\bm e^{n,\star}}_a^2\\
&\qquad+\frac{\mus}{4}\nnorm{\bm e^{n,\star}}_s^3
+\alpha_3\|\jn{\bm e^{n,\star}}\|_{L^2(\partial\Th)}^2
\le\frac1{2\tau}\|\bm e^{n-1}\|_{L^2}^2.
\end{aligned}
\]
The first coefficient is positive by~\eqref{eq:uniqueconditions}. Since the initial element velocities agree, induction proves equality of both the element and facet velocities. The inf-sup condition then proves equality of the pressure pairs.
\end{proof}

\begin{myremark}
The common-envelope formulation avoids applying a bound for one candidate solution to an unspecified pair of solutions. The conditions~\eqref{eq:uniqueconditions} are sufficient, not necessary, and are viscosity-dependent. The Reynolds-semi-robust error estimate below applies to any solution sequence and does not rely on these uniqueness conditions.
\end{myremark}

\section{Velocity error analysis}\label{Sec:err}
The analysis is performed on the discrete divergence-free kernel. This eliminates pressure terms without estimating a pressure interpolation error. All variational forms are understood with exact integration. Errors introduced by numerical quadrature or inexact nonlinear solves would require additional residual terms.

\begin{myassumption}[Regularity for nodal error estimates]\label{Assum1}
Let $1\le m\le k$. Assume that a solution of~\eqref{Smagorinsky_eq} exists with
\begin{equation}\label{eq:regularity}
\begin{aligned}
\w\in{}&H^2(0,T;[L^2(\Omega)]^d)
\cap H^1(0,T;[H^{m+1}(\Omega)]^d)\\
&\cap C\!\left([0,T];[W^{m+1,3}(\Omega)]^d
\cap[W^{1,\infty}(\Omega)]^d\right),\\
r\in{}&C([0,T];H^1(\Omega)\cap L^2_0(\Omega)),
\end{aligned}
\end{equation}
with the prescribed zero velocity trace and incompressibility, and with well-defined nodal forcing $\bm f^n=\bm f(t_n)$. The PDE and its consistency identity are assumed valid at the time nodes.
\end{myassumption}

The continuity assumptions in~\eqref{eq:regularity} provide bounded spatial norms at every time node. In particular, they justify the sums of cubed $W^{m+1,3}$ approximation errors below. An $L^2(0,T;W^{m+1,3})$ assumption alone would not control such nodal cubic sums. These are sufficient assumptions for the present proof; no claim of their necessity is made.

\subsection{Error splitting and a kernel error equation}
Set
\begin{align*}
\bm p^n&=\Pi_{RT}\w^n,&\widetilde{\bm p}^n&=\Pi_F\gamma_h\w^n,
&\bm p^{n,\star}&=(\bm p^n,\widetilde{\bm p}^n),\\
\bm\zeta^n&=\w^n-\bm p^n,&\widetilde{\bm\zeta}^n&=\gamma_h\w^n-\widetilde{\bm p}^n,
&\bm\zeta^{n,\star}&=(\bm\zeta^n,\widetilde{\bm\zeta}^n),\\
\bm\xi^n&=\w_h^n-\bm p^n,&\widetilde{\bm\xi}^n&=\widetilde\w_h^n-\widetilde{\bm p}^n,
&\bm\xi^{n,\star}&=(\bm\xi^n,\widetilde{\bm\xi}^n).
\end{align*}
Then $\w^n-\w_h^n=\bm\zeta^n-\bm\xi^n$, and
\begin{equation}\label{eq:errorkernel}
\bm\xi^{n,\star}\in\V_{\divop,h}^\star,\qquad
\bm\zeta^n,\bm\xi^n\in H_0(\divop,\Omega),\qquad
\nabla\cdot\bm\zeta^n=\nabla\cdot\bm\xi^n=0.
\end{equation}
Moreover, $\bm\xi^0=0$ and $\jn{\bm p^{n,\star}}=-\jn{\bm\zeta^{n,\star}}$. We use the abbreviations
$J_\xi^n=\jn{\bm\xi^{n,\star}}$, $J_p^n=\jn{\bm p^{n,\star}}$, and
$J_\zeta^n=\jn{\bm\zeta^{n,\star}}$.

Define the monotone part of the nonlinear difference by
\begin{equation}\label{eq:error1}
\begin{aligned}
\mathcal M_h^n(\v_h^\star)=\mus\sum_T\Bigl[&
(\Amap(\nabla\w_h^n)-\Amap(\nabla\bm p^n),\nabla\v_h)_T+\tfrac{\alpha_2}{h_T^2}
\langle\Amap(\jn{\w_h^{n,\star}})-\Amap(J_p^n),\jn{\v_h^\star}\rangle_{\partial T}\Bigr].
\end{aligned}
\end{equation}
By~\eqref{Mono},
\begin{equation}\label{eq:monotoneerror}
\mathcal M_h^n(\bm\xi^{n,\star})\ge\tfrac{\mus}{4}\nnorm{\bm\xi^{n,\star}}_s^3.
\end{equation}
For all $\v_h^\star\in\V_{\divop,h}^\star$, subtracting~\eqref{consist} from~\eqref{eq:fulld} and rearranging gives
\begin{equation}\label{eq:error}
\begin{aligned}
(\dt\bm\xi^n,\v_h)+a_h(\bm\xi^{n,\star},\v_h^\star)
+\mathcal M_h^n(\v_h^\star)
+c_h(\w_h^n,\bm\xi^{n,\star},\v_h^\star)&\\
+\alpha_3J_h(\bm\xi^{n,\star},\v_h^\star)
&=\sum_{i\in\{1,2,4,5,6\}}\mathcal D_i^n(\v_h^\star),
\end{aligned}
\end{equation}
where
\begin{align}
\mathcal D_1^n(\v_h^\star)&=(\dt\bm\zeta^n+\partial_t\w^n-\dt\w^n,\v_h),\label{def:D1}\\
\mathcal D_2^n(\v_h^\star)&=a_h(\bm\zeta^{n,\star},\v_h^\star),\label{def:D2}\\
\mathcal D_4^n(\v_h^\star)&=c_h(\w^n,\w^{n,\star},\v_h^\star)
-c_h(\w_h^n,\w_h^{n,\star},\v_h^\star)
+c_h(\w_h^n,\bm\xi^{n,\star},\v_h^\star),\label{def:D4}\\
\mathcal D_5^n(\v_h^\star)&=S_h(\w^{n,\star};\v_h^\star)
-S_h(\w_h^{n,\star};\v_h^\star)+\mathcal M_h^n(\v_h^\star),\label{def:D5}\\
\mathcal D_6^n(\v_h^\star)&=\alpha_3J_h(\bm\zeta^{n,\star},\v_h^\star).\label{def:D6}
\end{align}
There is no $\mathcal D_3$ pressure residual: for any kernel test, both the discrete pressure term and the exact pressure term vanish. For the latter, normal continuity, zero normal boundary trace, and elementwise zero divergence give
$b_h(\v_h^\star,(r^n,\gamma_h r^n))=0$. This is the precise pressure cancellation used in the proof.

For later use, define at each time level
\begin{equation}\label{def:KU}
K_T^n=\|\nabla\w^n\|_{L^\infty(T)},\qquad
U_T^n=\|\w^n\|_{L^\infty(T)},\qquad K_n=\max_TK_T^n.
\end{equation}

\subsection{Temporal, linear, and convective residuals}
\begin{mylemma}\label{lem:simpleresiduals}
For every $\varepsilon>0$,
\begin{align}
|\mathcal D_1^n(\bm\xi^{n,\star})|
&\le\varepsilon\|\bm\xi^n\|_{L^2}^2+C_\varepsilon\mathcal T_n,\label{eq:D1}\\
|\mathcal D_2^n(\bm\xi^{n,\star})|
&\le\tfrac{C_a\nu}{2}\nnorm{\bm\xi^{n,\star}}_a^2
+C\nu\nnorm{\bm\zeta^{n,\star}}_{a'}^2,\label{eq:D2}\\
|\mathcal D_6^n(\bm\xi^{n,\star})|
&\le\tfrac{\alpha_3}{4}\|J_\xi^n\|_{L^2(\partial\Th)}^2
+C\alpha_3\|J_\zeta^n\|_{L^2(\partial\Th)}^2,\label{eq:D6}
\end{align}
where
\begin{equation}\label{def:temporalresidual}
\mathcal T_n=\tau\int_{t_{n-1}}^{t_n}\|\partial_{tt}\w(t)\|_{L^2}^2dt
+\frac1\tau\int_{t_{n-1}}^{t_n}\|\partial_t\bm\zeta(t)\|_{L^2}^2dt.
\end{equation}
\end{mylemma}
\begin{proof}
The identities
\[
\partial_t\w^n-\dt\w^n
=\frac1\tau\int_{t_{n-1}}^{t_n}(t-t_{n-1})\partial_{tt}\w(t)\,dt,
\qquad
\dt\bm\zeta^n=\frac1\tau\int_{t_{n-1}}^{t_n}\partial_t\bm\zeta(t)\,dt
\]
give~\eqref{eq:D1} by Cauchy--Schwarz and Young's inequality. The time-independent interpolation operator commutes with the time derivative in the indicated spaces. The other two estimates follow from~\eqref{ah_bound} and Young's inequality. In particular, the temporal truncation residual contributes $\tau^2\|\partial_{tt}\w\|_{L^2(0,T;L^2)}^2$ after multiplication by $\tau$ and summation.
\end{proof}

\begin{mylemma}[Convective residual]\label{lemma:D4}
Under Assumption~\ref{Assum1},
\begin{equation}\label{eq:D4}
\begin{aligned}
\mathcal D_4^n(\bm\xi^{n,\star})\le{}&
C(1+K_n)\|\bm\xi^n\|_{L^2}^2
+C\sum_T(K_T^n)^2\|\bm\zeta^n\|_{L^2(T)}^2\\
&+C\sum_T(U_T^n+h_TK_T^n)
\left(\|\bm\zeta^n\|_{L^2(\partial T)}^2+
\|\widetilde{\bm\zeta}^n\|_{L^2(\partial T)}^2\right)\\
&+\tfrac14\langle|\w_h^n\cdot\n|,|J_\xi^n|^2\rangle_{\partial\Th}.
\end{aligned}
\end{equation}
The constant contains no inverse power of $\nu$.
\end{mylemma}
\begin{proof}
Suppress the superscript $n$. Since $\jn{\w^\star}=0$, the upwind part of $c_h(\bm b,\w^\star,\v_h^\star)$ vanishes. Therefore the difference in its advecting argument is linear in this particular expression. Using~\eqref{eq:ch1},
\begin{equation}\label{eq:convsplit}
\begin{aligned}
\mathcal D_4(\bm\xi^\star)={}&
(((\bm\zeta-\bm\xi)\cdot\nabla)\w,\bm\xi)_{\Th}
-((\w_h\cdot\grad)\bm\xi,\bm\zeta)_{\Th}\\
&+\tfrac12\langle(\w_h\cdot\n)(\bm\zeta+\widetilde{\bm\zeta}),J_\xi\rangle_{\partial\Th}
+\tfrac12\langle|\w_h\cdot\n|(\bm\zeta-\widetilde{\bm\zeta}),J_\xi\rangle_{\partial\Th}.
\end{aligned}
\end{equation}
The first volume term is bounded by
\[
C(1+K_n)\|\bm\xi\|_{L^2}^2
+C\sum_TK_T^2\|\bm\zeta\|_{L^2(T)}^2.
\]
For the second term, let $\overline{\w}_T$ be the element mean of $\w$. The RT moment property~\eqref{eq:RTdef} gives
$((\overline{\w}_T\cdot\nabla)\bm\xi,\bm\zeta)_T=0$.
Since $\w_h=\w-\bm\zeta+\bm\xi$, the remaining terms are bounded using
\[
\|\bm\zeta\|_{L^\infty(T)}+
\|\w-\overline{\w}_T\|_{L^\infty(T)}\lesssim h_TK_T,
\qquad \|\nabla\bm\xi\|_{L^2(T)}\lesssim h_T^{-1}\|\bm\xi\|_{L^2(T)}.
\]
In particular,
\[
|((\bm\xi\cdot\nabla)\bm\xi,\bm\zeta)_T|
\lesssim K_T\|\bm\xi\|_{L^2(T)}^2.
\]
The other pieces are bounded by
$C K_T\|\bm\xi\|_{L^2(T)}\|\bm\zeta\|_{L^2(T)}$ and then by Young's inequality. This retains the $K_T\|\bm\xi\|^2$ term rather than incorrectly absorbing it into an arbitrary small constant.

For each of the two facet terms in~\eqref{eq:convsplit}, a weighted Young inequality absorbs one eighth of the upwind quadratic form and leaves a constant times
\[
\sum_T\int_{\partial T}|\w_h\cdot\n|
\bigl(|\bm\zeta|^2+|\widetilde{\bm\zeta}|^2\bigr).
\]
Use $\w_h=\bm p+\bm\xi$, $\|\bm p\|_{L^\infty(\partial T)}\lesssim U_T+h_TK_T$, and
\[
\|\bm\zeta\|_{L^\infty(\partial T)}+
\|\widetilde{\bm\zeta}\|_{L^\infty(\partial T)}\lesssim h_TK_T.
\]
If $Z_T^2=\|\bm\zeta\|_{L^2(\partial T)}^2+
\|\widetilde{\bm\zeta}\|_{L^2(\partial T)}^2$, the term with $\bm\xi$ is at most
\[
C h_TK_T\|\bm\xi\|_{L^2(\partial T)}Z_T
\le C K_T\|\bm\xi\|_{L^2(T)}^2+C h_TK_T Z_T^2.
\]
Here the inverse trace is applied to the discrete polynomial $\bm\xi$, not to an interpolation error. Summing the estimates proves~\eqref{eq:D4}.
\end{proof}

\subsection{Nonlinear eddy-diffusion residual}
The local approximation quantity required to control nonlinear facet fluxes is
\begin{equation}\label{def:Es}
\mathcal E_{s,T}^n=
\left(\|\nabla\bm\zeta^n\|_{L^3(T)}^3
+h_T\|\nabla\bm\zeta^n\|_{L^3(\partial T)}^3
+\tfrac{\alpha_2}{h_T^2}\|J_\zeta^n\|_{L^3(\partial T)}^3\right)^{1/3}.
\end{equation}
The gradient trace is deliberately retained. By Lemma~\ref{RT-lem},
\begin{equation}\label{eq:Esapprox}
\mathcal E_{s,T}^n\lesssim h_T^m|\w^n|_{W^{m+1,3}(T)}.
\end{equation}

\begin{mylemma}[Nonlinear residual]\label{lemma:D5}
Let $\alpha_3>0$ be fixed and choose $\alpha_2$ sufficiently large. Under Assumption~\ref{Assum1},
\begin{equation}\label{eq:D5}
\begin{aligned}
\mathcal D_5^n(\bm\xi^{n,\star})\le{}&
\tfrac{\mus}{8}\nnorm{\bm\xi^{n,\star}}_s^3
+\tfrac{\alpha_3}{4}\|J_\xi^n\|_{L^2(\partial\Th)}^2
+\frac{C\mus^2}{\alpha_3}\sum_T h_T^{-3}(K_T^n)^2\|\bm\xi^n\|_{L^2(T)}^2\\
&+C\mus\sum_T\left[(\mathcal E_{s,T}^n)^3
+|T|^{1/2}(K_T^n)^{3/2}(\mathcal E_{s,T}^n)^{3/2}\right].
\end{aligned}
\end{equation}
The constants and the penalty threshold do not involve inverse powers of $\nu$.
\end{mylemma}
\begin{proof}
Suppress the time index and write $P=\nabla\bm p$, $X=\nabla\bm\xi$, and $J_p=\jn{\bm p^\star}$. The following splitting is an exact algebraic identity:
\begin{equation}\label{eq:S1}
\mathcal D_5(\v_h^\star)=\mathcal R_s(\v_h^\star)+\mathcal B_s(\v_h^\star),
\end{equation}
where
\begin{align*}
\mathcal R_s(\v_h^\star)&=S_h(\w^\star;\v_h^\star)-S_h(\bm p^\star;\v_h^\star),\\
\mathcal B_s(\v_h^\star)&=\mus\sum_T\Bigl[
\langle(\Amap(P+X)-\Amap(P))\n,\jn{\v_h^\star}\rangle_{\partial T}\\
&\hspace{28mm}+\langle|P+X|(J_p+J_\xi)-|P|J_p,\nabla\v_h\n\rangle_{\partial T}\Bigr].
\end{align*}
Thus the entire monotone volume and jump difference has been removed from $\mathcal B_s$, but its two consistency-flux differences remain.

\paragraph{Approximation terms}
For $\mus>0$, expansion of $\mathcal R_s$ gives
\begin{equation}\label{eq:D51}
\begin{aligned}
\mus^{-1}\mathcal R_s(\v_h^\star)=\sum_T\Bigl[&
(\Amap(\nabla\w)-\Amap(P),\nabla\v_h)_T
-\langle(\Amap(\nabla\w)-\Amap(P))\n,\jn{\v_h^\star}\rangle_{\partial T}\\
&+\langle|P|J_p,\nabla\v_h\n\rangle_{\partial T}
-\tfrac{\alpha_2}{h_T^2}\langle\Amap(J_p),\jn{\v_h^\star}\rangle_{\partial T}\Bigr].
\end{aligned}
\end{equation}
For $\mus=0$ the assertion is immediate, so division by $\mus$ causes no difficulty. On a fixed element, let
\[
\begin{aligned}
a_T&=\|\nabla\bm\zeta\|_{L^3(T)},&
b_T&=\|\nabla\bm\zeta\|_{L^3(\partial T)},&
z_T&=\|J_p\|_{L^3(\partial T)},\\
x_T&=\|X\|_{L^3(T)},&
y_T&=(\alpha_2/h_T^2)^{1/3}\|J_\xi\|_{L^3(\partial T)}.
\end{aligned}
\]
From~\eqref{eq:RTbound3} and shape regularity,
\[
\|P\|_{L^3(T)}\lesssim h_T^{d/3}K_T,\qquad
\|P\|_{L^3(\partial T)}\lesssim h_T^{(d-1)/3}K_T.
\]
Using~\eqref{Conti} and Young's inequality, the four terms in~\eqref{eq:D51}, tested with $\bm\xi^\star$, are bounded, respectively, by
\begin{align}
&\varepsilon x_T^3+C_\varepsilon\left(a_T^3+h_T^{d/2}K_T^{3/2}a_T^{3/2}\right),\label{eq:rsvol}\\
&\varepsilon y_T^3+C_\varepsilon\left(h_T b_T^3+h_T^{(d+1)/2}K_T^{3/2}b_T^{3/2}\right),\label{eq:rsflux}\\
&\varepsilon x_T^3+C_\varepsilon h_T^{(d-2)/2}K_T^{3/2}z_T^{3/2},\label{eq:rsadjoint}\\
&\varepsilon y_T^3+C_\varepsilon\tfrac{\alpha_2}{h_T^2}z_T^3.\label{eq:rspenalty}
\end{align}
For~\eqref{eq:rsflux}, the factor $h_T$ in front of $b_T^3$ comes from pairing a boundary flux with the $h_T^{-2/3}$-weighted jump norm. In~\eqref{eq:rsadjoint}, the only inverse trace is
$\|X\|_{L^3(\partial T)}\lesssim h_T^{-1/3}x_T$ for the discrete gradient $X$.
Because $|J_p|=|J_\zeta|$ and $|T|\simeq h_T^d$, all the data terms in~\eqref{eq:rsvol}--\eqref{eq:rspenalty} are controlled by
\[
C_\varepsilon\left[\mathcal E_{s,T}^3+
|T|^{1/2}K_T^{3/2}\mathcal E_{s,T}^{3/2}\right].
\]
Consequently, for any prescribed $\varepsilon>0$ after rescaling it by an absolute factor,
\begin{equation}\label{eq:S2}
|\mathcal R_s(\bm\xi^\star)|
\le\varepsilon\mus\nnorm{\bm\xi^\star}_s^3
+C_\varepsilon\mus\sum_T
\left[\mathcal E_{s,T}^3+|T|^{1/2}K_T^{3/2}\mathcal E_{s,T}^{3/2}\right].
\end{equation}
No inverse-trace estimate has been applied to $\nabla\bm\zeta$.

\paragraph{Discrete nonlinear facet differences}
The pointwise inequalities
\begin{align*}
|\Amap(P+X)-\Amap(P)|&\le C(|X|^2+|P||X|),\\
\bigl||P+X|(J_p+J_\xi)-|P|J_p\bigr|
&\le(|P|+|X|)|J_\xi|+|X||J_p|
\end{align*}
imply
\begin{equation}\label{eq:bsbound}
|\mathcal B_s(\bm\xi^\star)|
\le C\mus\sum_T\int_{\partial T}
\left(|X|^2|J_\xi|+K_T|X||J_\xi|+|X|^2|J_p|\right).
\end{equation}
For the cubic mixed term,
\[
\int_{\partial T}|X|^2|J_\xi|
\lesssim\alpha_2^{-1/3}x_T^2y_T.
\]
A sufficiently large fixed $\alpha_2$ makes its coefficient in front of $x_T^3+y_T^3$ as small as required. The term involving $J_p$ satisfies
\[
\int_{\partial T}|X|^2|J_p|
\lesssim h_T^{-2/3}x_T^2 z_T
\le\varepsilon x_T^3+C_\varepsilon h_T^{-2}z_T^3.
\]
Its last contribution is included in $\mathcal E_{s,T}^3$.

It remains to estimate the quadratic term involving $K_T$. The polynomial inverse and trace inequalities give
\[
\|X\|_{L^2(\partial T)}\lesssim h_T^{-3/2}\|\bm\xi\|_{L^2(T)}.
\]
Therefore, for any $\varepsilon>0$,
\begin{equation}\label{eq:S3}
C\mus K_T\|X\|_{L^2(\partial T)}\|J_\xi\|_{L^2(\partial T)}
\le\varepsilon\alpha_3\|J_\xi\|_{L^2(\partial T)}^2
+\frac{C_\varepsilon\mus^2}{\alpha_3h_T^3}K_T^2\|\bm\xi\|_{L^2(T)}^2.
\end{equation}
This is where the viscosity-independent quadratic facet penalty is essential to the present proof. We choose the Young parameters and then $\alpha_2$ so that the total cubic contribution of~\eqref{eq:S2} and~\eqref{eq:bsbound} is at most $\mus\nnorm{\bm\xi^\star}_s^3/8$, and the total quadratic jump contribution is at most $\alpha_3\|J_\xi\|_{L^2(\partial\Th)}^2/4$. Summation proves~\eqref{eq:D5}.
\end{proof}

\begin{myremark}[Local-to-global summation]\label{rem:summation}
The mixed approximation term must be summed before replacing local quantities by global norms. In particular, Cauchy--Schwarz gives
\begin{equation}\label{eq:sumcorrect}
\begin{aligned}
\sum_T|T|^{1/2}(K_T^n)^{3/2}(\mathcal E_{s,T}^n)^{3/2}
&\le |\Omega|^{1/2}K_n^{3/2}
\left(\sum_T(\mathcal E_{s,T}^n)^3\right)^{1/2}\\
&\lesssim h^{3m/2}K_n^{3/2}|\w^n|_{W^{m+1,3}(\Omega)}^{3/2}.
\end{aligned}
\end{equation}
The factor $|T|^{1/2}$ is absorbed by the sum of element volumes. It does not survive as an additional global factor $h^{d/2}$. Likewise, the gradient-trace term in~\eqref{def:Es} has the correct $h_T$ weight and cannot be discarded by treating the interpolation error as a polynomial.
\end{myremark}

\subsection{A priori error bound}
Define the nonnegative approximation residual
\begin{equation}\label{eq:Lambda}
\begin{aligned}
\mathcal R_n={}&\nu\nnorm{\bm\zeta^{n,\star}}_{a'}^2
+\sum_T(K_T^n)^2\|\bm\zeta^n\|_{L^2(T)}^2
+\alpha_3\|J_\zeta^n\|_{L^2(\partial\Th)}^2\\
&+\sum_T(U_T^n+h_TK_T^n)
\left(\|\bm\zeta^n\|_{L^2(\partial T)}^2+
\|\widetilde{\bm\zeta}^n\|_{L^2(\partial T)}^2\right)\\
&+\mus\sum_T\left[(\mathcal E_{s,T}^n)^3+
|T|^{1/2}(K_T^n)^{3/2}(\mathcal E_{s,T}^n)^{3/2}\right],\\
\Lambda_N={}&\tau^2\|\partial_{tt}\w\|_{L^2(0,T;L^2)}^2
+\|\partial_t\bm\zeta\|_{L^2(0,T;L^2)}^2+\tau\sum_{n=1}^N\mathcal R_n.
\end{aligned}
\end{equation}
For a fixed constant $C_0$ large enough to dominate the residual estimates, set
\begin{equation}\label{eq:errcon}
\Gamma_n=C_0\left(1+K_n+
\frac{\mus^2}{\alpha_3}\max_{T\in\Th}h_T^{-3}(K_T^n)^2\right),
\qquad \tau\max_{1\le n\le N}\Gamma_n<1,
\end{equation}
and define
\begin{equation}\label{def:Gronwall}
\mathcal G_N=\exp\!\left(\tau\sum_{n=1}^N\frac{\Gamma_n}{1-\tau\Gamma_n}\right).
\end{equation}
The constant $C_0$ has no dependence on inverse powers of $\nu$. We display the factor $\alpha_3^{-1}$ that enters the nonlinear facet estimate rather than hiding it in the mesh/filter-scale ratio.

\begin{mytheorem}[Pressure-robust, Reynolds-semi-robust velocity estimate]\label{the0::err1}
Suppose Assumption~\ref{Assum1} holds and the penalty parameters satisfy the preceding coercivity and residual bounds. Let $(\w_h^{n,\star},r_h^{n,\star})$ be any solution sequence of~\eqref{eq:fulld}, initialized by $\w_h^0=\Pi_{RT}\w_0$. Under~\eqref{eq:errcon},
\begin{equation}\label{eq::err1}
\begin{aligned}
&\max_{0\le n\le N}\|\bm\xi^n\|_{L^2}^2
+\sum_{n=1}^N\|\bm\xi^n-\bm\xi^{n-1}\|_{L^2}^2\\
&\quad+\tau\sum_{n=1}^N\Bigl[
\nu\nnorm{\bm\xi^{n,\star}}_a^2
+\mus\nnorm{\bm\xi^{n,\star}}_s^3
+\alpha_3\|J_\xi^n\|_{L^2(\partial\Th)}^2
+\langle|\w_h^n\cdot\n|,|J_\xi^n|^2\rangle_{\partial\Th}\Bigr]
\le C\mathcal G_N\Lambda_N.
\end{aligned}
\end{equation}
The constants contain no explicit inverse powers of $\nu$, and no pressure approximation term occurs on the right-hand side.
\end{mytheorem}
\begin{proof}
Take $\v_h^\star=\bm\xi^{n,\star}$ in~\eqref{eq:error}, which is admissible by~\eqref{eq:errorkernel}. Apply coercivity~\eqref{ah_eq}, monotonicity~\eqref{eq:monotoneerror}, the upwind identity~\eqref{tri_eq}, and Lemmas~\ref{lem:simpleresiduals}, \ref{lemma:D4}, and~\ref{lemma:D5}. Choose a fixed positive $\varepsilon$ in~\eqref{eq:D1}. With $E_n=\|\bm\xi^n\|_{L^2}^2$ and $I_n=\|\bm\xi^n-\bm\xi^{n-1}\|_{L^2}^2$, multiplication by $2\tau$ yields
\begin{equation}\label{eq:energyrecurrence}
E_n-E_{n-1}+I_n+\tau\mathcal H_n
\le\tau\Gamma_n E_n+C\tau(\mathcal T_n+\mathcal R_n),
\end{equation}
where
\[
\mathcal H_n=C_a\nu\nnorm{\bm\xi^{n,\star}}_a^2
+\tfrac{\mus}{4}\nnorm{\bm\xi^{n,\star}}_s^3
+\alpha_3\|J_\xi^n\|_{L^2(\partial\Th)}^2
+\tfrac12\langle|\w_h^n\cdot\n|,|J_\xi^n|^2\rangle_{\partial\Th}.
\]
The fixed constant $C_0$ in~\eqref{eq:errcon} is chosen to include all coefficients of $E_n$, including the factor two introduced in this step.

Summing~\eqref{eq:energyrecurrence} from $n=1$ to any $M\le N$ and using $E_0=0$ gives
\[
E_M+\sum_{n=1}^M I_n+\tau\sum_{n=1}^M\mathcal H_n
\le\tau\sum_{n=1}^M\Gamma_nE_n+C\Lambda_M.
\]
The discrete Gronwall inequality with $\tau\Gamma_n<1$ yields
\[
E_M+\sum_{n=1}^M I_n+\tau\sum_{n=1}^M\mathcal H_n
\le C\exp\!\left(\tau\sum_{n=1}^M\frac{\Gamma_n}{1-\tau\Gamma_n}\right)\Lambda_M;
\]
see, for example, \cite{Heywood,Quarteroni}. Applying this for every $M$ controls the maximum of $E_M$; applying it for $M=N$ controls all accumulated terms. Absorbing the fixed positive coefficients and combining the two bounds proves~\eqref{eq::err1}.
\end{proof}

\begin{mycorollary}[Approximation-order form]\label{cor:rates}
Under the hypotheses of Theorem~\ref{the0::err1},
\begin{equation}\label{eq:err}
\begin{aligned}
&\max_{0\le n\le N}\|\w^n-\w_h^n\|_{L^2}^2
+\nu\tau\sum_{n=1}^N\|\grad(\w^n-\w_h^n)\|_{L^2}^2\\
&\qquad\le C\mathcal G_N
\left[\tau^2+\nu h^{2m}+h^{2m+1}
+\mus\bigl(h^{3m}+h^{3m/2}\bigr)\right].
\end{aligned}
\end{equation}
The constant may depend on the solution norms in Assumption~\ref{Assum1}, $T$, and the fixed parameters. It contains no explicit inverse power of $\nu$. The quantity $\mathcal G_N$ still uses the local mesh sizes in~\eqref{eq:errcon}.
\end{mycorollary}
\begin{proof}
The interpolation and trace estimates give
\[
\nnorm{\bm\zeta^{n,\star}}_{a'}^2\lesssim h^{2m}|\w^n|_{H^{m+1}}^2,
\qquad
\sum_T\bigl(\|\bm\zeta^n\|_{L^2(\partial T)}^2+
\|\widetilde{\bm\zeta}^n\|_{L^2(\partial T)}^2\bigr)
\lesssim h^{2m+1}|\w^n|_{H^{m+1}}^2,
\]
and
\[
\|\partial_t\bm\zeta\|_{L^2(0,T;L^2)}^2
\lesssim h^{2m+2}\|\partial_t\w\|_{L^2(0,T;H^{m+1})}^2.
\]
By~\eqref{eq:Esapprox} and~\eqref{eq:sumcorrect},
\[
\begin{aligned}
\tau\sum_n\sum_T(\mathcal E_{s,T}^n)^3
&\lesssim T h^{3m}\max_t|\w(t)|_{W^{m+1,3}}^3,\\
\tau\sum_n\sum_T|T|^{1/2}(K_T^n)^{3/2}(\mathcal E_{s,T}^n)^{3/2}
&\lesssim T h^{3m/2}\max_t\|\nabla\w(t)\|_{L^\infty}^{3/2}
\max_t|\w(t)|_{W^{m+1,3}}^{3/2}.
\end{aligned}
\]
These bounds control $\Lambda_N$. Add the element interpolation error to~\eqref{eq::err1}, use the triangle inequality for the velocity and its broken gradient, and absorb $h^{2m+2}$ into $h^{2m+1}$ since $h\le1$. This proves~\eqref{eq:err}.
\end{proof}

\begin{mycorollary}[Pre-asymptotic regime]\label{cor:preasymptotic}
Assume additionally that the meshes are quasi-uniform, $\alpha_3>0$ is fixed, $\delta\le C_\delta h$, and
\begin{equation}\label{eq:margin}
\tau\max_n\Gamma_n\le\theta<1
\end{equation}
with a fixed $\theta$. If the solution norms entering~\eqref{eq:err} are bounded uniformly in the parameters being varied, then $\mathcal G_N$ is uniformly bounded. For $m=k$,
\begin{equation}\label{eq:generalrate}
\max_n\|\w^n-\w_h^n\|_{L^2}
\le C\left[\tau+\nu^{1/2}h^k+h^{k+1/2}
+C_s\delta\bigl(h^{3k/2}+h^{3k/4}\bigr)\right].
\end{equation}
If also $\nu\le C_\nu h$, the right-hand side is bounded by
\begin{equation}\label{eq:prerate}
C\left(\tau+h^{k+1/2}+h^{3k/2+1}+h^{3k/4+1}\right).
\end{equation}
For $\mus=0$, the convection-dominated bound reduces to $C(\tau+h^{k+1/2})$.
\end{mycorollary}
\begin{proof}
Quasi-uniformity gives
$\mus^2\max_T h_T^{-3}\lesssim C_s^4\delta^4h^{-3}\lesssim h$.
Consequently, $\Gamma_n$ is uniformly bounded when $K_n$ is. By~\eqref{eq:margin},
$\mathcal G_N\le\exp(T\max_n\Gamma_n/(1-\theta))$.
Taking square roots in~\eqref{eq:err} proves~\eqref{eq:generalrate}, and $\nu^{1/2}h^k\lesssim h^{k+1/2}$ gives~\eqref{eq:prerate}.
\end{proof}

\begin{myremark}[Interpretation and limitations]\label{rem:scope}
These are nodal $\ell^\infty(L^2)$ estimates. A continuous-in-time $L^\infty(L^2)$ assertion additionally requires a specified time reconstruction and a bound for its interpolation error. For $k=1$ and $k=2$, the spatial orders guaranteed by~\eqref{eq:prerate} are $3/2$ and $5/2$, respectively. 
For $k\ge3$, the term $h^{3k/4+1}$ may become dominant and therefore cannot be regarded as uniformly higher order.

The condition~\eqref{eq:errcon} is imposed to control the Gronwall factor in the error estimate and is distinct from the energy-stability condition. In particular, a fixed margin in the inequality $\tau\Gamma_n<1$ ensures a uniformly bounded Gronwall factor. When $\delta$ depends on $h$, the corresponding Smagorinsky solution and the regularity constants may also depend on $h$; uniform convergence therefore relies on the uniform regularity assumptions stated in Corollary~\ref{cor:preasymptotic}.

Finally, pressure robustness and the absence of explicit $\nu^{-1}$ factors in the error bound concern the viscosity dependence of the estimate, rather than uniform control of every unweighted velocity-gradient error. In~\eqref{eq:err}, the gradient contribution is weighted by the viscosity. Parameter-uniform uniqueness is addressed separately in Theorem~\ref{thm:unique_full_discrete} under stronger hypotheses.
\end{myremark}

\begin{myremark}[Adding the modelling error]
Under the additional hypotheses of Lemma~\ref{lemma::pertur_err}, the triangle inequality gives
\[
\max_n\|\u(t_n)-\w_h^n\|_{L^2}
\le\|\u-\w\|_{L^\infty(0,T;L^2)}
+\max_n\|\w(t_n)-\w_h^n\|_{L^2}.
\]
Thus a comparison with Navier--Stokes adds an $O(\delta)$ or, under the stronger regularity assumption, an $O(\delta^2)$ modelling term. For example, if $\delta\simeq h$, an $O(h^2)$ modelling bound may dominate the $O(h^{5/2})$ discretization bound for $k=2$. 
The manufactured solutions used in the experiments below are constructed for the corresponding Smagorinsky model problems, and the reported errors therefore measure the associated discretization errors rather than a combined modeling and discretization error relative to an unforced Navier--Stokes flow.

\end{myremark}

\section{Numerical experiments}\label{Sec:exp}
This section retains the numerical error values supplied with the manuscript. The tabulated orders are recomputed from those values. The manufactured-solution test is directly connected to the error estimates, whereas the periodic, open-boundary, curved-boundary, and inviscid examples require extensions of the setting analyzed in Sections~\ref{Sec:num}--\ref{Sec:err}. Their results are therefore used as computational illustrations rather than as direct validations of all theorem hypotheses.

The numerical experiments are implemented in NGSolve~\cite{Sch} with the NETGEN mesh generator~\cite{Sch1997}. Nonlinear systems are solved by Newton-type iterations, and the resulting linear systems are treated using the sparse direct solver UMFPACK~\cite{Davis}. Structured meshes are generated by subdividing uniform Cartesian cells into triangles, while unstructured meshes are generated using the parameter \texttt{maxh}. We distinguish this mesh-generation parameter from the mathematical mesh size $h=\max_T\operatorname{diam}(T)$ which is measured from the resulting mesh. Static condensation is employed to eliminate element-interior unknowns locally, thereby reducing the size of the globally coupled system. Depending on the implementation, the elementwise pressure means may remain globally coupled with the facet unknowns.

Unless otherwise stated, we set penalty parameters $\alpha_1=\alpha_2=\alpha_3=10k^2$, grid scale $\delta=h$ and consider $C_s=0$ and $C_s=0.1$. Here, $h$ refers to the actual maximum element diameter of the generated mesh, rather than the nominal mesh-generation parameter \texttt{maxh} or a cellwise mesh size $h_T$. 
Representative structured and unstructured meshes are associated with Fig.~\ref{fig:mesh}. On unstructured grids, actual element counts and measured maximum diameters should accompany the generator parameters.

\begin{figure}[htbp]
\centering
\PaperGraphic[width=0.4\textwidth,height=70mm
]{0.4\textwidth}{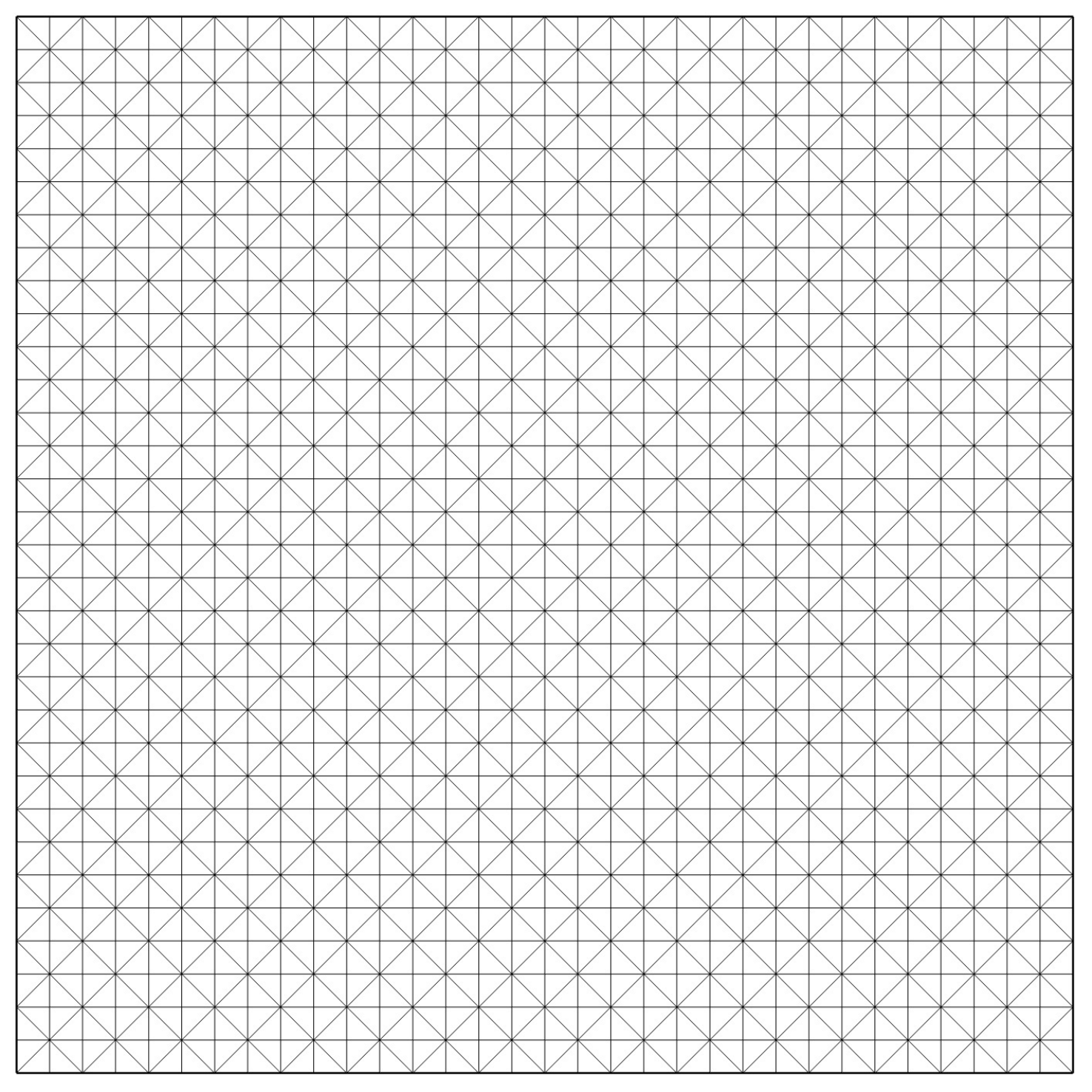}
\hspace{0.05\textwidth}
\PaperGraphic[width=0.4\textwidth,height=70mm
]{0.4\textwidth}{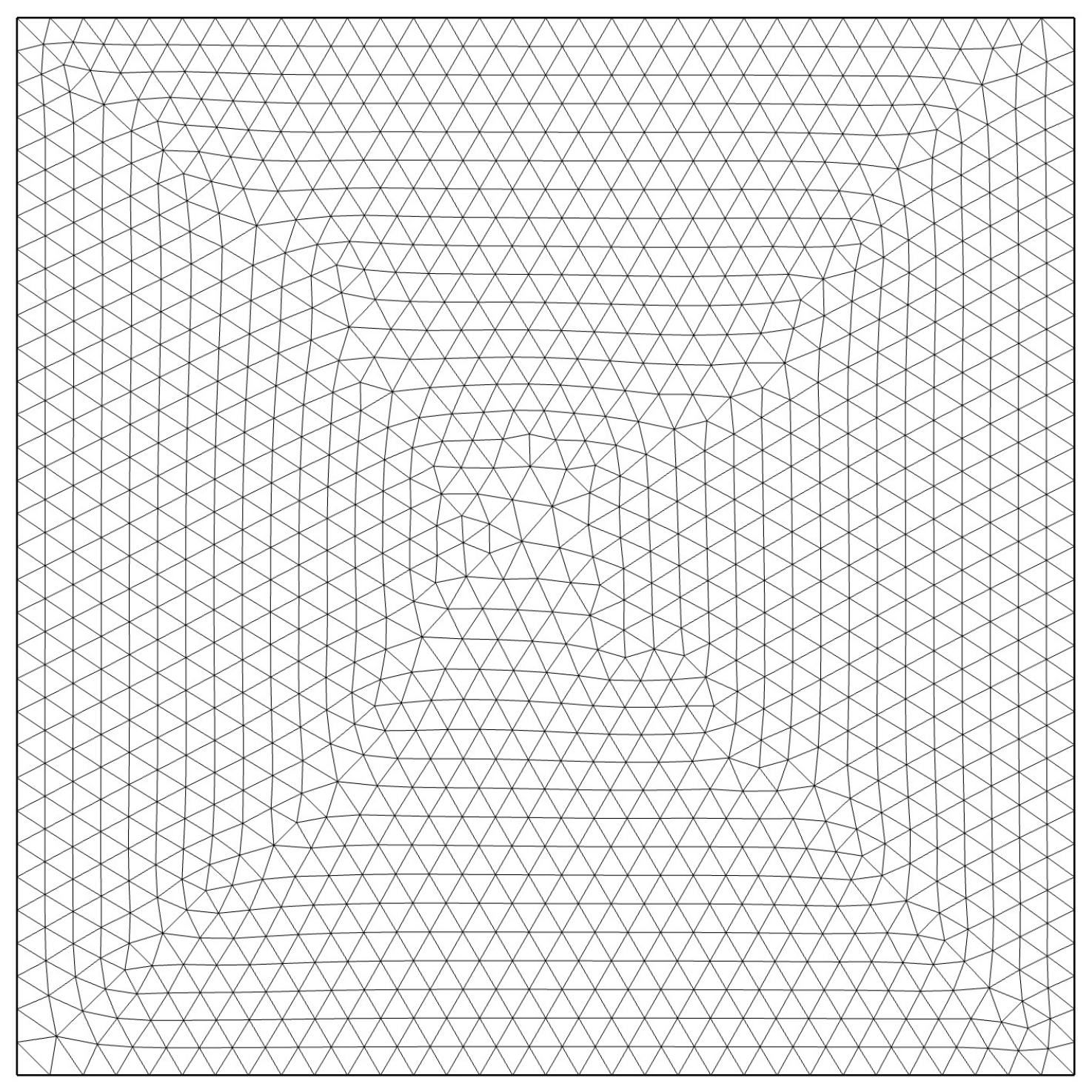}
\caption{\footnotesize Structured uniform triangulation ($N=64$, left) and unstructured 
Delaunay-type triangulation ($\texttt{maxh}=1/64$, right).}
\label{fig:mesh}
\end{figure}

\subsection{Two-dimensional Taylor--Green vortex}
Consider $\Omega=(0,2\pi)^2$ with periodic boundary conditions, zero body force $\bm f=\bm 0$ and initial velocity
\[
\w_0(x,y)=\begin{pmatrix}\sin x\cos y\\-\cos x\sin y\end{pmatrix}.
\]
The reported setup uses a structured mesh with 64 subdivisions per side, $\tau=0.01$, $T=20$, polynomial degrees $k=1,2$, and viscosities $\nu=1,10^{-1},10^{-2},10^{-4}$. The discrete kinetic energy is
\[
E_h^n=\frac12\int_\Omega|\w_h^n|^2\,d\x.
\]

The reported energy histories in Fig.~\ref{fig:energy} compare the effects of molecular and Smagorinsky dissipation for $k=1$ and $k=2$. In all cases, the kinetic energy decays monotonically, consistent with the discrete energy identity when periodic facets are paired so that their flux contributions cancel. As the viscosity decreases, the decay becomes progressively slower; for example, when $\nu=10^{-4}$, the energy remains close to its initial value over the simulation interval, whereas for $\nu=10^{0}$ it is rapidly dissipated. The Smagorinsky model with $C_s=0.1$ produces slightly stronger dissipation than the Navier--Stokes case with $C_s=0$, with the difference becoming more pronounced at lower viscosities. The results for $k=1$ and $k=2$ exhibit similar qualitative behavior, indicating that the observed energy-dissipation properties are insensitive to the polynomial degree.

These results demonstrate that the proposed method captures the expected energy-decay behavior across a broad range of viscosities and consistently reflects the additional dissipation introduced by the Smagorinsky model. The monotone decay should, however, be interpreted as an energy-stability diagnostic rather than an accuracy test, since temporal damping, upwind dissipation, and facet penalties may also contribute to the observed energy reduction.

\begin{figure}[htbp]
\centering
\PaperGraphic[width=0.48\textwidth]{0.48\textwidth}{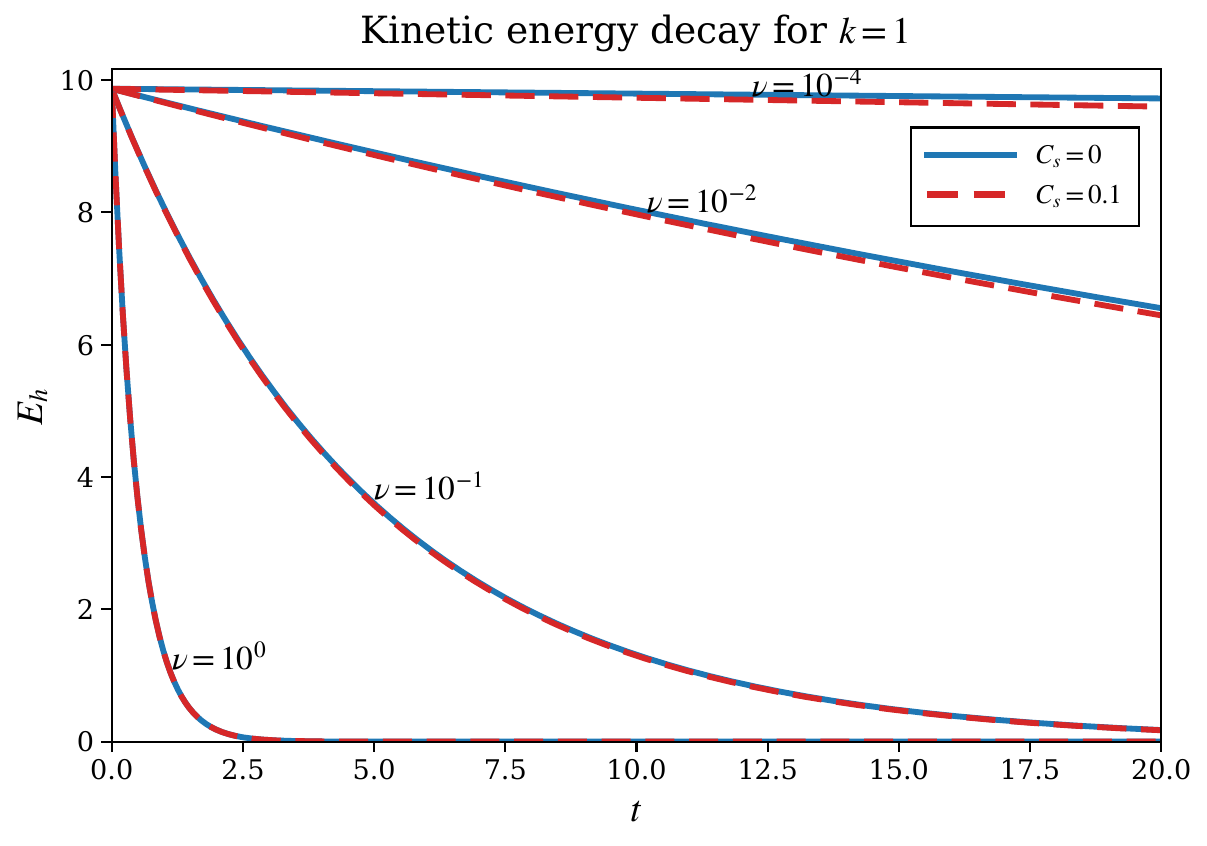}
\PaperGraphic[width=0.48\textwidth]{0.48\textwidth}{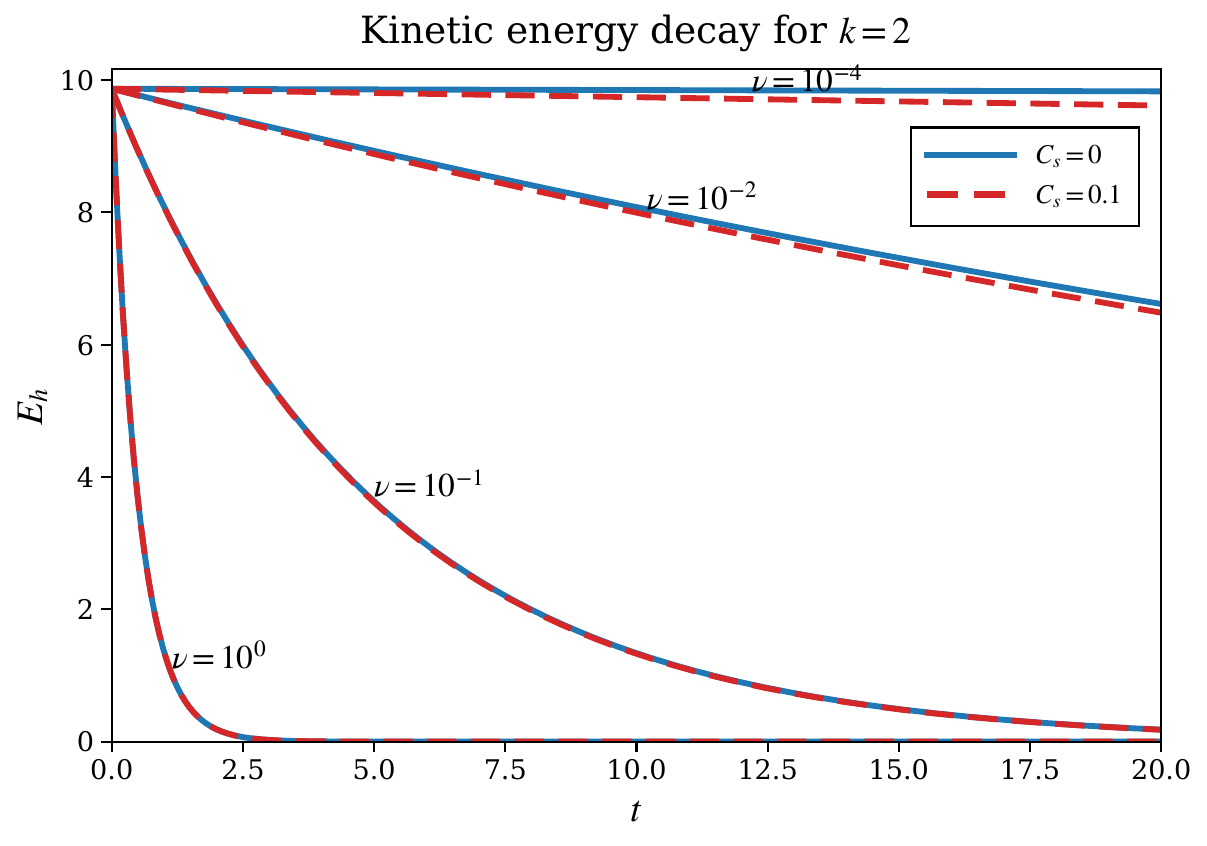}
\caption{\footnotesize Energy evolution of the 2D Taylor-Green vortex for polynomial degrees $k=1$ (left) and $k=2$ (right).}
\label{fig:energy}
\end{figure}

\subsection{Manufactured-solution convergence test}
On $\Omega=(0,1)^2$ and $0\le t\le1$, prescribe the smooth exact fields~\cite{Han2023}
\begin{equation}\label{eq:manufactured}
\begin{aligned}
\w(t,x,y)&=\frac{6+4\cos(4t)}{10}
\begin{pmatrix}
16y(1-y)(1-2y)\sin^2(\pi x)\\
-8\pi y^2(1-y)^2\sin(2\pi x)
\end{pmatrix},\\
r(t,x,y)&=\frac{6+4\cos(4t)}{10}\sin(\pi x)\cos(\pi y).
\end{aligned}
\end{equation}
The velocity is divergence-free and has zero trace on the boundary, and the pressure has zero mean. The force is obtained by substituting~\eqref{eq:manufactured} into the full gradient-based equation~\eqref{Smagorinsky_eq}, including the nonlinear eddy-viscosity term. Thus the manufactured force depends on $C_s$ and $\delta$. If $\delta$ changes with the mesh, the same exact fields solve a family of problems with correspondingly changed forcing.

The reported errors at $T=1$ are
\[
E_u=\|\w(T)-\w_h^N\|_{L^2(\Omega)},\qquad
E_{\nabla u}=\|\grad(\w(T)-\w_h^N)\|_{L^2(\Omega)},\qquad
E_{\rm div}=\|\grad\cdot\w_h^N\|_{L^2(\Omega)}.
\]

We denote the mesh-generator inputs by $h_{\rm nom}=1/4,1/8,1/16,1/32$. The nominal time step is prescribed as
\[
\tau_{\rm nom}=h_{\rm nom}^{3/2}\quad\text{for }k=1,
\qquad
\tau_{\rm nom}=h_{\rm nom}^{5/2}\quad\text{for }k=2.
\]
In the small-viscosity regime with $\delta=O(h)$, these choices correspond to the temporal orders $h^{3/2}$ and $h^{5/2}$ predicted by Corollary~\ref{cor:preasymptotic}; the additional Smagorinsky terms do not improve these guaranteed orders. For each spatial mesh, the actual number of time steps and time-step size are then defined by
\[
N_t=\left\lceil\frac{T}{\tau_{\rm nom}}\right\rceil,
\qquad
\tau=\frac{T}{N_t},
\]
so that a uniform time step is used and the final time is attained exactly.

Tables~\ref{tab:1} and~\ref{tab:2} report the complete set of computed errors. The experimental orders of convergence (EOCs) are calculated as
\[
\mathrm{EOC}
=\frac{\log\big(E(h_{{\rm nom},1})/E(h_{{\rm nom},2})\big)}
{\log\big(h_{{\rm nom},1}/h_{{\rm nom},2}\big)}=
\log_2\frac{E(1/16)}{E(1/32)}.
\]

For $C_s=0$, the results show that the $L^2$-velocity error converges with rates close to $1.5$ for $k=1$ and $2.5$ for $k=2$ as $\nu\to0$, in agreement with the theoretical prediction. When $C_s=0.1$, a slightly improved convergence behavior is observed. In particular, for $k=1$, the convergence rates exceed $1.5$ and approach $2$ in the convection-dominated regime, while for $k=2$, the rates remain close to $2.5$. 
The observed convergence rates reflect the finite-resolution behavior and should not be interpreted as evidence of a higher asymptotic convergence order or superconvergence.

In all cases, the divergence errors remain at the level of machine precision, consistent with the exactly divergence-free property of the discrete velocity. Together with the normal continuity across interelement facets, this property yields exact global mass conservation for the proposed method.

\begin{table}[htbp]
\centering
\fontsize{8.2}{10.2}\selectfont
\setlength{\tabcolsep}{2.1pt}
\renewcommand{\arraystretch}{1.16}
\caption{Numerical errors for $k=1$, $T=1$, and nominal time steps $\tau_{\rm nom}=h_{\rm nom}^{1.5}$. EOCs are $\log_2(E_{1/16}/E_{1/32})$, recalculated from the displayed rounded errors using the nominal mesh parameters. Divergence errors are at roundoff level and their EOCs are not reported.}

\begin{tabular}{c|c|ccc|ccc|ccc}
\hline
\multirow{2}{*}{$C_s$} & \multirow{2}{*}{$h_{\rm nom}$}
& \multicolumn{3}{c|}{$\nu=10^{0}$}
& \multicolumn{3}{c|}{$\nu=10^{-2}$}
& \multicolumn{3}{c}{$\nu=10^{-4}$} \\

&
& $E_u$ & $E_{\nabla u}$ & $E_{\rm div}$
& $E_u$ & $E_{\nabla u}$ & $E_{\rm div}$
& $E_u$ & $E_{\nabla u}$ & $E_{\rm div}$ \\

\hline

\multirow{5}{*}{$C_s=0$}
 & 1/4  & 6.42e-02 & 1.23e+00 & 8.94e-17 & 6.47e-02 & 1.32e+00 & 1.90e-16  & 1.05e-01 & 1.55e+00 & 1.71e-16  \\
 & 1/8  & 1.60e-02 & 6.51e-01 & 1.49e-16 & 2.60e-02 & 6.87e-01 & 1.64e-16  & 4.41e-02 & 9.05e-01 & 1.86e-16  \\
 & 1/16 & 3.12e-03 & 2.84e-01 & 1.38e-16 & 9.67e-03 & 2.95e-01 & 1.46e-16 & 1.34e-02 & 3.95e-01 & 1.39e-16 \\
 & 1/32 & 7.43e-04 & 1.43e-01 & 1.36e-16  & 3.59e-03 & 1.46e-01 & 1.39e-16  & 4.47e-03 & 1.75e-01 & 1.42e-16 \\
 & EOC & 2.070 & 0.990 & -- & 1.430 & 1.015 & -- & 1.584 & 1.174 & -- \\

\hline

\multirow{5}{*}{$C_s=0.1$}
 & 1/4  & 8.47e-02 & 1.26e+00 & 9.25e-17 & 2.05e-01 & 1.66e+00 & 5.54e-17  & 2.39e-01 & 1.84e+00 & 3.51e-17 \\
 & 1/8  & 1.92e-02 & 6.54e-01 & 1.59e-16 & 7.14e-02 & 8.00e-01 & 1.48e-16 & 9.94e-02 & 9.69e-01 & 1.12e-16 \\
 & 1/16 & 3.45e-03 & 2.84e-01 & 1.45e-16 & 7.07e-03 & 2.94e-01 & 1.42e-16 & 1.47e-02 & 3.56e-01 & 1.45e-16 \\
 & 1/32 & 7.82e-04 & 1.43e-01 & 1.38e-16 & 1.76e-03 & 1.46e-01 & 1.36e-16 & 2.60e-03 & 1.62e-01 & 1.41e-16 \\
 & EOC & 2.141 & 0.990 & -- & 2.006 & 1.010 & -- & 2.499 & 1.136 & -- \\

\hline
\hline

\multirow{2}{*}{$C_s$} & \multirow{2}{*}{$h_{\rm nom}$}
& \multicolumn{3}{c|}{$\nu=10^{-6}$}
& \multicolumn{3}{c|}{$\nu=10^{-8}$}
& \multicolumn{3}{c}{$\nu=10^{-10}$} \\

&
& $E_u$ & $E_{\nabla u}$ & $E_{\rm div}$
& $E_u$ & $E_{\nabla u}$ & $E_{\rm div}$
& $E_u$ & $E_{\nabla u}$ & $E_{\rm div}$ \\

\hline

\multirow{5}{*}{$C_s=0$}
 & 1/4  & 1.06e-01 & 1.57e+00 & 1.72e-16  & 1.06e-01 & 1.57e+00 & 3.20e-16  & 1.06e-01 & 1.57e+00 & 1.55e-16 \\
 & 1/8  & 4.61e-02 & 9.50e-01 & 1.62e-16  & 4.61e-02 & 9.50e-01 & 1.74e-16 & 4.61e-02 & 9.50e-01 & 1.59e-16 \\
 & 1/16 & 1.54e-02 & 5.68e-01 & 1.32e-16  & 1.55e-02 & 5.69e-01 & 1.41e-16 & 1.55e-02 & 5.69e-01 & 1.36e-16 \\
 & 1/32 & 5.51e-03 & 3.86e-01 & 1.44e-16 & 5.54e-03 & 3.92e-01 & 1.38e-16 & 5.54e-03 & 3.92e-01 & 1.32e-16 \\
 & EOC & 1.483 & 0.557 & -- & 1.484 & 0.538 & -- & 1.484 & 0.538 & -- \\

\hline

\multirow{5}{*}{$C_s=0.1$}
 & 1/4  & 2.39e-01 & 1.84e+00 & 3.09e-17  & 2.39e-01 & 1.84e+00 & 2.95e-17 & 2.39e-01 & 1.84e+00 & 4.19e-17 \\
 & 1/8  & 9.97e-02 & 9.72e-01 & 1.20e-16 & 9.95e-02 & 9.68e-01 & 1.09e-16 & 9.97e-02 & 9.72e-01 & 1.34e-16 \\
 & 1/16 & 1.48e-02 & 3.60e-01 & 1.26e-16 & 1.49e-02 & 3.62e-01 & 1.38e-16 & 1.48e-02 & 3.60e-01 & 1.32e-16 \\
 & 1/32 & 2.67e-03 & 1.66e-01 & 1.31e-16 & 2.72e-03 & 1.68e-01 & 1.37e-16 & 2.67e-03 & 1.66e-01 & 1.35e-16 \\
 & EOC & 2.471 & 1.117 & -- & 2.454 & 1.108 & -- & 2.471 & 1.117 & -- \\

\hline
\end{tabular}
\label{tab:1}
\end{table}

\begin{table}[htbp]
\centering
\fontsize{8.2}{10.2}\selectfont
\setlength{\tabcolsep}{2.1pt}
\renewcommand{\arraystretch}{1.16}
\caption{Numerical errors for $k=2$, $T=1$, and nominal time steps $\tau_{\rm nom}=h_{\rm nom}^{2.5}$. EOCs are $\log_2(E_{1/16}/E_{1/32})$, recalculated from the displayed rounded errors using the nominal mesh parameters. Divergence errors are at roundoff level and their EOCs are not reported.}

\begin{tabular}{c|c|ccc|ccc|ccc}
\hline
\multirow{2}{*}{$C_s$} & \multirow{2}{*}{$h_{\rm nom}$}
& \multicolumn{3}{c|}{$\nu=10^{0}$}
& \multicolumn{3}{c|}{$\nu=10^{-2}$}
& \multicolumn{3}{c}{$\nu=10^{-4}$} \\

&
& $E_u$ & $E_{\nabla u}$ & $E_{\rm div}$
& $E_u$ & $E_{\nabla u}$ & $E_{\rm div}$
& $E_u$ & $E_{\nabla u}$ & $E_{\rm div}$ \\

\hline

\multirow{5}{*}{$C_s=0$}
 & 1/4  & 1.39e-02 & 3.54e-01 & 1.88e-16  & 2.06e-02 & 3.90e-01 & 2.10e-16 & 3.62e-02 & 6.47e-01 & 2.32e-16 \\
 & 1/8  & 1.23e-03 & 7.68e-02 & 1.94e-16 & 3.90e-03 & 8.39e-02 & 1.84e-16 & 4.53e-03 & 1.06e-01 & 1.88e-16 \\
 & 1/16 & 1.16e-04 & 1.65e-02 & 1.74e-16 & 7.01e-04 & 1.76e-02 & 1.77e-16 & 7.28e-04 & 2.18e-02 & 1.68e-16 \\
 & 1/32 & 1.45e-05 & 4.03e-03 & 1.67e-16 & 1.24e-04 & 4.19e-03 & 1.70e-16 & 1.24e-04 & 4.98e-03 & 1.76e-16 \\
 & EOC & 3.000 & 2.034 & -- & 2.499 & 2.071 & -- & 2.554 & 2.130 & -- \\

\hline

\multirow{5}{*}{$C_s=0.1$}
 & 1/4  & 1.69e-02 & 3.74e-01 & 1.97e-16    & 5.04e-02 & 8.48e-01 & 1.78e-16  & 7.37e-02 & 1.14e+00 & 1.62e-16 \\
 & 1/8  & 1.34e-03 & 7.80e-02 & 1.89e-16 & 5.12e-03 & 1.68e-01 & 2.01e-16 & 1.14e-02 & 3.55e-01 & 1.85e-16 \\
 & 1/16 & 1.19e-04 & 1.66e-02 & 1.74e-16  & 6.94e-04 & 2.61e-02 & 1.71e-16 & 1.16e-03 & 1.01e-01 & 1.77e-16 \\
 & 1/32 & 1.46e-05 & 4.03e-03 & 1.65e-16 & 1.23e-04 & 4.75e-03 & 1.71e-16 & 1.56e-04 & 2.64e-02 & 1.73e-16 \\
 & EOC & 3.027 & 2.042 & -- & 2.496 & 2.458 & -- & 2.895 & 1.936 & -- \\

\hline
\hline

\multirow{2}{*}{$C_s$} & \multirow{2}{*}{$h_{\rm nom}$}
& \multicolumn{3}{c|}{$\nu=10^{-6}$}
& \multicolumn{3}{c|}{$\nu=10^{-8}$}
& \multicolumn{3}{c}{$\nu=10^{-10}$} \\

&
& $E_u$ & $E_{\nabla u}$ & $E_{\rm div}$
& $E_u$ & $E_{\nabla u}$ & $E_{\rm div}$
& $E_u$ & $E_{\nabla u}$ & $E_{\rm div}$ \\

\hline

\multirow{5}{*}{$C_s=0$}
 & 1/4  & 3.94e-02 & 7.23e-01 & 2.22e-16 & 3.94e-02 & 7.24e-01 & 2.17e-16 & 3.94e-02 & 7.24e-01 & 2.13e-16 \\
 & 1/8  & 4.96e-03 & 1.26e-01 & 1.79e-16 & 4.97e-03 & 1.27e-01 & 1.76e-16 & 4.97e-03 & 1.27e-01 & 1.88e-16 \\
 & 1/16 & 7.69e-04 & 2.41e-02 & 1.72e-16 & 7.70e-04 & 2.44e-02 & 1.71e-16 & 7.70e-04 & 2.44e-02 & 1.70e-16 \\
 & 1/32 & 1.33e-04 & 5.35e-03 & 1.67e-16 & 1.34e-04 & 5.45e-03 & 1.66e-16 & 1.34e-04 & 5.45e-03 & 1.66e-16 \\
 & EOC & 2.532 & 2.171 & -- & 2.523 & 2.163 & -- & 2.523 & 2.163 & -- \\
\hline

\multirow{5}{*}{$C_s=0.1$}
 & 1/4  & 7.42e-02 & 1.15e+00 & 1.63e-16 & 7.42e-02 & 1.15e+00 & 1.88e-16 & 7.42e-02 & 1.15e+00 & 1.80e-16 \\
 & 1/8  & 1.18e-02 & 3.70e-01 & 1.73e-16 & 1.18e-02 & 3.70e-01 & 1.69e-16 & 1.18e-02 & 3.70e-01 & 1.73e-16 \\
 & 1/16 & 1.38e-03 & 1.20e-01 & 1.71e-16 & 1.38e-03 & 1.20e-01 & 1.72e-16 & 1.38e-03 & 1.20e-01 & 1.75e-16 \\
 & 1/32 & 2.28e-04 & 4.56e-02 & 1.68e-16 & 2.30e-04 & 4.62e-02 & 1.73e-16 & 2.30e-04 & 4.62e-02 & 1.75e-16 \\
 & EOC & 2.598 & 1.396 & -- & 2.585 & 1.377 & -- & 2.585 & 1.377 & -- \\
\hline
\end{tabular}
\label{tab:2}
\end{table}

To further examine the behavior of the proposed exactly divergence-free, $H(\mathrm{div})$-conforming HDG method in a convection-dominated regime, we compare it with the classical Taylor--Hood discretization based on the continuous $P_2/P_1$ velocity--pressure pair. We set $\nu=10^{-4}$, $T=1$, and $\tau=10^{-3}$. Figure~\ref{fig:TH-HDG} presents the computed velocity magnitude for the unstabilized Taylor--Hood discretization at nominal mesh resolutions $1/32$, $1/64$, and $1/128$, together with the stabilized Taylor--Hood result for $C_s=0.1$ and the proposed HDG results for $C_s=0$ and $C_s=0.1$ at the nominal resolution $1/32$.

On the coarse mesh with nominal resolution $1/32$, the unstabilized Taylor--Hood discretization develops pronounced oscillations in the computed velocity field. Adding the eddy-viscosity stabilization with $C_s=0.1$ substantially suppresses these oscillations, although visible nonphysical features remain. In contrast, the proposed HDG method produces smooth velocity fields for both $C_s=0$ and $C_s=0.1$ on the same nominal mesh, while preserving the main flow structures.

For the unstabilized Taylor--Hood discretization, the oscillatory features progressively diminish with mesh refinement. A marked reduction is observed at the nominal resolution $1/64$, and the computed solution at $1/128$ is essentially free of visible oscillations. This behavior highlights the pronounced mesh sensitivity of the Taylor--Hood approximation in the considered convection-dominated regime. By comparison, the proposed HDG formulation maintains a stable and non-oscillatory approximation on the coarse mesh considered here, for both the Navier--Stokes case ($C_s=0$) and the Smagorinsky model ($C_s=0.1$).

\begin{figure}[htbp]
\centering
\PaperGraphic[width=0.24\textwidth]{0.24\textwidth}{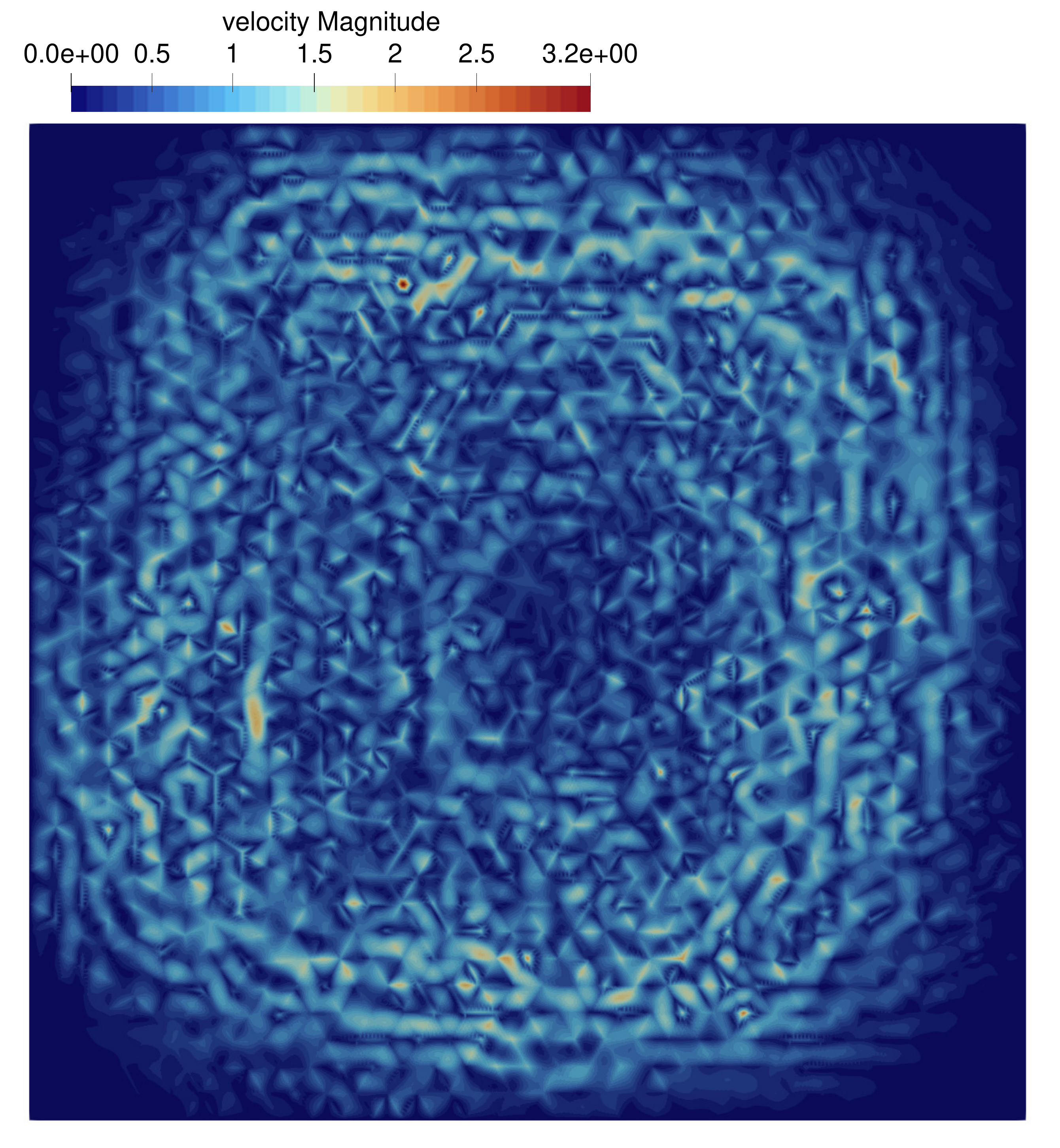}
\PaperGraphic[width=0.24\textwidth]{0.24\textwidth}{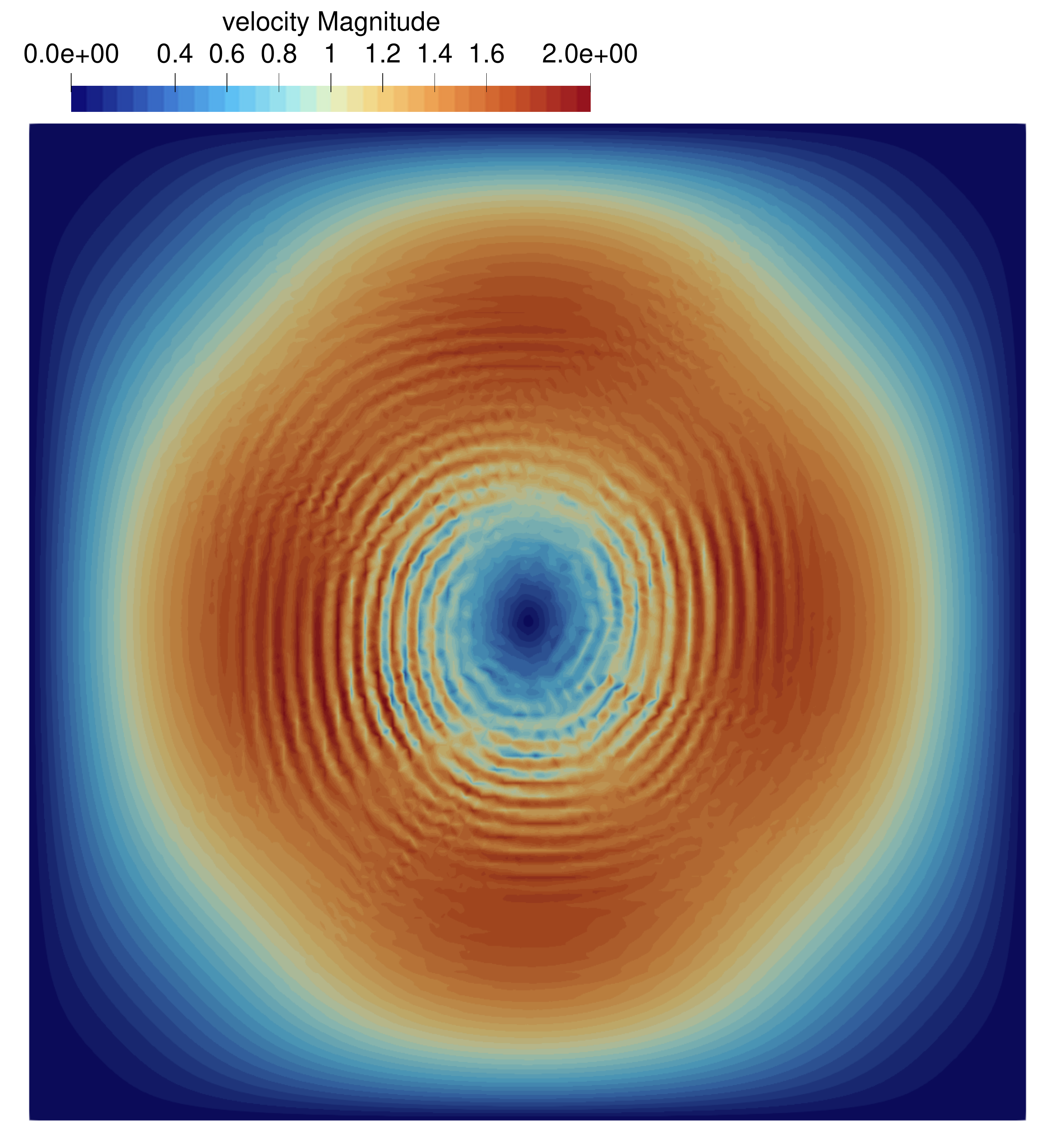}
\PaperGraphic[width=0.24\textwidth]{0.24\textwidth}{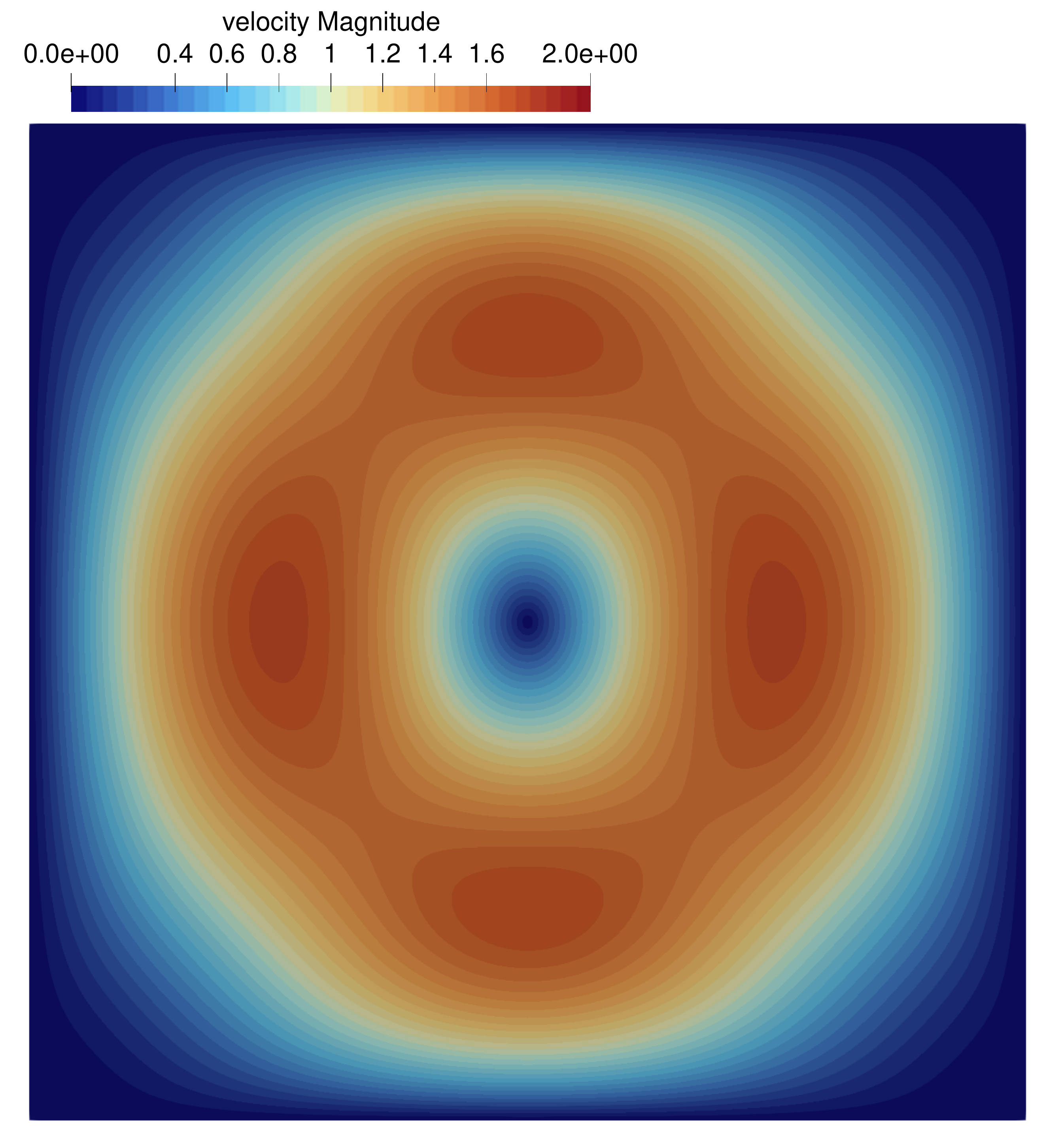}
\\
\PaperGraphic[width=0.24\textwidth]{0.24\textwidth}{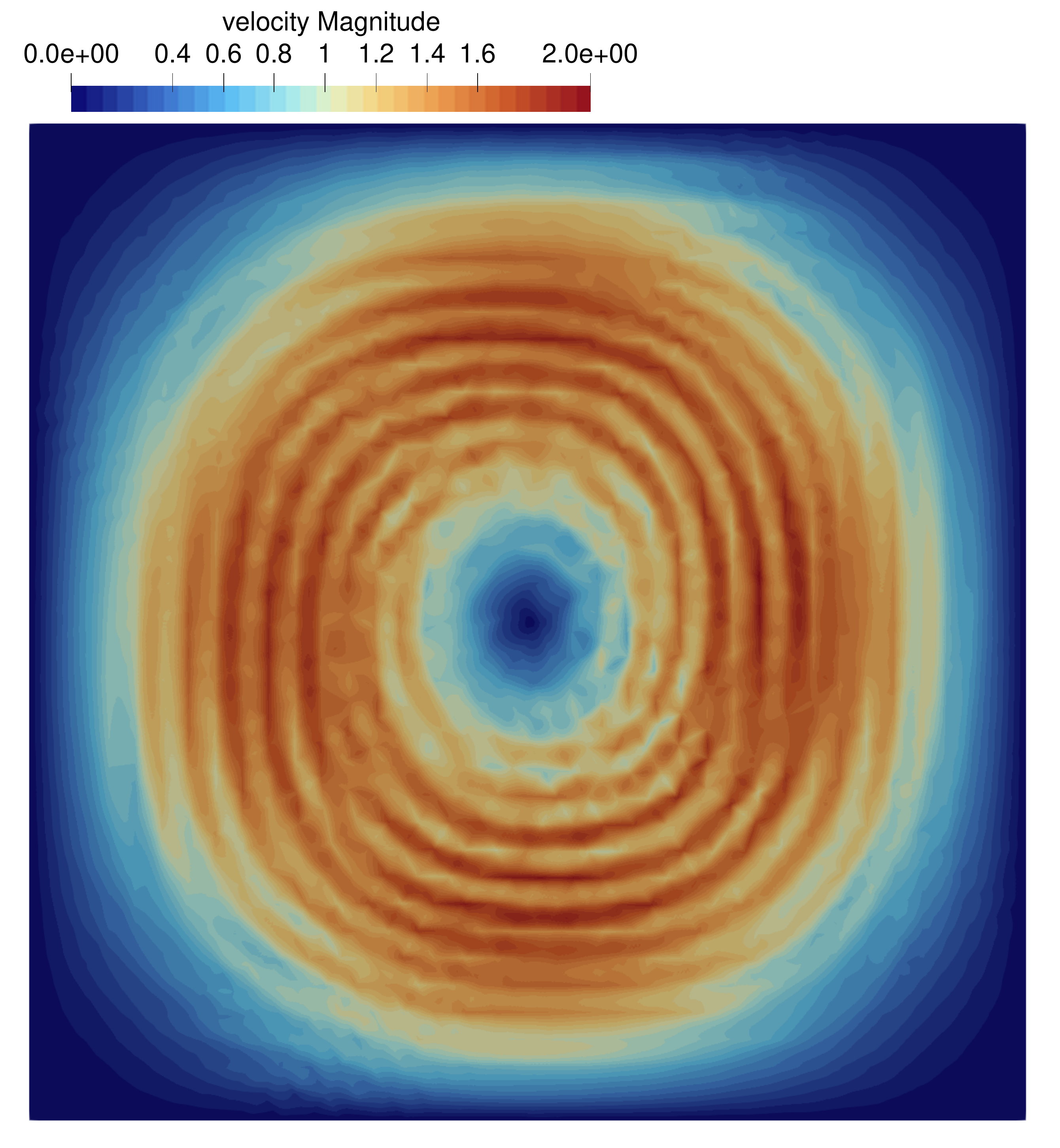}
\PaperGraphic[width=0.24\textwidth]{0.24\textwidth}{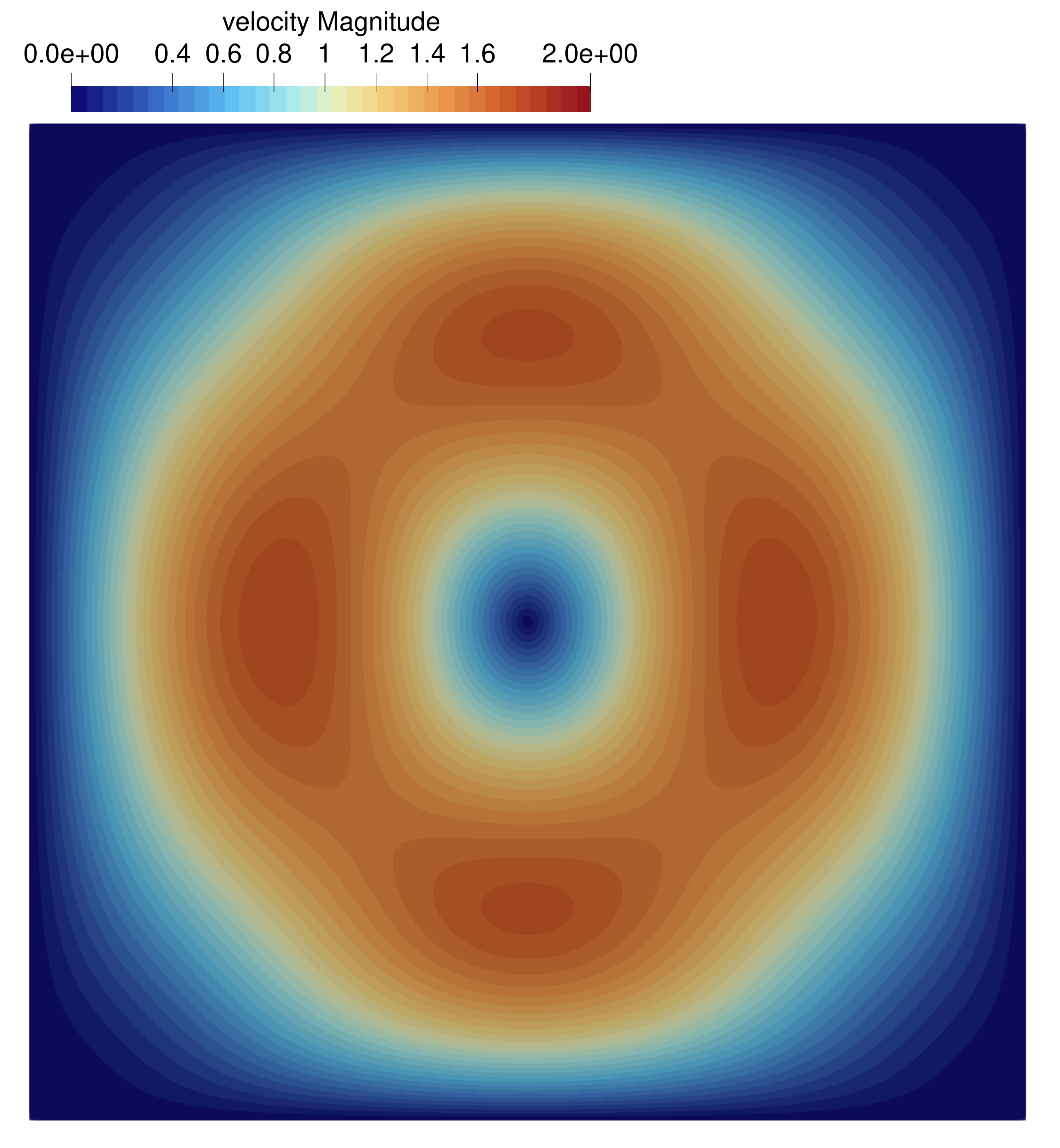}
\PaperGraphic[width=0.24\textwidth]{0.24\textwidth}{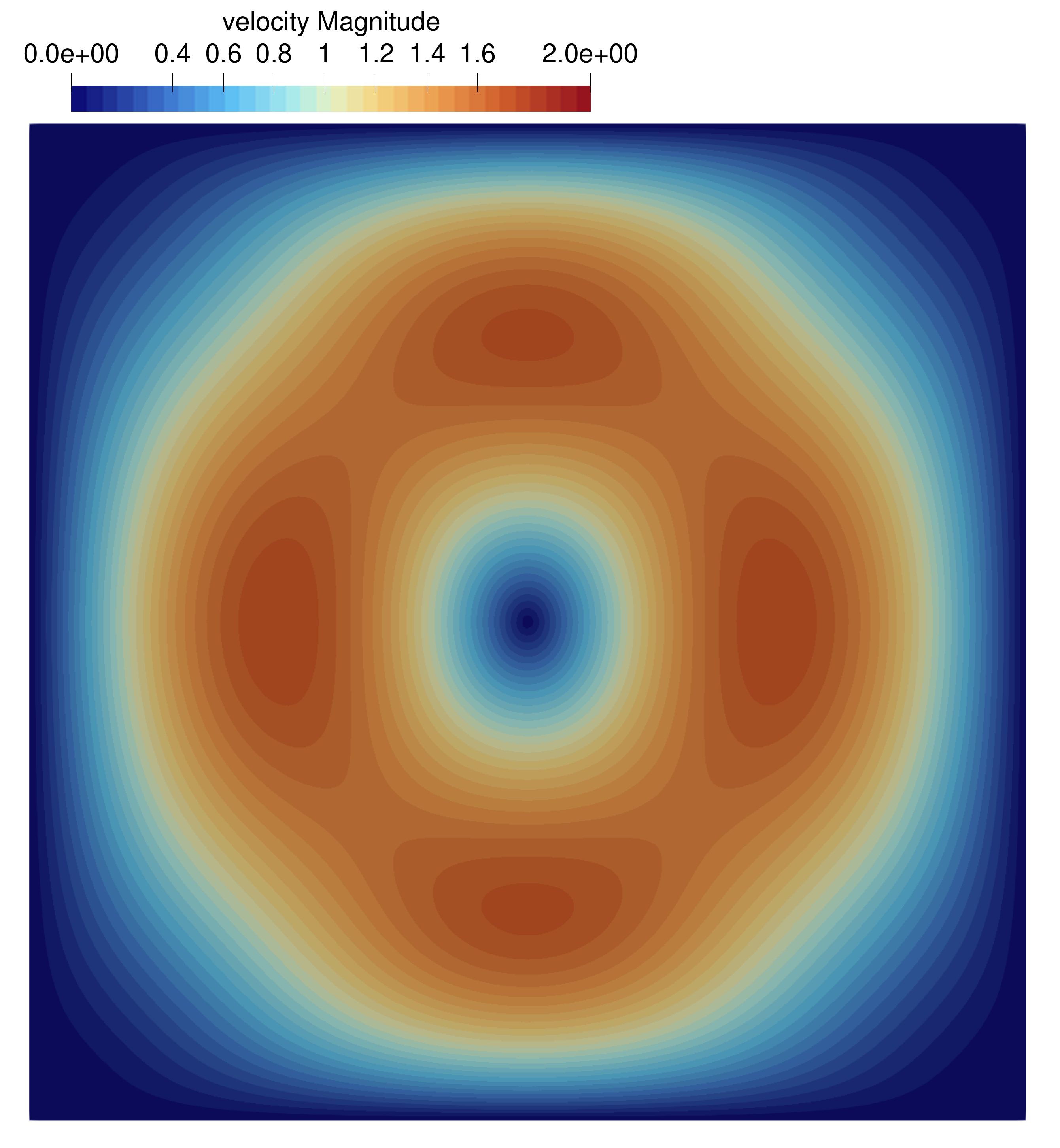}
\caption{\footnotesize
Velocity magnitude for $\nu = 10^{-4}$ and $k=2$.
Top row (left to right): Taylor-Hood solutions with $C_s=0$ on meshes
$h=1/32$, $1/64$, and $1/128$.
Bottom row (left to right): Taylor-Hood ($C_s=0.1$), HDG ($C_s=0$), and HDG ($C_s=0.1$) solutions on the mesh with $h=1/32$.
}
\label{fig:TH-HDG}
\end{figure}

\subsection{Vortex shedding behind a cylinder}
The third example investigates the influence of dissipation on the formation of the von K\'arm\'an vortex street behind a circular cylinder. The computational domain is $(-\tfrac{1}{2}, 2)\times(-\tfrac{1}{2},\,\tfrac{1}{2})$, and contains a circular cylinder centered at the origin with radius $r=0.1$. The boundary conditions are specified as follows:
\begin{itemize}
\item $\w =\bm 0$ on the cylinder boundary and on the walls $y=\pm \tfrac{1}{2}$ (no-slip condition);
\item homogeneous Neumann condition on the outflow boundary $x=2$;
\item prescribed parabolic inflow profile at $x=-\tfrac{1}{2}$, $\mathbf u = (\tfrac{3}{2} - 6y^2,\,0).$
\end{itemize}

The initial condition is taken as the stationary Stokes solution associated with the prescribed boundary conditions. The computational mesh employed in the simulations is generated using NETGEN \cite{Sch1997} and consists of 4295 triangular elements with local refinement around the cylinder and in the wake region to accurately resolve the shear layers and vortex formation (see Fig.~\ref{fig:test2:mesh}). The time-step size is chosen as $\tau = 0.01$, the polynomial degree is set to $k=3$, and the viscosity is set to $\nu = 10^{-4}$, corresponding to a convection-dominated flow regime. To investigate the influence of the additional dissipation introduced by the stabilization, two cases are considered: the Navier-Stokes equations ($C_s=0$) and the Smagorinsky model ($C_s=0.1$).

\begin{figure}[htbp]
\centering
\includegraphics[width=0.8\textwidth]{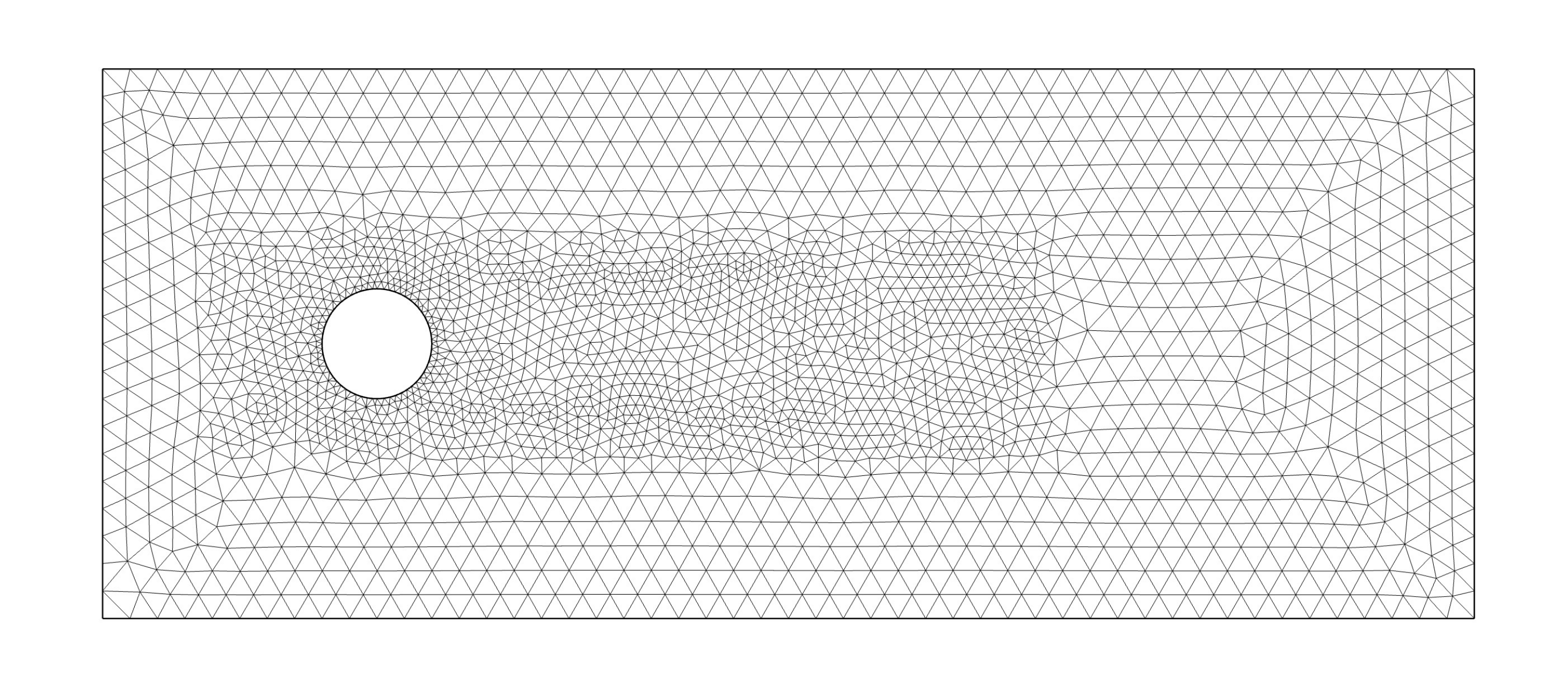}
\caption{\footnotesize Computational mesh used for the vortex shedding computation. 
The mesh consists of 4295 elements with local refinement around the cylinder and in the wake region.}
\label{fig:test2:mesh}
\end{figure}

Fig.~\ref{fig:test2:velocity} presents the velocity magnitude contours at two representative time instances, $t=8$ and $t=10$. In both cases, the characteristic vortex shedding pattern behind the cylinder is clearly observed. Alternating vortices are generated in the shear layers separating from the cylinder surface and convected downstream, forming a typical von K\'arm\'an vortex street. Comparing the two rows of Fig.~\ref{fig:test2:velocity}, the overall large-scale flow structures remain similar in both simulations.

\begin{figure}[htbp]
\centering
\includegraphics[width=0.48\textwidth]{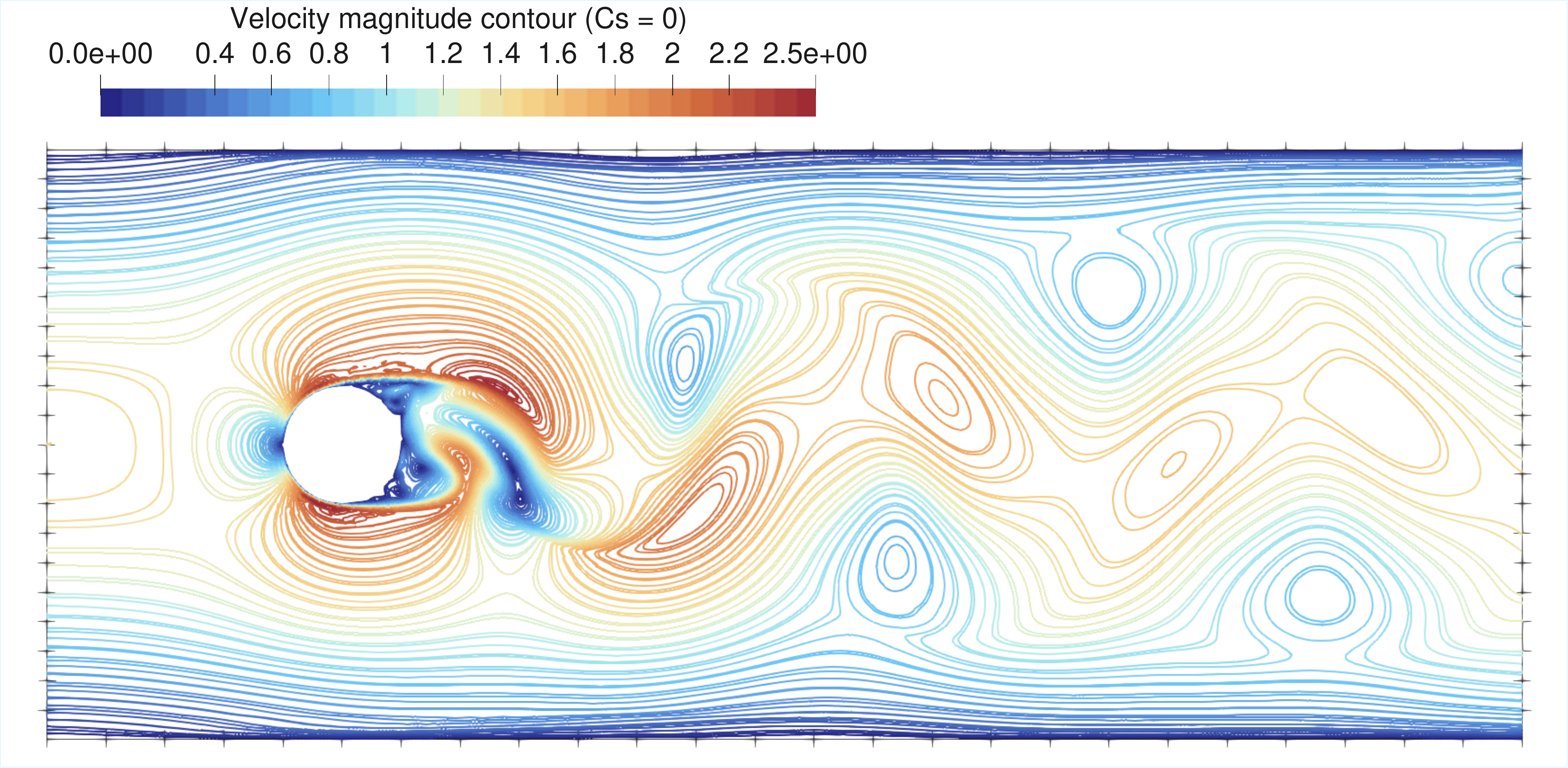}
\includegraphics[width=0.48\textwidth]{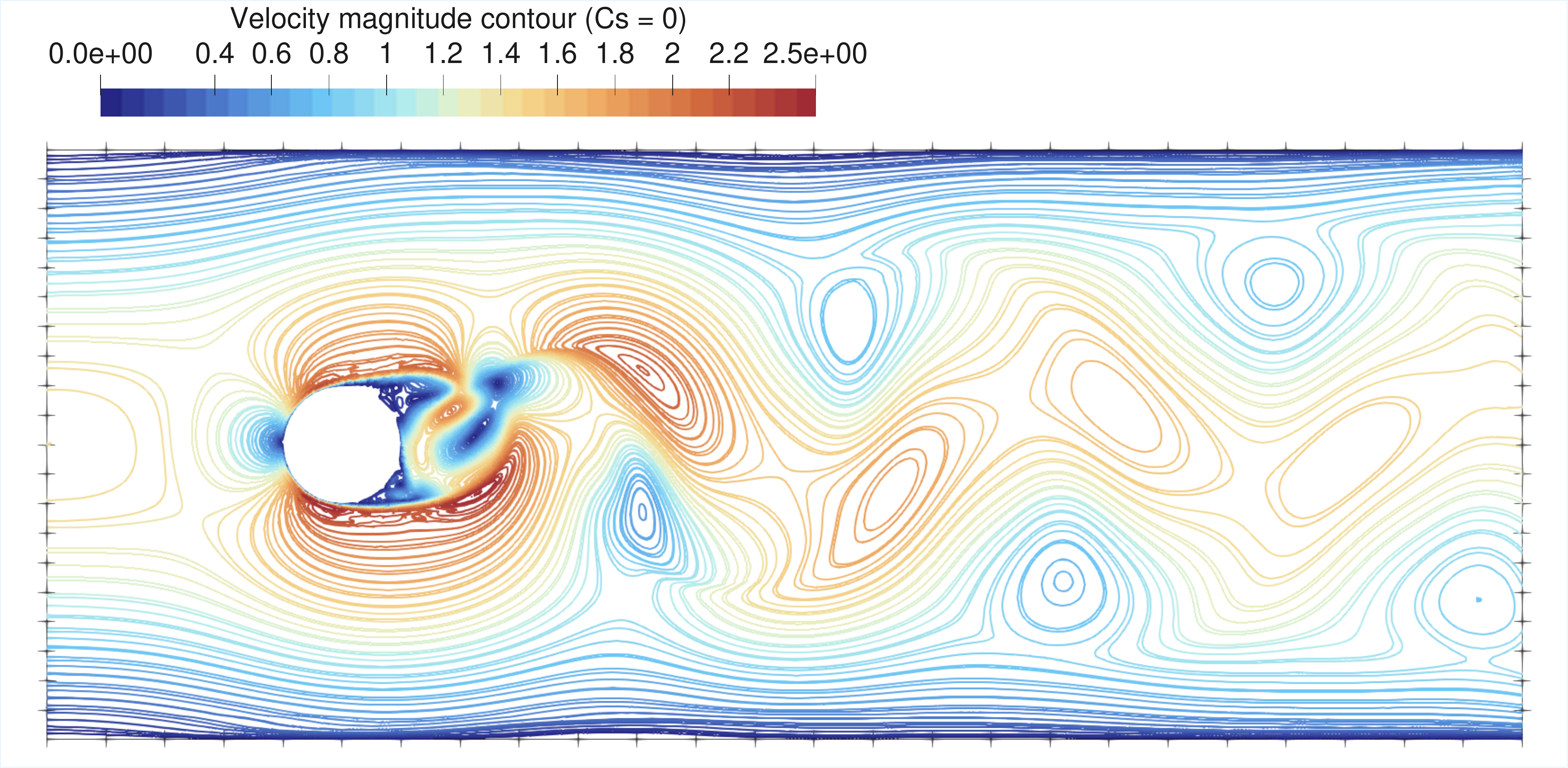}
\vspace{0.3cm}
\includegraphics[width=0.48\textwidth]{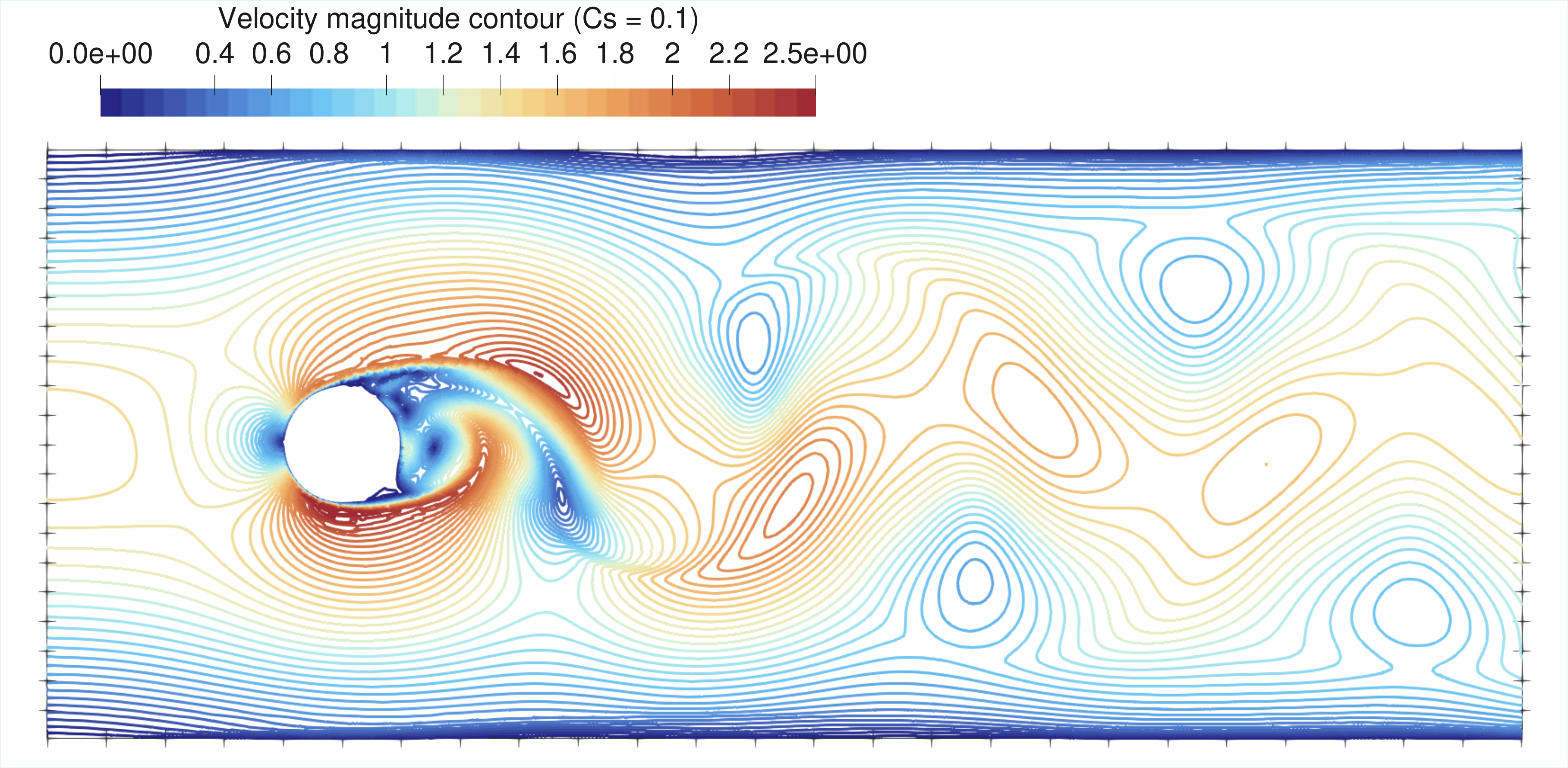}
\includegraphics[width=0.48\textwidth]{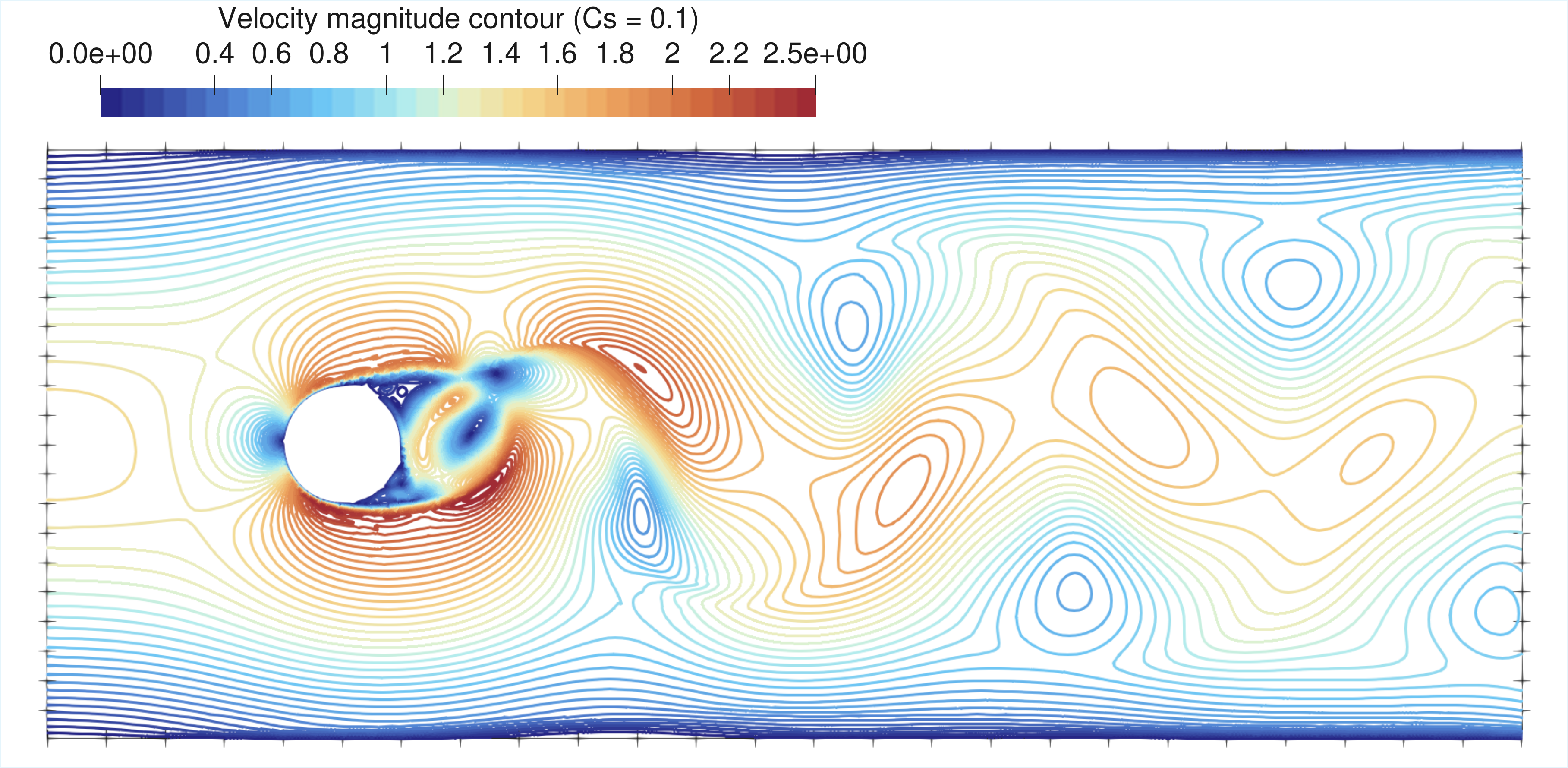}
\caption{\footnotesize Velocity magnitude contours for $\nu=10^{-4}$. 
Top row: $C_s=0$; bottom row: $C_s=0.1$. 
Columns correspond to $t=8$ and $t=10$.}
\label{fig:test2:velocity}
\end{figure}

The quantities of interest are the lift and drag forces acting on the cylinder,
\[
(F_D, F_L)= \int_S \left(-r_h\bm I + \nu\nabla {\w}_h\right)\mathbf{n}\,ds,
\]
where $S$ denotes the boundary between the fluid domain and the cylinder. The quantitative influence of the stabilization is further illustrated in Fig.~\ref{fig:test2:forces}, which presents the time histories of the drag and lift forces acting on the cylinder. After a short transient phase, the flow reaches a statistically periodic regime associated with vortex shedding. Both models exhibit nearly identical shedding frequencies, indicating that the additional dissipation does not significantly affect the dominant wake dynamics.

However, noticeable differences in the force amplitudes are observed in Fig.~\ref{fig:test2:forces}. 
In particular, for the present configuration, the Smagorinsky solution ($C_s=0.1$) exhibits slightly larger lift oscillations and a higher mean drag compared to the Navier-Stokes solution ($C_s=0$).  At the same time, the dominant periodic behavior of the flow remains comparable in both cases.

\begin{figure}[htbp]
\centering
\includegraphics[width=0.9\textwidth]{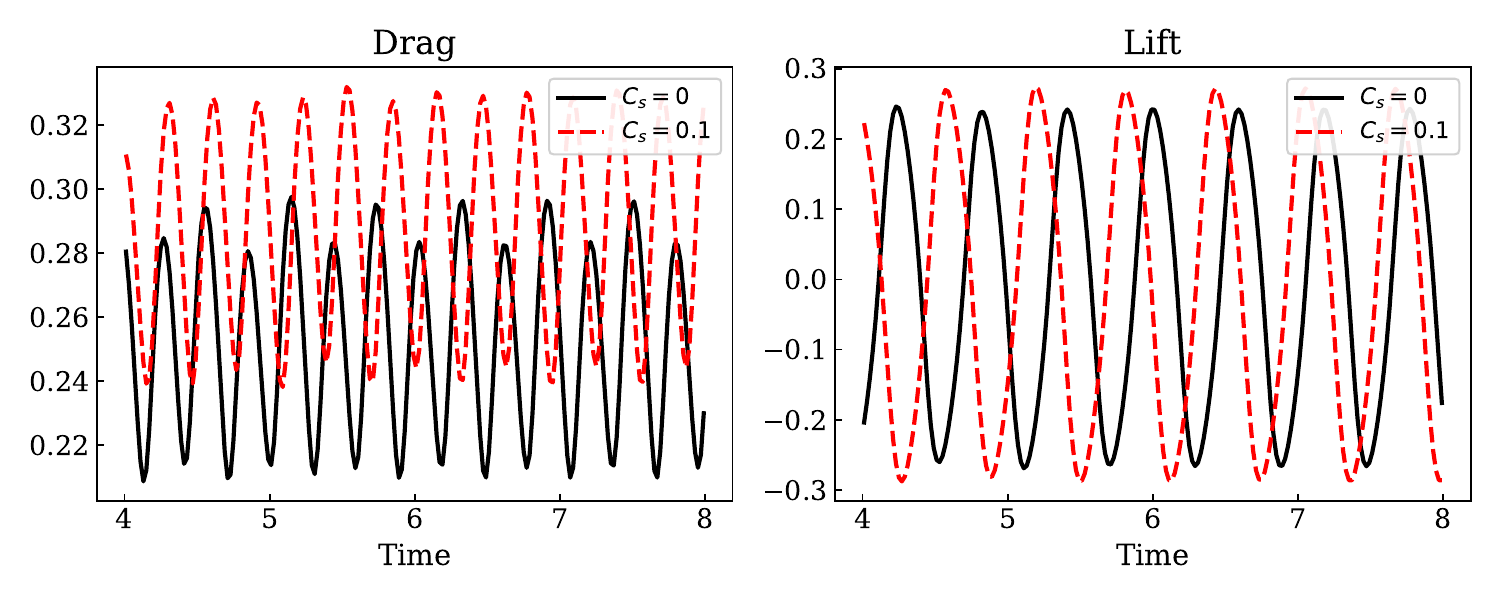}
\caption{\footnotesize Evolution of the drag (left) and lift (right) forces on the cylinder for $C_s=0$ and $C_s=0.1$.}
\label{fig:test2:forces}
\end{figure}

These observations indicate that the additional dissipation introduced by the Smagorinsky model affects the force amplitudes, while the overall vortex-shedding dynamics are largely preserved.

\subsection{Double shear layers}
We consider the double shear-layer problem~\cite{Bell}. In the absence of molecular viscosity ($\nu=0$), \eqref{Smagorinsky_eq} reduces to the Euler equations when $C_s=0$, whereas $C_s>0$ introduces the Smagorinsky subgrid-scale regularization.
The computational domain is $\Omega = (0,2\pi)^2$, with the following initial conditions:
\begin{align*}
u_y(x,y,0) = \kappa \sin(x), \qquad
u_x(x,y,0) =
\begin{cases}
\tanh\!\left(\frac{y-\pi/2}{\rho}\right), & y \le \pi, \\[6pt]
\tanh\!\left(\frac{3\pi/2-y}{\rho}\right), & y > \pi .
\end{cases}
\end{align*}
These initial conditions produce two horizontal shear layers perturbed by a small vertical velocity component. We take $\rho=\pi/15$ and $\kappa=0.05$, and impose periodic boundary conditions.

To illustrate the resolution of flow structures, we plot $99$ equally spaced contours of the discrete vorticity
\[
\omega_h := \mathrm{curl}(\w_h)
= \partial_{x_1} w_2 - \partial_{x_2} w_1,
\]
in the range $[-4.9,\,4.9]$. 
The domain $(0,2\pi)^2$ is discretized by structured uniform triangulations with $N=64$ and $N=128$ subdivisions per side.
In all simulations, the time step is fixed as $\tau = 0.01$, polynomial degrees $k=2$ and the vorticity contours are reported at $t=6,\,8,\,10,$ and $12$.

Figure~\ref{fig:test3:1} shows the vorticity contours obtained without additional regularization ($C_s=0$), corresponding to the inviscid Euler equations. As the flow evolves, the shear layers roll up and form large coherent vortices. At later times, nonlinear advection generates increasingly thin vorticity filaments. On the coarse mesh ($N=64$), these fine-scale structures become under-resolved and give rise to oscillatory patterns in the vorticity field. The refined mesh ($N=128$) resolves these structures more accurately and produces a smoother vorticity distribution.

Figure~\ref{fig:test3:2} presents the results obtained by adding a Smagorinsky viscosity term with $C_s = 0.1$. The mesh resolutions and output times are identical to those used in Figure~\ref{fig:test3:1}. Compared with the pure Euler simulations, the additional turbulent viscosity effectively suppresses under-resolved small-scale oscillations, particularly on the coarse mesh, while the large-scale vortex structures remain essentially unchanged. This suggests that the added turbulent viscosity introduces controlled dissipation and acts as a subgrid-scale regularization mechanism that stabilizes the numerical solution while preserving the dominant flow dynamics. 

The qualitative agreement between the coarse and refined meshes indicates that the dominant flow dynamics are well captured by the proposed $H(\rm{div})$-conforming HDG method~\eqref{eq:fulld}.

\begin{figure}[htbp]
\centering
\includegraphics[width=0.24\textwidth]{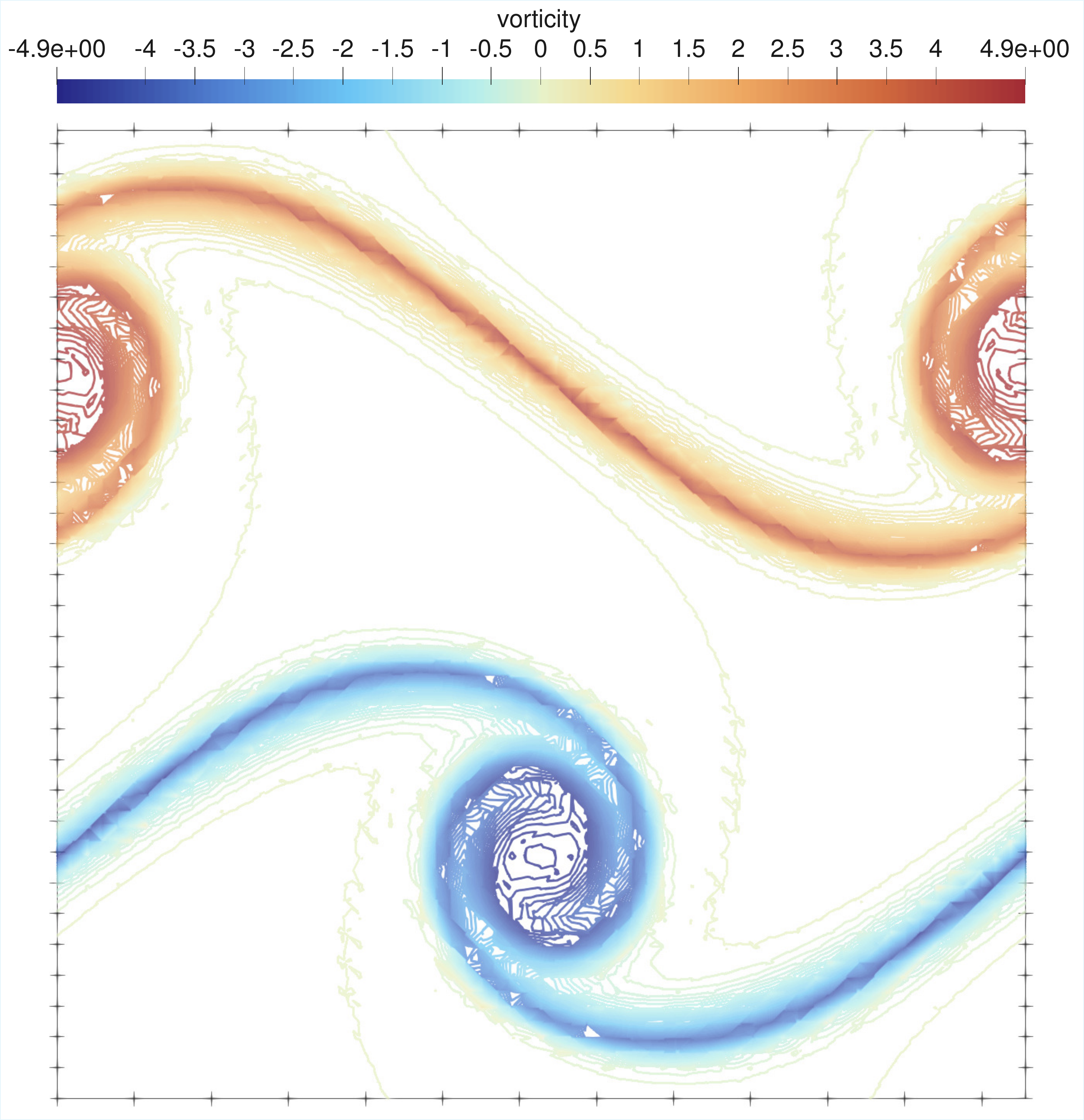}%
\includegraphics[width=0.24\textwidth]{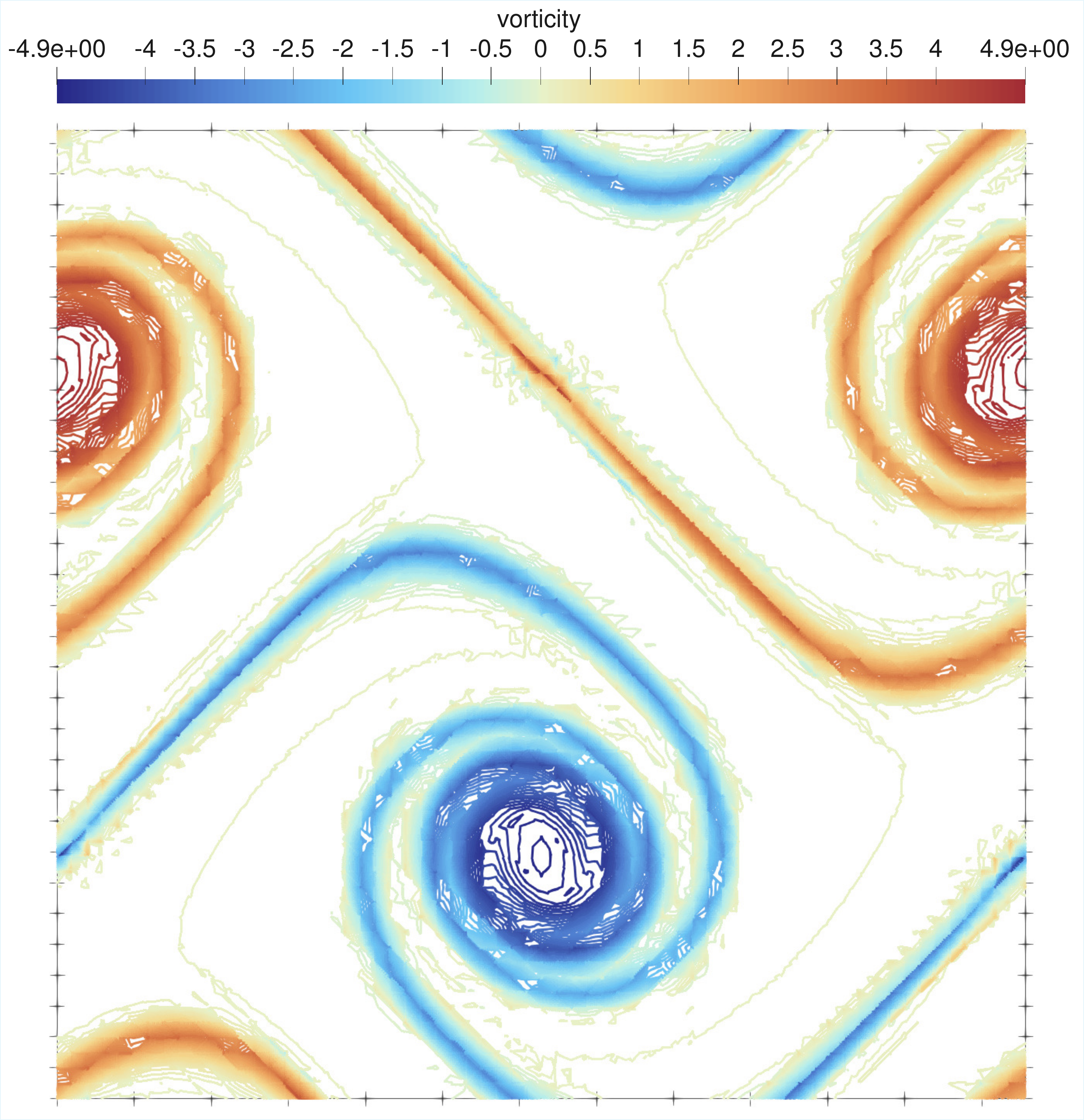}%
\includegraphics[width=0.24\textwidth]{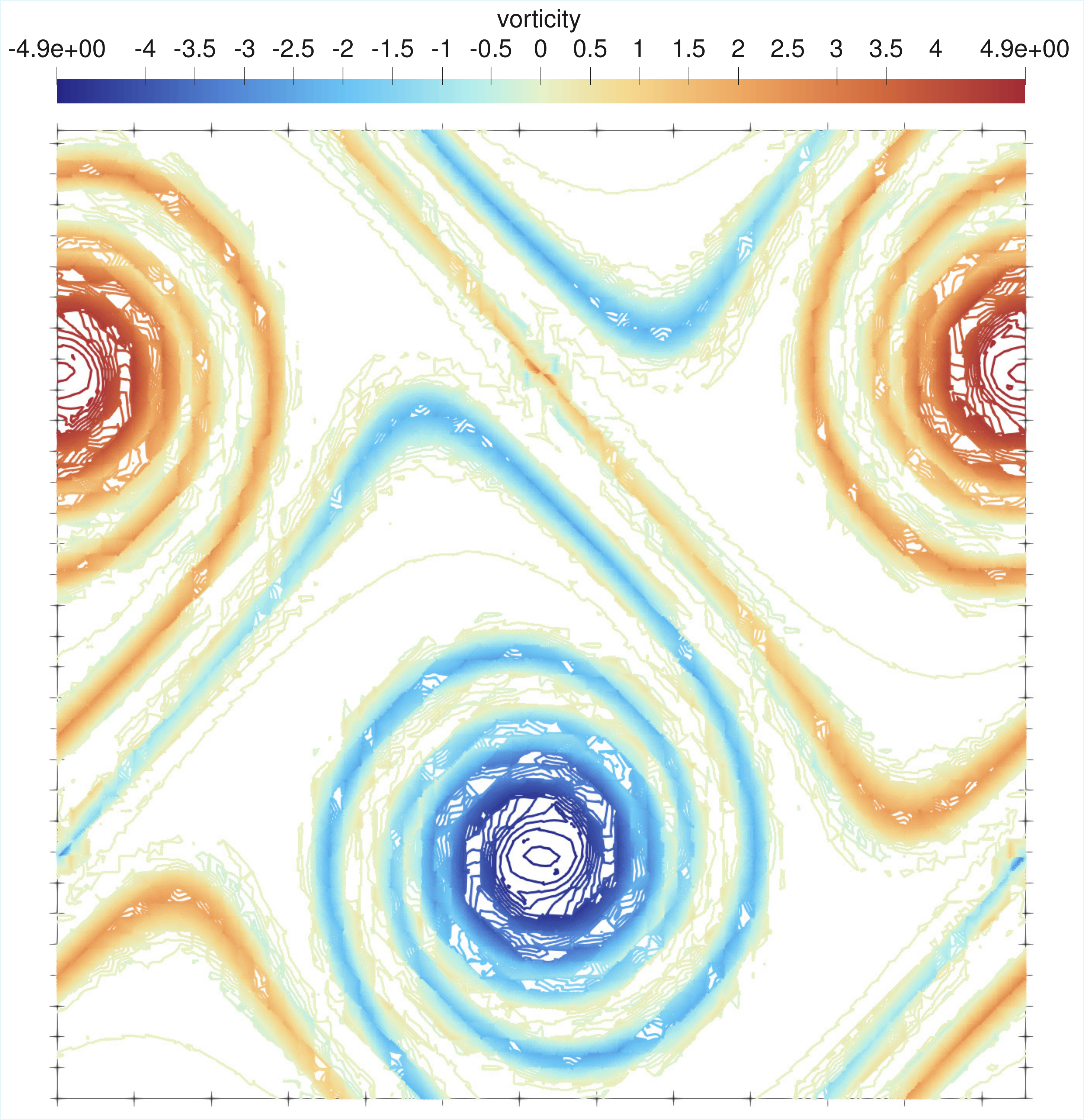}%
\includegraphics[width=0.24\textwidth]{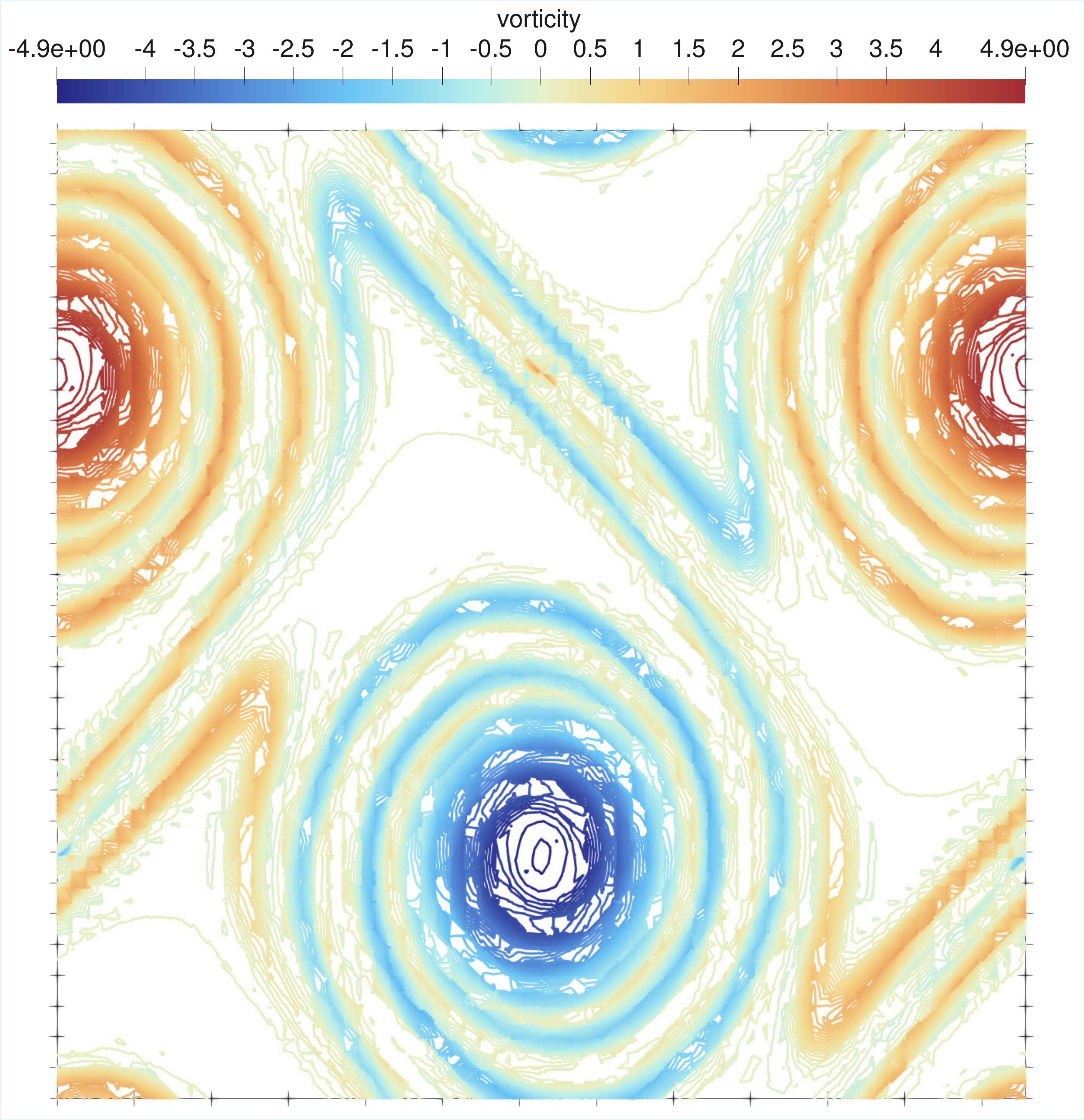}

\includegraphics[width=0.24\textwidth]{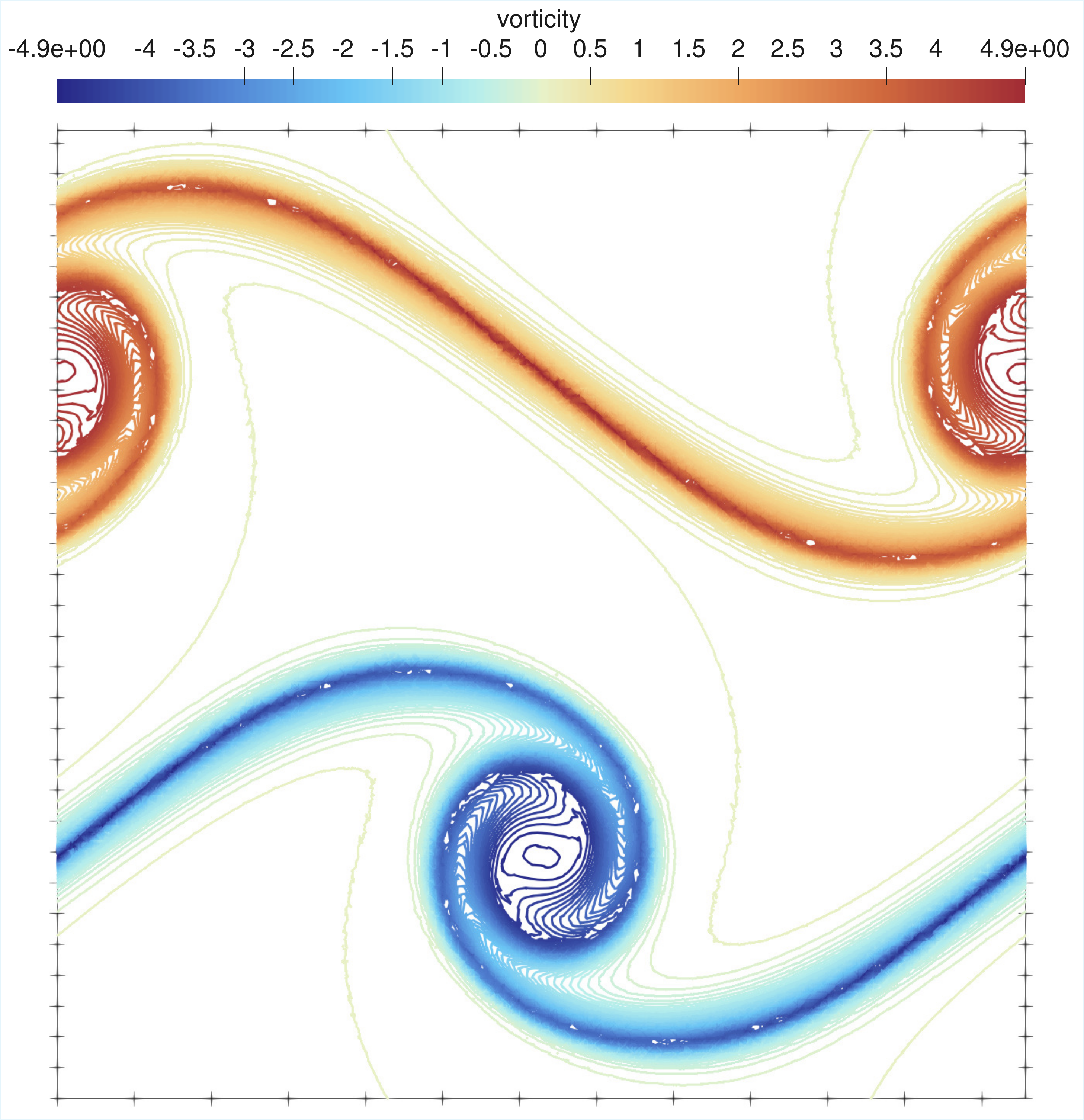}%
\includegraphics[width=0.24\textwidth]{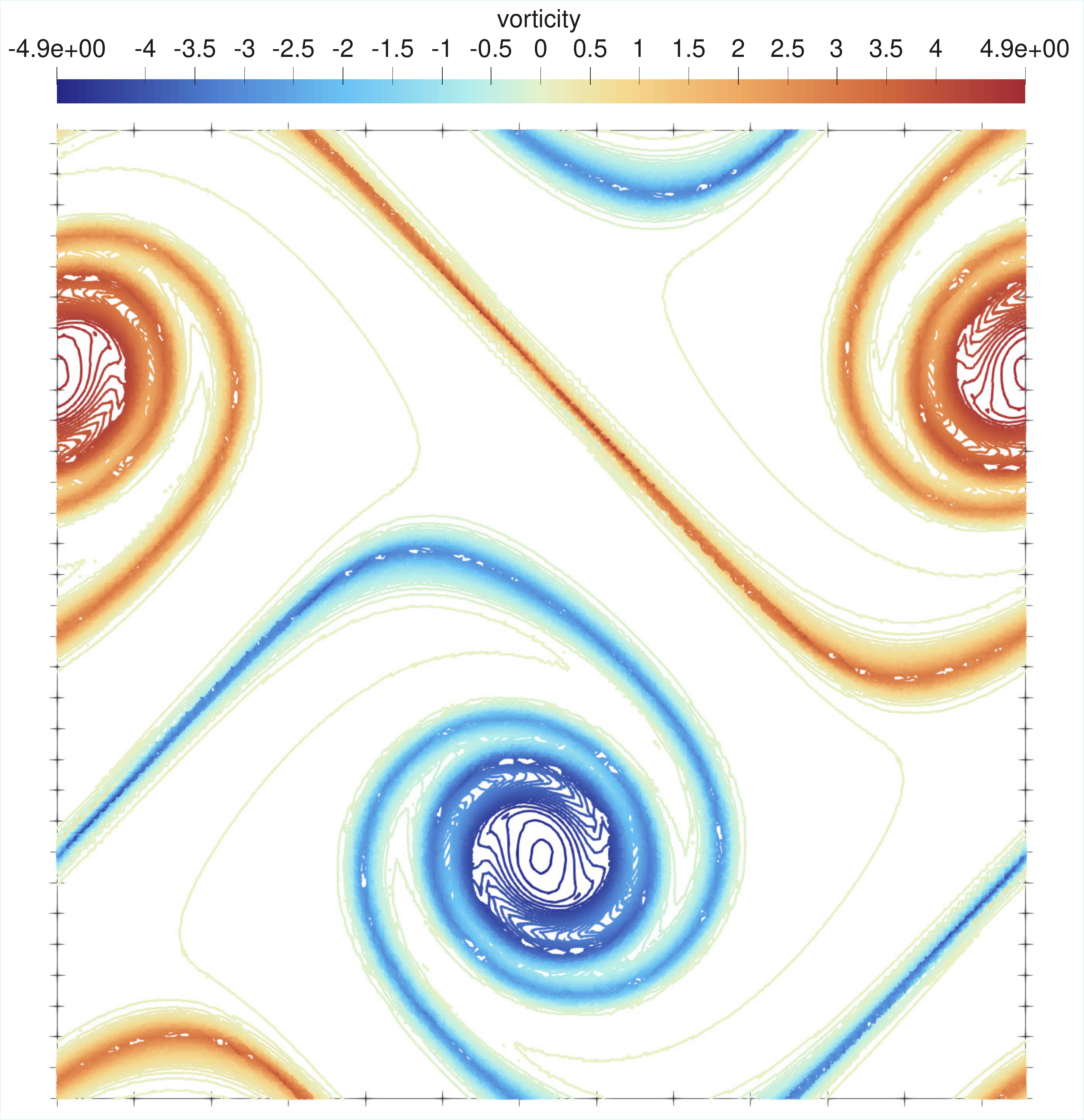}%
\includegraphics[width=0.24\textwidth]{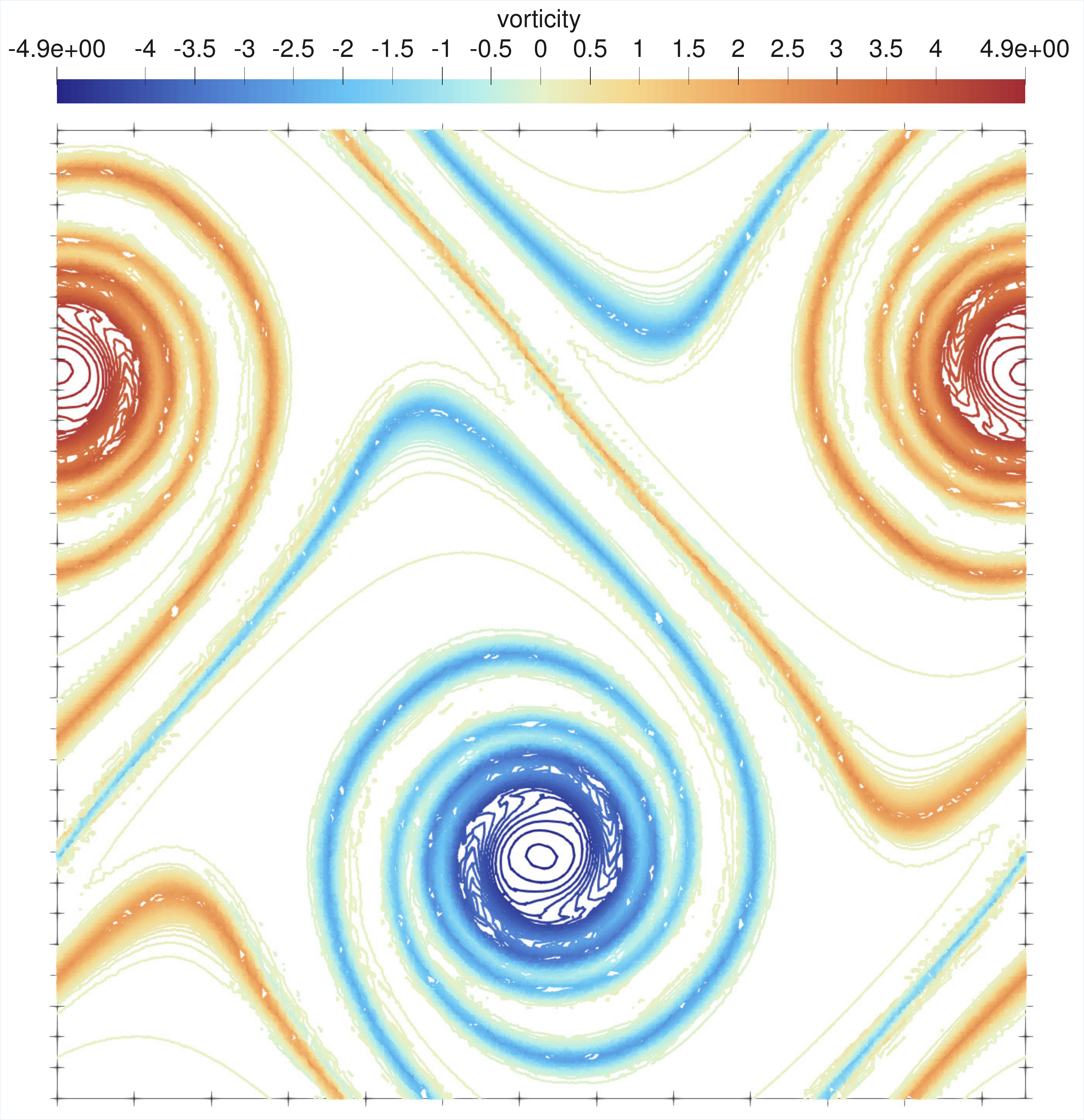}%
\includegraphics[width=0.24\textwidth]{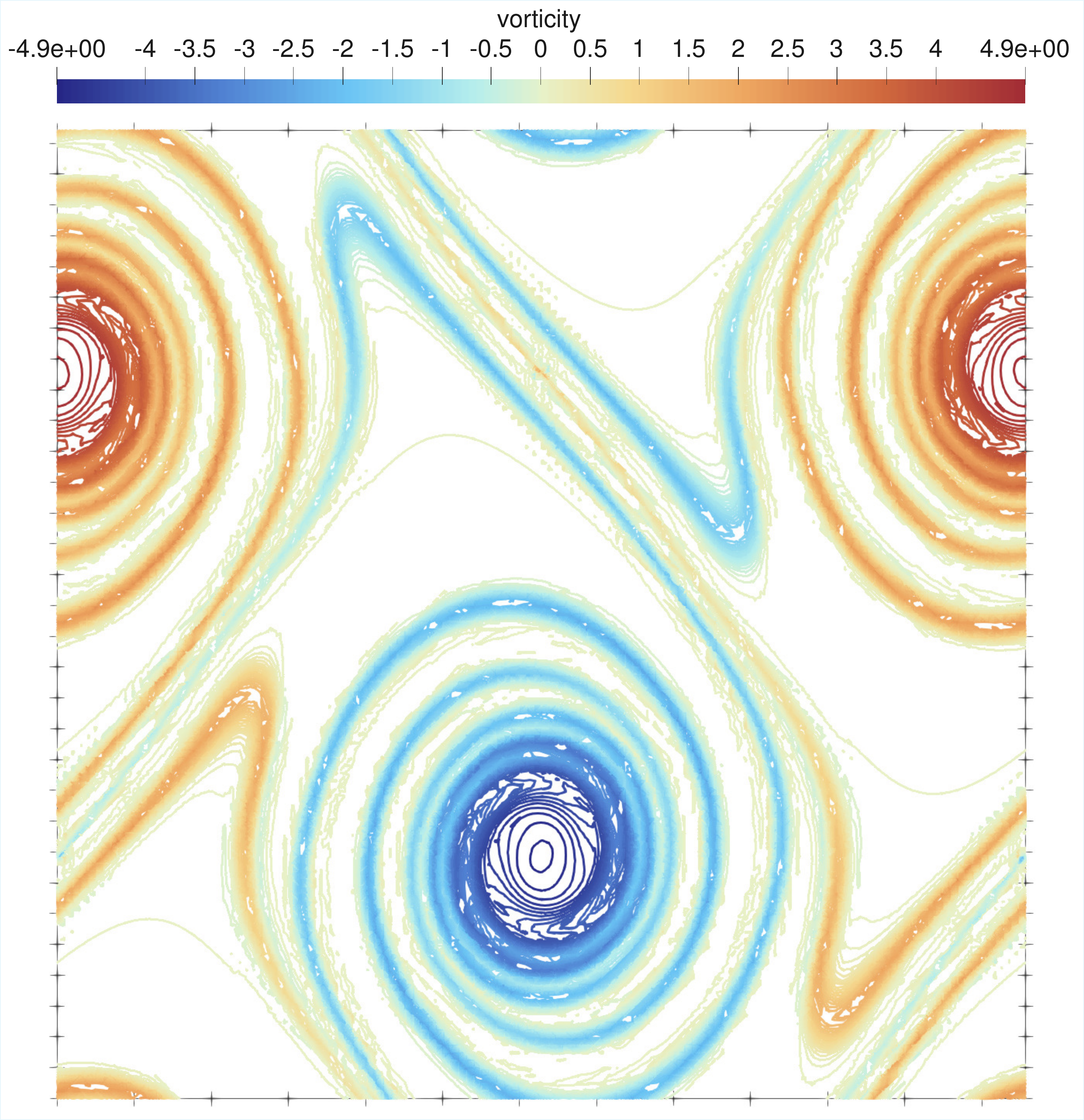}

\caption{\footnotesize Vorticity contours for the double shear layer problem ($\nu=0$ and $C_s=0$). First row: $N=64$; second row: $N=128$. Columns correspond to $T=6,8,10,12$.}
\label{fig:test3:1}
\end{figure}

\begin{figure}[htbp]
\centering
\includegraphics[width=0.24\textwidth]{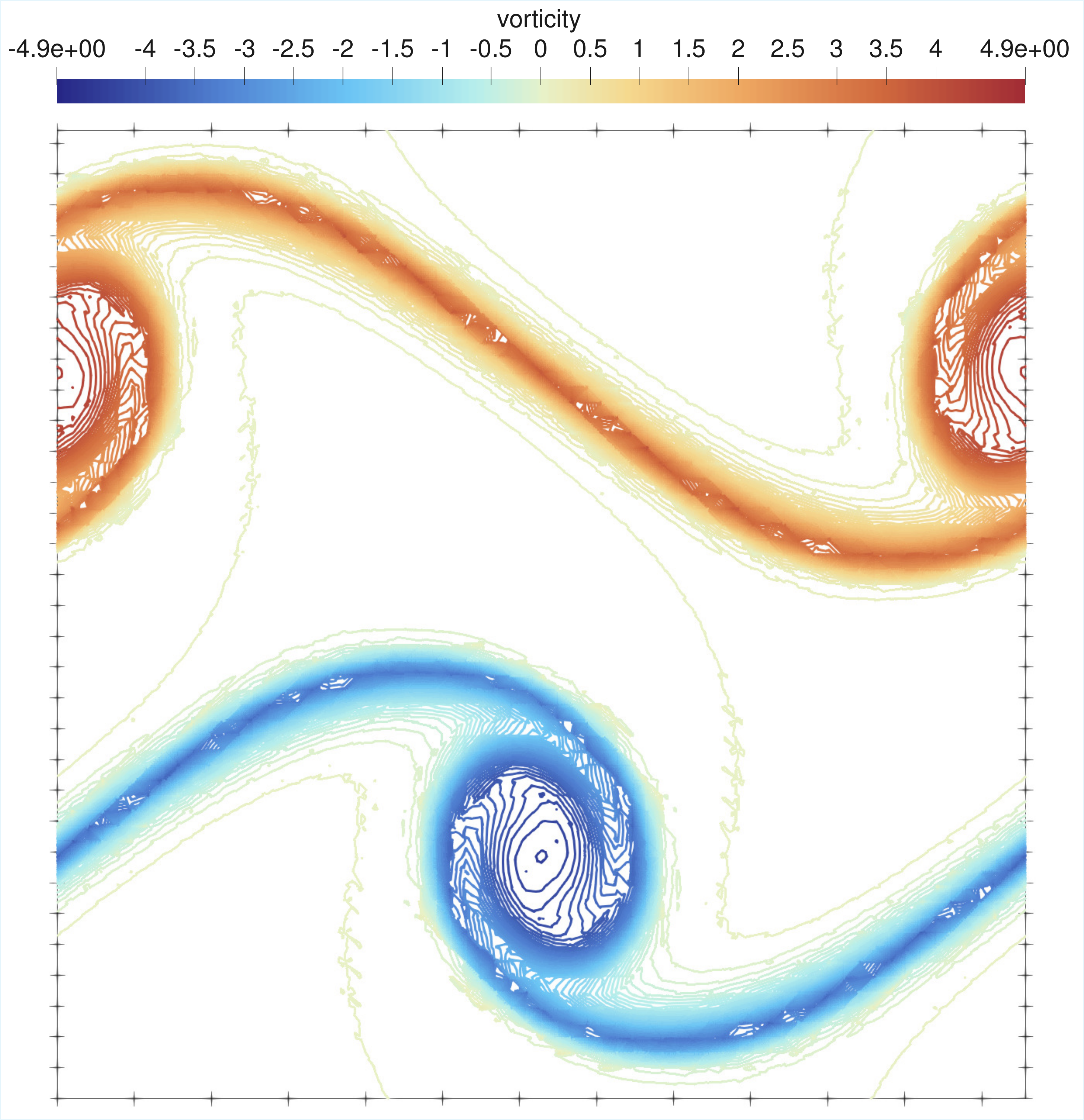}%
\includegraphics[width=0.24\textwidth]{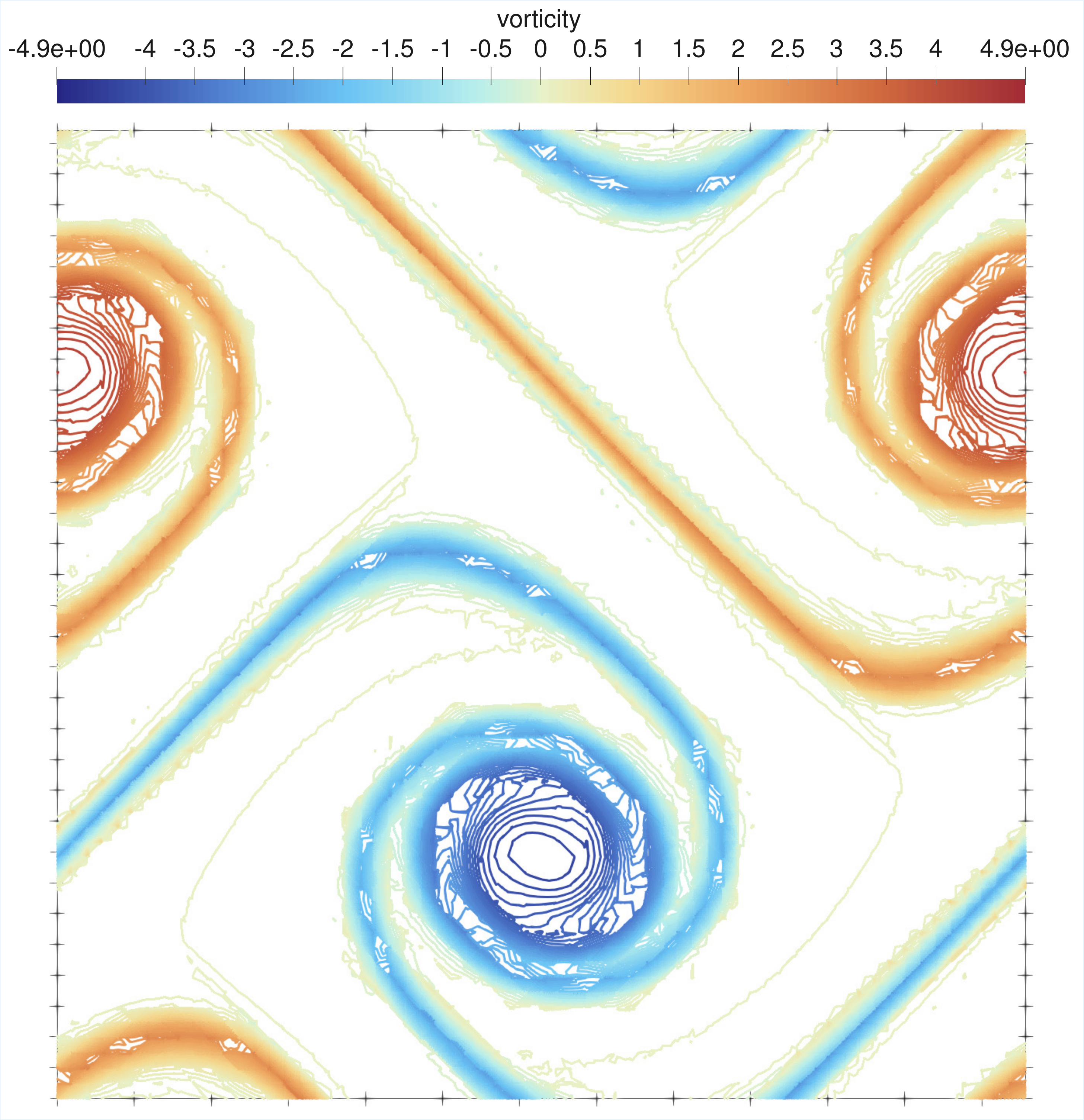}%
\includegraphics[width=0.24\textwidth]{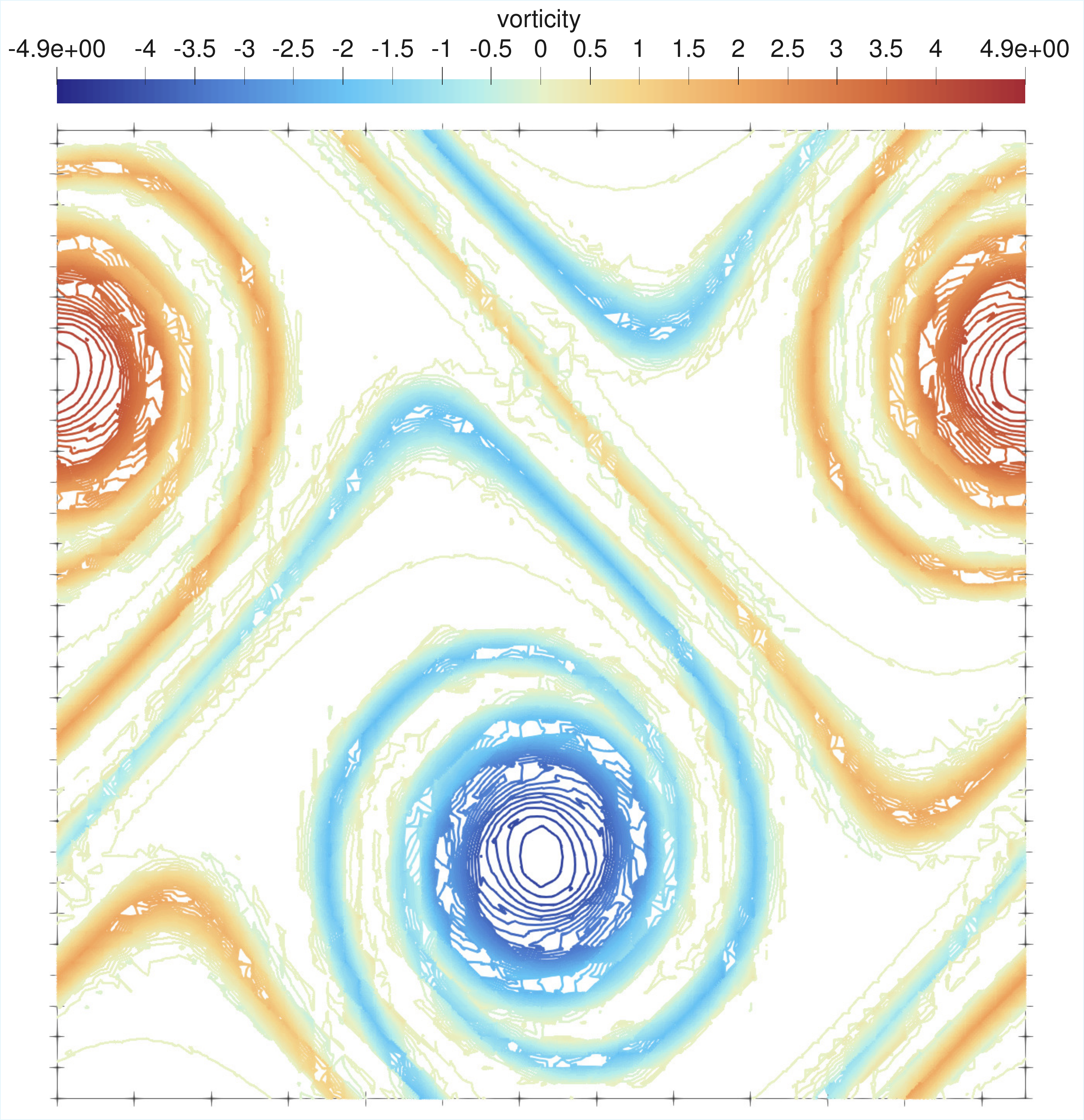}%
\includegraphics[width=0.24\textwidth]{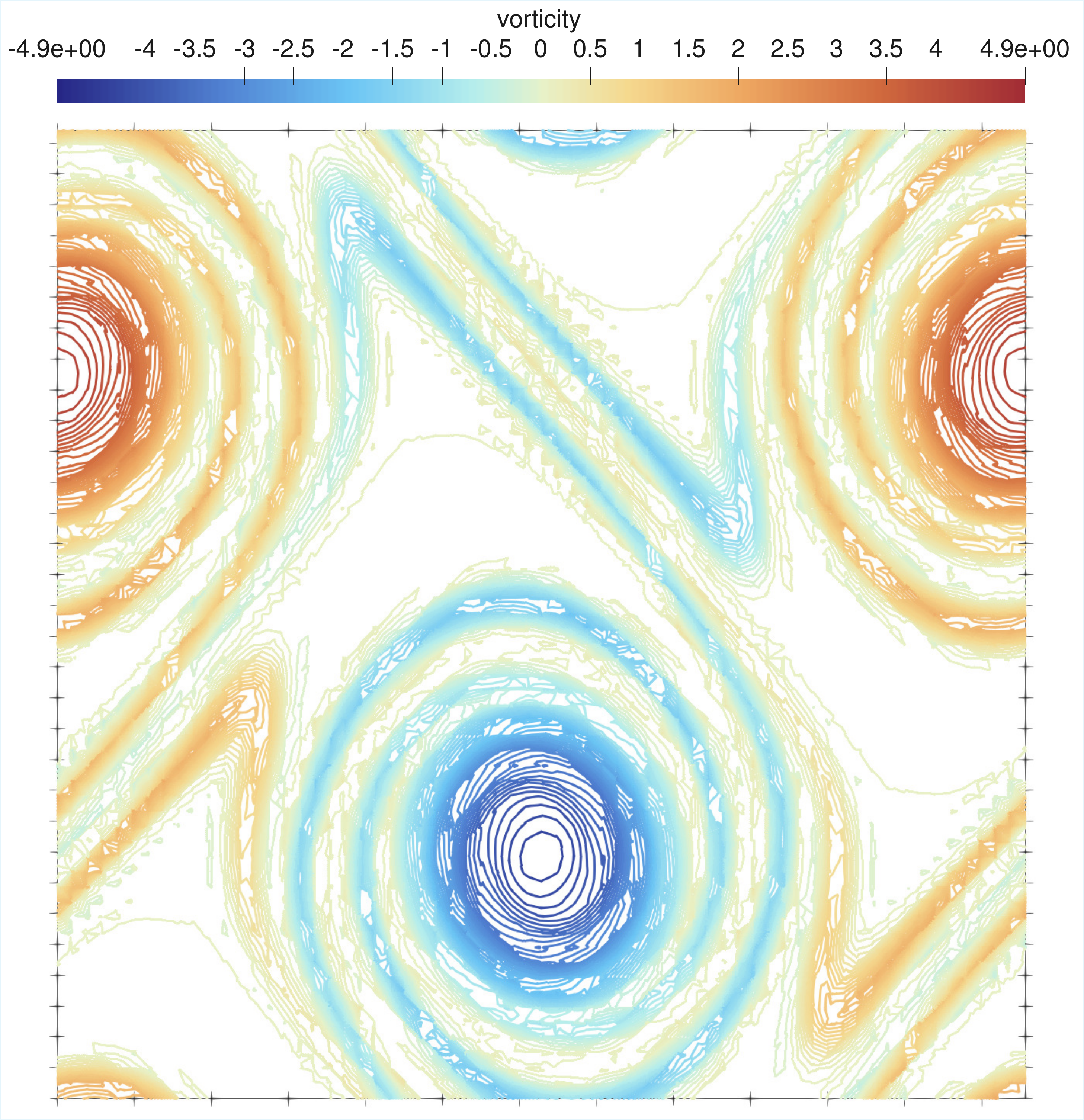}

\includegraphics[width=0.24\textwidth]{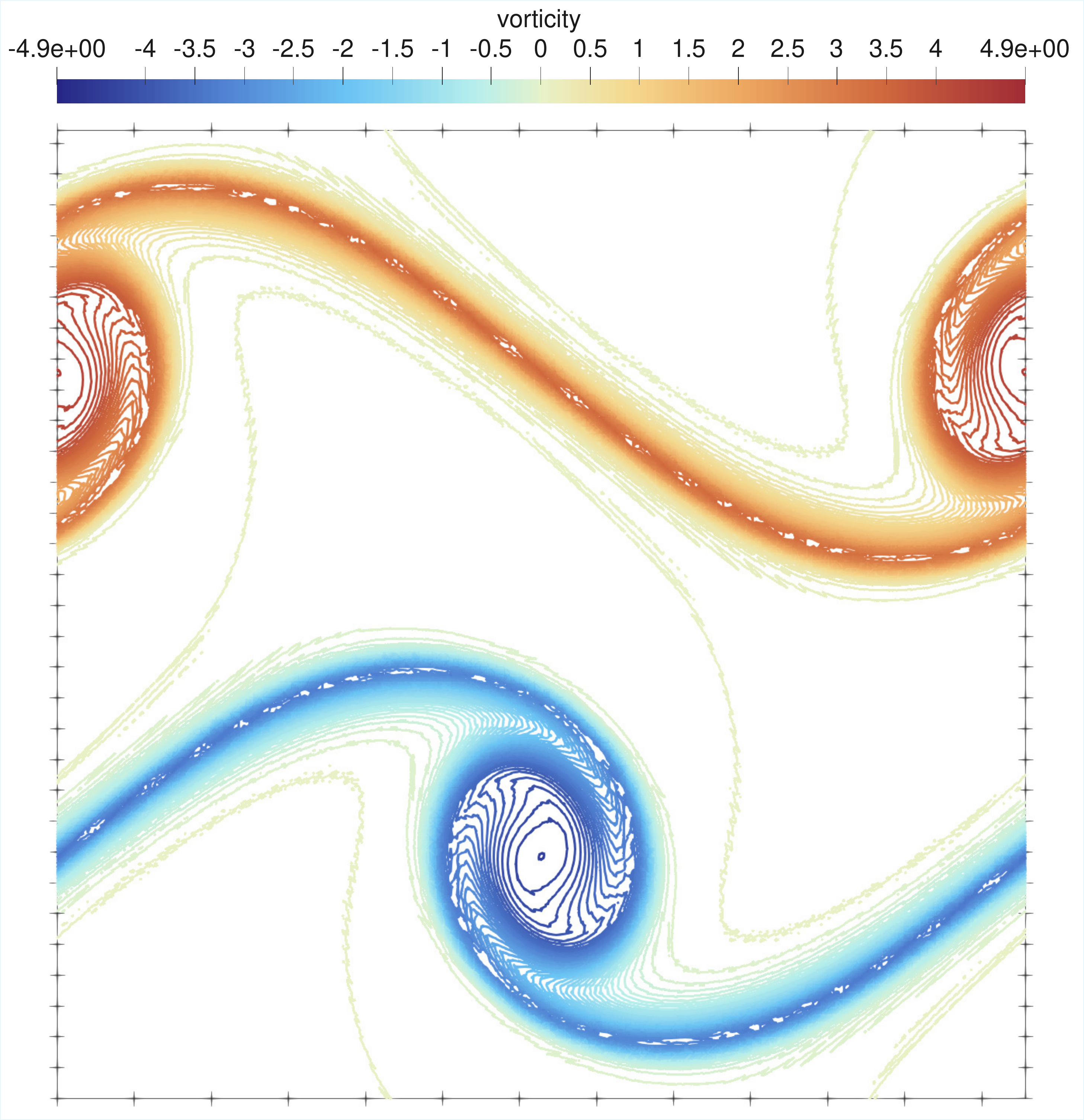}%
\includegraphics[width=0.24\textwidth]{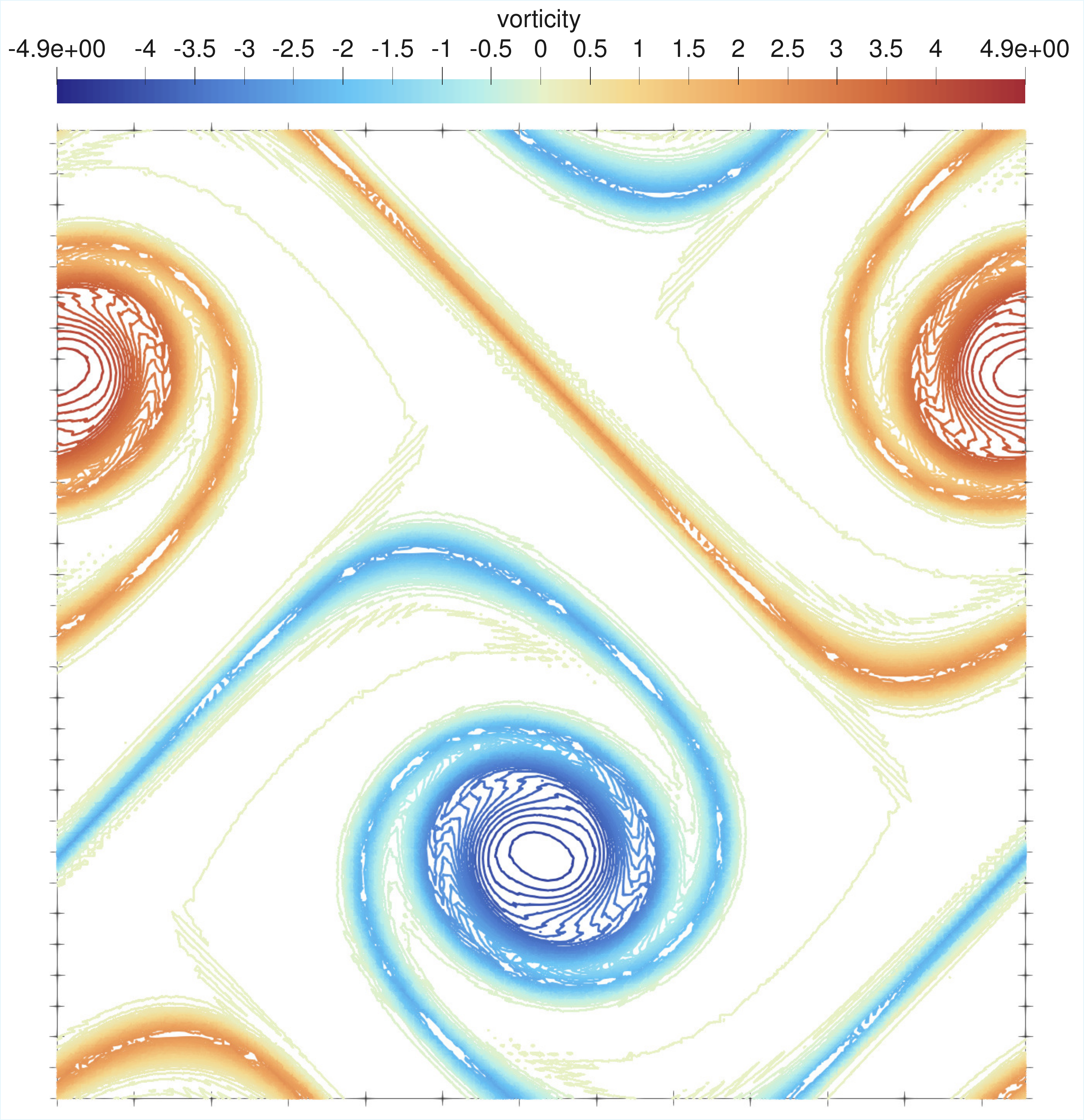}%
\includegraphics[width=0.24\textwidth]{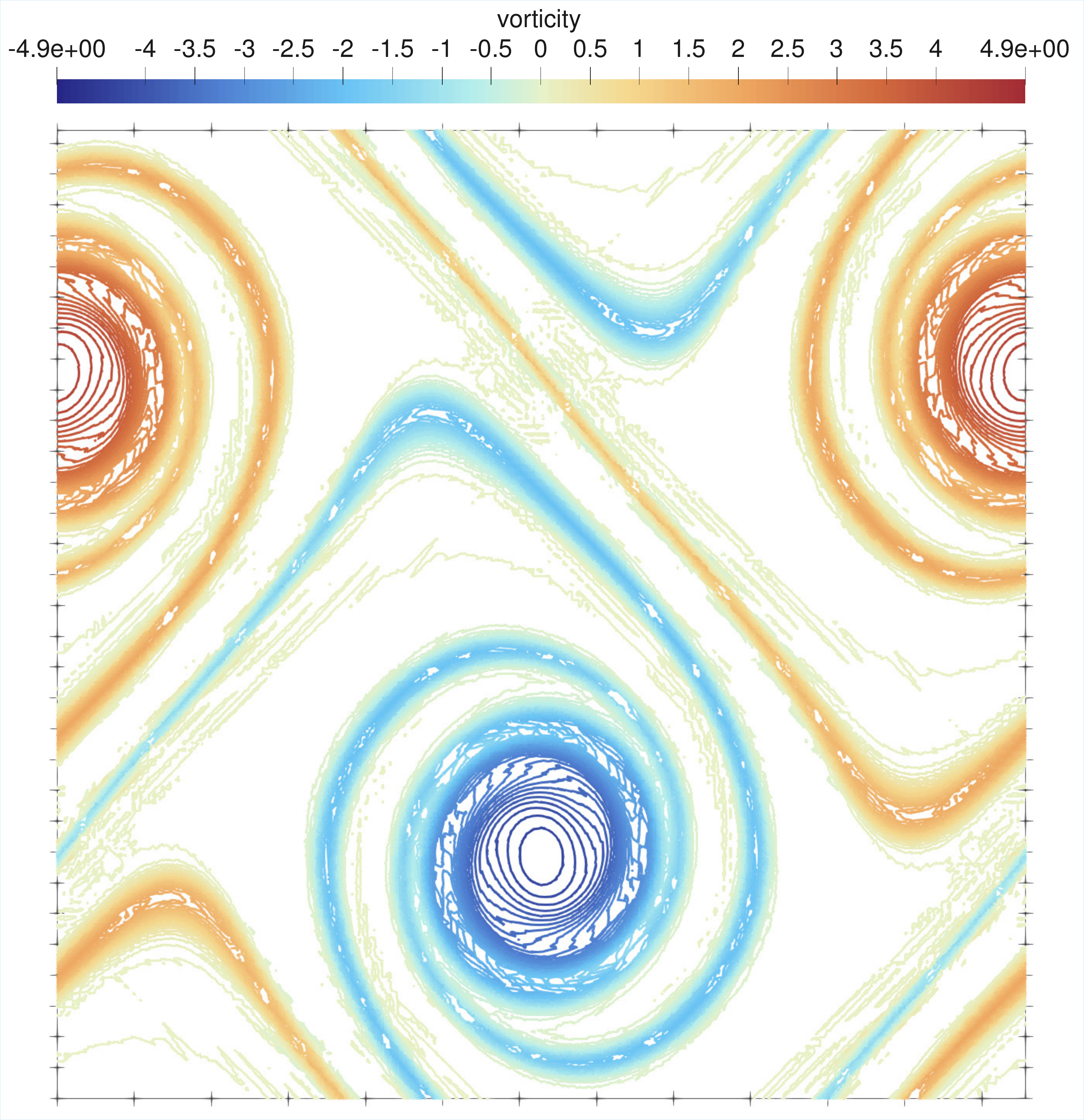}%
\includegraphics[width=0.24\textwidth]{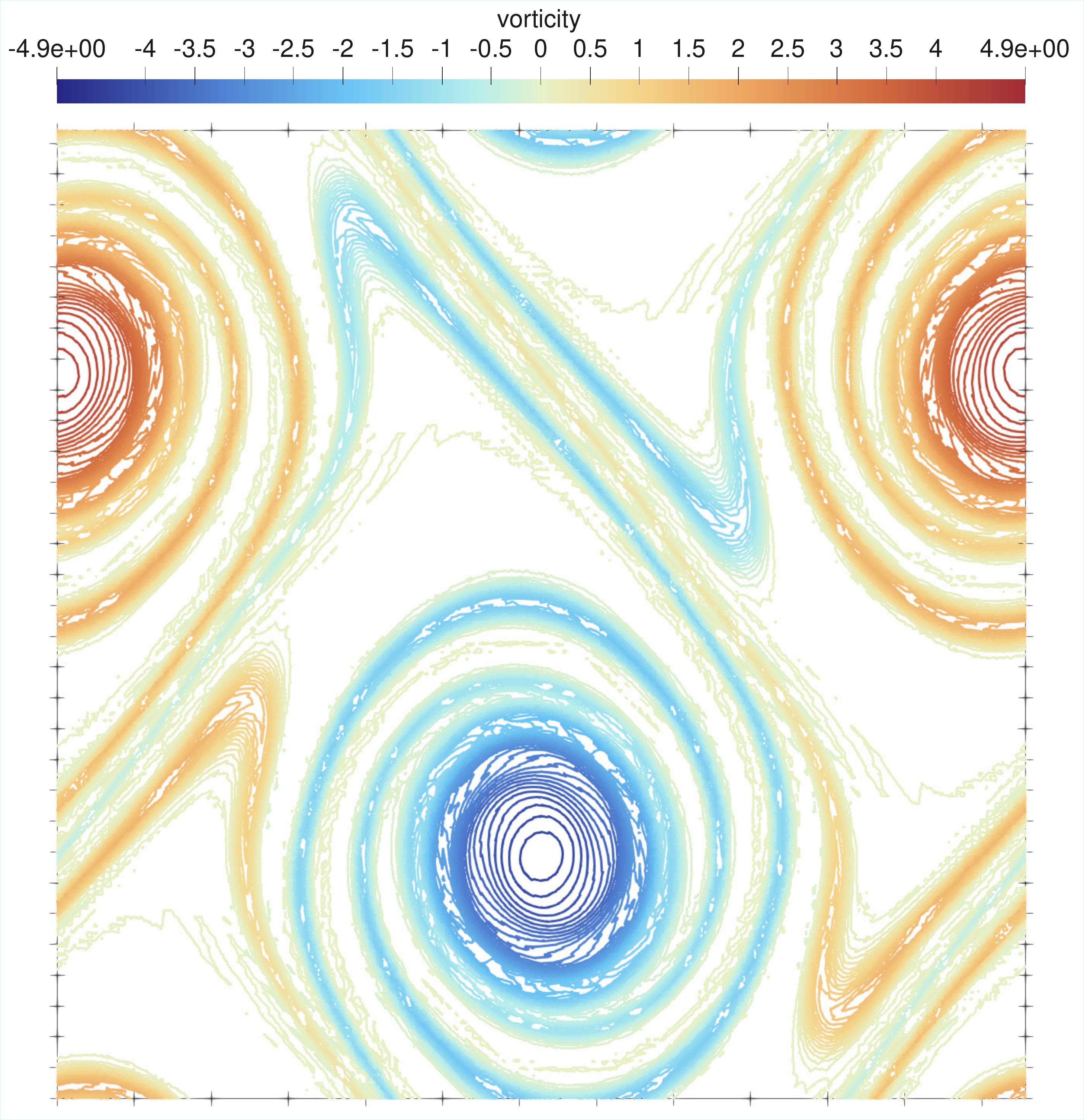}

\caption{\footnotesize Vorticity contours for the double shear layer problem with turbulent  viscosity ($\nu=0$ and $C_s=0.1$). First row: $N=64$; second row: $N=128$. Columns correspond to $T=6,8,10,12$.}
\label{fig:test3:2}
\end{figure}

\section{Conclusion}\label{Sec:conclusion}
We have developed and analyzed a fully discrete, globally divergence-free HDG method for the gradient-based Smagorinsky model. The method combines backward Euler time integration with interior-penalty discretizations of the molecular and nonlinear eddy-viscosity terms and an upwind convective flux. The pressure coupling yields an $H(\divop)$-conforming, pointwise divergence-free velocity, providing exact discrete mass conservation and pressure robustness with respect to irrotational force perturbations. Energy stability and existence hold without a time-step restriction for $\nu>0$ under sufficiently large penalty parameters, while uniqueness follows under separate sufficient smallness conditions.

The error analysis yields pressure-independent velocity bounds without explicit negative powers of the molecular viscosity $\nu$. Under $\delta=O(h)$, uniform regularity, and a time-step condition with a fixed positive margin, the resulting estimates provide mesh-uniform pre-asymptotic error bounds relevant to high-Reynolds-number regimes. Numerical experiments support the predicted velocity-error behavior and demonstrate machine-precision divergence residuals, the expected energy-dissipation properties, and the stabilizing effect of the Smagorinsky term in convection-dominated flows.

\end{document}